\documentclass[a4paper,11pt]{amsart}

\usepackage{graphicx}
\usepackage{amsmath}
\usepackage{amssymb}
\usepackage[active]{srcltx}
\usepackage{hyperref}

\newtheorem{thm}{Theorem}[section]
\newtheorem{lem}[thm]{Lemma}
\newtheorem{prop}[thm]{Proposition}
\newtheorem{cor}[thm]{Corollary}
\theoremstyle{definition}

\theoremstyle{remark}
\newtheorem{remark}{Remark}[section] %
\theoremstyle{plain}

\def\CC{{\mathbb C}}

\def\HH{{\mathbb H}}
\def\NN{{\mathbb N}}

\def\RR{{\mathbb R}}

\def\ZZ{{\mathbb Z}}

\def\new{\operatorname{new}}

\def\veca{{\text{\boldmath$a$}}}
\def\vecb{{\text{\boldmath$b$}}}

\def\vecc{{\text{\boldmath$c$}}}

\def\vece{{\text{\boldmath$e$}}}
\def\vecf{{\text{\boldmath$f$}}}

\def\veck{{\text{\boldmath$k$}}}

\def\vecell{{\text{\boldmath$\ell$}}}
\def\vecm{{\text{\boldmath$m$}}}
\def\vecn{{\text{\boldmath$n$}}}
\def\vecq{{\text{\boldmath$q$}}}

\def\vecr{{\text{\boldmath$r$}}}

\def\vecu{{\text{\boldmath$u$}}}
\def\vecv{{\text{\boldmath$v$}}}

\def\vecw{{\text{\boldmath$w$}}}
\def\vecx{{\text{\boldmath$x$}}}

\def\vecy{{\text{\boldmath$y$}}}
\def\vecz{{\text{\boldmath$z$}}}
\def\vecalf{{\text{\boldmath$\alpha$}}}
\def\vecbeta{{\text{\boldmath$\beta$}}}

\def\veceta{{\text{\boldmath$\eta$}}}
\def\vecpsi{{\text{\boldmath$\psi$}}}

\def\vecxi{{\text{\boldmath$\xi$}}}

\def\scrL{{\mathcal L}}

\def\scrU{{\mathcal U}}

\def\scrY{{\mathcal Y}}

\def\fR{{\mathfrak R}}
\def\fS{{\mathfrak S}}

\def\tre{\operatorname{Re}}
\def\tim{\operatorname{Im}}

\def\C{\operatorname{C{}}}
\def\L{\operatorname{L{}}}

\def\GL{\operatorname{GL}}

\def\SL{\operatorname{SL}}
\def\ASL{\operatorname{ASL}}

\def\tr{\operatorname{tr}}

\def\supp{\operatorname{supp}}

\def\ord{\operatorname{ord}}

\def\trans{\,^\mathrm{t}\!}

\def\Onder#1#2#3#4#5{#1 \setbox0=\hbox{$#1$}\setbox1=\hbox{$#2$}
       \dimen0=.5\wd0 \dimen1=\dimen0 \dimen2=\dp0 \dimen3=\dimen2
       \advance\dimen0 by .5\wd1 \advance\dimen0 by -#4
       \advance\dimen1 by -.5\wd1 \advance\dimen1 by -#4
       \advance\dimen2 by -#3 \advance\dimen2 by \ht1
       \advance\dimen2 by 0.3ex \advance\dimen3 by #5
        \kern-\dimen0\raisebox{-\dimen2}[0ex][\dimen3]{\box1}
       \kern\dimen1}

\renewcommand{\mod}{\:\text{mod}\:}

\newcommand{\G}{\mathrm{G}}
\newcommand{\intr}{\int_{-\infty}^{\infty}}

\newcommand{\tvecpsi}{\widetilde{\vecpsi}}

\newcommand{\p}{\mathsf{p}}

\newcommand{\hh}{\widehat{h}}

\newcommand{\tM}{\tilde{M}}

\newcommand{\tth}{\tilde{h}}
\newcommand{\hF}{\widehat{F}}

\newcommand{\oGamma}{\overline{\Gamma}}

\newcommand{\bs}{\backslash}

\newcommand{\lsl}{\mathfrak{sl}}

\newcommand{\ig}{\mathfrak{g}}

\newcommand{\GaG}{\Gamma\backslash \mathrm{G}}
\newcommand{\wh}{\widehat}

\newcommand{\hf}{\widehat f}

\newcommand{\Q}{\mathbb{Q}}
\newcommand{\R}{\mathbb{R}}
\newcommand{\Z}{\mathbb{Z}}
\newcommand{\N}{\mathbb{N}}

\newcommand{\col}{\: : \:}

\newcommand{\bn}{\mathbf{0}}

\newcommand{\ve}{\varepsilon}

\newcommand{\cmatr}[2]{\left( \begin{matrix} #1 \\ #2 \end{matrix} \right) }

\newcommand{\matr}[4]{\left( \begin{matrix} #1 & #2 \\ #3 & #4 \end{matrix} \right) }
\newcommand{\smatr}[4]{\bigl( \begin{smallmatrix} #1 & #2 \\ #3 & #4 \end{smallmatrix} \bigr) }

\newcommand{\starsum}{\sideset{}{^*}\sum}

\newcommand{\y}{{\text{\boldmath$y$}}}
\renewcommand{\v}{{\text{\boldmath$v$}}}

\title[Effective equidistribution in homogeneous spaces of $\SL(2,\R)\ltimes(\R^2)^{k}$]{Effective equidistribution of unipotent orbits in homogeneous spaces of $\SL(2,\R)\ltimes(\R^2)^{k}$}

\author{Andreas Str\"ombergsson}
\address{Department of Mathematics, Box 480, Uppsala University,
SE-75106 Uppsala, Sweden}
\email{astrombe@math.uu.se}

\author{Anders S\"odergren}
\address{Department of Mathematical Sciences, Chalmers University of Technology and the \newline 
\rule[0ex]{0ex}{0ex}\hspace{8pt} University of
Gothenburg, SE-412 96 Gothenburg, Sweden}
\email{andesod@chalmers.se}

\author{Pankaj Vishe}
\address{Department of Mathematical Sciences, Durham University, Durham DH1 3LE, U.K.}
\email{pankaj.vishe@durham.ac.uk}

\thanks{Str\"ombergsson was supported by the Knut and Alice Wallenberg Foundation and also by the Swedish Research Council Grant 2023-03411. S\"odergren was supported by the grant 2021-04605 from the Swedish Research Council.}

\begin{document}

\begin{abstract}
Let $\G=\SL(2,\R)\ltimes(\R^2)^{k}$, let $\Gamma$ be a congruence subgroup of $\SL(2,\Z)\ltimes(\Z^2)^{k}$,
and let $\mathrm{u}_{\R}=(\mathrm{u}_x)_{x\in\R}$ be the one-parameter subgroup of $\G$ given by
$\mathrm{u}_x=\left(\matr 1x01,\bn\right)$.
We prove polynomially effective asymptotic equidistribution results for 
expanding translates of $\mathrm{u}_{\R}$--orbits and for
long pieces of individual $\mathrm{u}_{\R}$--orbits in $\GaG$.
An important ingredient of the proof is %
the delta symbol version of the circle method. %
\end{abstract}

\maketitle

\tableofcontents

\newpage

\section{Introduction}
\label{intro}

In the theory of unipotent dynamics on homogeneous
spaces, a fundamental role is played by Ratner’s theorems \cite{Ratner91,mR91} 
on measure rigidity, topological rigidity, and orbit equidistribution.
These results also have 
a large number of applications in other areas,   %
ranging from number theory to mathematical physics,
see e.g.\
\cite{EMS96,EMM98,Elkies04,EMM05,EO06,Sha10,MS10,MS14}.
We refer to the book by Morris
\cite{Morris} for a detailed exposition and historical background.

In the last few decades,  %
because of its intrinsic interest and in view of the applications,
there has been 
an increased focus on %
establishing
\textit{effective} versions of Ratner's results in special cases, 
that is, to provide an explicit rate of density or equidistribution
for the orbits of a (non-horospherical) 
unipotent flow. 
Recently,
significant advances on this problem
have been obtained by Lindenstrauss and Mohammadi \cite{LM23}
and later with Wang and Yang and by Yang
\cite{LMW22,Yang25,LMWY25}.
In particular, in \cite{LMWY25}, effective %
equidistribution theorems with
polynomial rate for orbits of unipotent subgroups were established in quotients of
quasi-split, almost simple linear algebraic groups of absolute rank 2.
See also the very %
recent papers %
\cite{Lin2025a,Lin2025b} and \cite{LMY26},
where such effective results are obtained 
in several further cases.

Our purpose in the present paper is to establish results on effective equidistribution 
for unipotent orbits in quotients %
of the %
Lie group
$\G=\SL(2,\R)\ltimes(\R^2)^{k}$.
Here $(\R^2)^{k}$ denotes the direct sum of $k$ copies of $\R^2$, each provided with the standard action of
$\SL(2,\R)$.
Equidistribution of unipotent orbits in this setting
has been applied in the solution of several problems 
in mathematical physics; see works of Marklof
\cite{MarklofpaircorrI,MarklofpaircorrII,DMS}
and also Marklof and Str\"ombergsson \cite[\S\S 5.2 and 5.3.3]{MS24}, 
and Palmer and Str\"ombergsson
\cite{powerlaw}.
Effective equidistribution in the case of $\G=\SL(2,\R)\ltimes\R^2$
(i.e.\ $k=1$) was obtained by the first author in \cite{SASL} (see also a work by Browning and Vinogradov \cite{BV});
and for general $k$,
effective results concerning certain special families of unipotent orbits
were proved in a work by Str\"ombergsson and Vishe \cite{effopp1}.
Related effective %
results for the group
$\SL(d,\R)\ltimes\R^d$ with $d\geq2$ arbitrary
have been obtained by 
Kim \cite{Wooyeon}.
It should also be noted that the main result in the very recent paper 
by Lin \cite{Lin2025b}
as special cases yields %
effective equidistribution with a polynomial rate for
unipotent flows (subject to certain conditions)
in quotients of groups of the form
$H\ltimes V$, where $H$ is a semisimple real
algebraic group and $V$ is a non-trivial irreducible representation of $H$;
see \cite[Ex.\ 1.9 and 1.10]{Lin2025b}.

As far as we are aware, effective equidistribution results for unipotent orbits in the setting
of %
$\G=\SL(2,\R)\ltimes(\R^2)^k$
with $k\geq2$
have not been previously obtained,
apart from the results concerning special orbits in \cite{effopp1}.
One may also note that the explicit polynomial rates of equidistribution
which we prove in the present paper
are from a certain point of view %
significantly stronger than those obtained in
\cite{Lin2025a,Lin2025b,LMW22,LMWY25,LMY26,Wooyeon}:
As we will see below, for initial conditions subject to a Diophantine condition
(which is satisfied in the generic case),
our error bounds decay as rapidly as 
the power $-\frac14+\ve$ of the length of the unipotent orbit considered.
The corresponding powers in the main results of 
\cite{Lin2025a,Lin2025b,LMW22,LMWY25,LMY26,Wooyeon}
appear to be quite tiny. %

\vspace{5pt}

Before we state the main results, we will set some notation.
We view $(\R^2)^{k}$ as
the set of real $k\times 2$ matrices.
The (right) action of $\mathrm{G}'=\SL(2,\R)$ on $(\R^2)^{k}$
is given by matrix multiplication.
The elements of $\mathrm{G}=\mathrm{G}'\ltimes(\R^2)^{k}$
are then represented by pairs $(M,\vecv)$ 
in $\mathrm{G}'\times(\R^2)^{k}$,
and the multiplication law is:
\begin{align}\label{multlaw}
(M,\vecv)(M',\vecv')=(MM',\vecv M'+\vecv'). %
\end{align}
We will always view $\mathrm{G}'$ as a subgroup of $\mathrm{G}$ via the embedding $M\mapsto(M,\bn)$.
We set
\begin{align*}
\overline\Gamma=\SL(2,\Z)\ltimes(\Z^2)^{k},
\end{align*}
that is, the subgroup of all $(M,\vecv)\in \mathrm{G}$ with $M\in\SL(2,\Z)$ and 
the matrix $\vecv$ having integer entries.
Given a subgroup $\Gamma$ of $\overline\Gamma$ of finite index,
we consider the homogeneous space 
\begin{align*}
X=\GaG,
\end{align*}
and we fix $\mu$ to be the (left and right invariant) Haar measure on $\mathrm{G}$, normalized so as to induce a probability measure
on $X$,
which we also denote by $\mu$.

Throughout we denote the $2\times 2$ identity matrix by $1_2$. Let 
\begin{align*}
\mathrm{a}_y=\matr{\sqrt y}00{1/\sqrt y}\qquad
\text{and} \qquad
\mathrm{u}_x=\matr 1x01
\qquad (y>0,\: x\in\R).
\end{align*}
Then $\{\mathrm{a}_y\col y>0\}$
and $\{\mathrm{u}_x\col x\in\R\}$
are one-parameter subgroups of $\G'$, and thus also of $\G$.

For any $d$, we identify vectors in $\R^d$ with row matrices.
In particular, for any $\vecq\in\R^k$ and $\vecxi\in(\R^2)^k$ 
we have $\vecq\vecxi\in\R^2$.

The following equidistribution result
concerning expanding translates of $\mathrm{u}_{\R}$-orbits
is a special case of
\cite[Thm.\ 1.4]{Shah}
and also of \cite[Thm.\ 3]{DMS},
both of which depend crucially on Ratner's classification of invariant measures.

\begin{thm}\label{INEFFECTIVETHM}
Let $\Gamma$ be a finite index subgroup of $\overline\Gamma=\SL(2,\Z)\ltimes(\Z^2)^{k}$.
Fix a vector $\vecxi\in(\R^2)^{k}$
such that 
$\vecq \vecxi\notin\Z^2$ for all $\vecq\in\Z^k\setminus\{\bn\}$.
Then for any $M\in\SL(2,\R)$,
any Borel probability measure $\lambda$ on $\R$ which is absolutely continuous with respect to the Lebesgue
measure,
and any bounded continuous function $f$ on $X=\GaG$,
\begin{align}\label{INEFFECTIVETHMRES}
\lim_{y\to0^+}\int_\R f\bigl(\Gamma (1_2,\vecxi)M \mathrm{u}_x\mathrm{a}_y\bigr)\,d\lambda(x)
=%
\int_X f\,d\mu.
\end{align}
\end{thm}

Let us note that the condition imposed on $\vecxi$ in Theorem \ref{INEFFECTIVETHM} cannot be weakened. %
Indeed, if $\vecq \vecxi\in\Z^2$ for some $\vecq\in\Z^k\setminus\{\bn\}$,
then all the points $\Gamma (1_2,\vecxi)M \mathrm{u}_x\mathrm{a}_y$
belong to the set
\begin{align*}
X_\vecq:=\bigl\{\Gamma(1_2,\vecv)T\col T\in\SL(2,\R),\:\vecv\in(\R^2)^k,\: \vecq\vecv\in\Z^2\bigr\},
\end{align*}
which is a closed embedded submanifold of codimension 2 in $X$.
Therefore, the corresponding orbits cannot equidistribute in $(X,\mu)$, %
i.e.\ \eqref{INEFFECTIVETHMRES} cannot hold for all $f$.
(For example, consider any bounded continuous $f\geq0$ satisfying $f_{|X_\vecq}\equiv0$
while $\int_X f\,d\mu>0$.)

To prepare for the statement of our first main theorem we introduce some further notation.
For a positive integer $N$, $\Gamma(N)$ denotes the principal congruence subgroup of $\SL(2,\Z)$ of level
$N$: %
\begin{align*}
\Gamma(N)=\biggl\{\matr abcd\in\SL(2,\Z)\col \matr abcd\equiv \matr 1001 \mod N\biggr\}.
\end{align*}
We will consider $X=\GaG$ where $\Gamma$ is a subgroup of $\overline\Gamma$ of the form
$\Gamma=\Gamma(N)\ltimes(\Z^2)^k$.\label{GAMMAdef}
(The case of an arbitrary congruence subgroup of $\overline\Gamma$ can easily be reduced to %
the case of $\Gamma=\Gamma(N)\ltimes(\Z^2)^k$,
by using the fact that for any $q\in\Z^+$, the map $(M,\vecv)\mapsto(M,q\vecv)$ is an automorphism of $\mathrm{G}$.)

For $T\in \mathrm{G}'$ we write $\|T\|$ for its Frobenius matrix norm:
\begin{align}\label{Frobnorm}
\bigl\|T\bigr\|:=\tr\bigl(\,T\trans T\,\bigr)
=\sqrt{a^2+b^2+c^2+d^2},\qquad \text{for }\: T=\matr abcd \in \mathrm{G}'.
\end{align}
We also introduce the following cuspidal height function, for $(M,\vecv)\in \mathrm{G}$:
\begin{align}\label{CHdef}
\scrY(M,\vecv)=\scrY(M):=\sup \{\tim \gamma M(i)\col\gamma\in\SL(2,\Z)\},
\end{align}
where on the right-hand side we use the standard action of $\mathrm{G'}=\SL(2,\R)$ on the
Poincar\'e upper half-plane $\HH=\{\tau=u+iv\in\CC\col v>0\}$.
Then $\scrY(M,\vecv)\geq\sqrt3/2$ for all $(M,\vecv)\in \mathrm{G}$.
Note that $\scrY(M,\vecv)$ depends only on the coset $\oGamma(M,\vecv)$,
and in particular $\scrY$ can be viewed as a function on $X$.
Given $p_1,p_2,\ldots\in X$, we have $\scrY(p_j)\to\infty$ if and only if the sequence $p_1,p_2,\ldots$
leaves all compact subsets of $X$.

For $m\geq0$ %
and $a\in\R$, let $\C_a^m(X)$ be the space of all $m$ times continuously differentiable
functions
on $X$, all of whose derivatives up to order $m$ are $\ll\scrY^{-a}$ throughout $X$.
Let $\ig$ be the Lie algebra of $\mathrm{G}$, and let $X_1,\ldots,X_{2k+3}$ be a fixed basis  of $\ig$.
Each $Y\in\ig$ can be realised as a left invariant differential operator
on the space of functions on $\mathrm{G}$, and therefore also a differential operator on $X=\GaG$,
which we will also denote by $Y$.
For any $f\in\C^m(X)$, set
\begin{align}\label{CMNORMDEF}
\|f\|_{\C^m_a}:=\sum_{\ord(D)\leq m}\,\sup_{p\in X}\,\scrY(p)^a\bigl|(Df)(p)\bigr|, %
\end{align}
where the sum is taken %
over all monomials in $X_1,\ldots,X_{2k+3}$ of degree $\leq m$.
In particular, $\|\cdot\|_{\C^0_0}$ is the supremum norm.
Then $\C_a^m(X)$ is the space of all $f\in\C^m(X)$ with $\|f\|_{\C^m_a}<\infty$.

In an analogous manner,
for any integer $m\geq0$, real number $a\geq0$,
and $h\in\C^m(\R)$, set
\begin{align}\label{Span1dimDEF}
\|h\|_{\C_a^m}=\sum_{j=0}^m\sup_{x\in\R}\,(1+|x|)^a\,\bigl|\partial^j h(x) \bigr|,
\end{align}
and let $\C_a^m(\R)$ be the space of all $h\in\C^m(\R)$ with
$\|h\|_{\C_a^m}<\infty$.

\vspace{5pt}

For $\vecv\in\R^n$ (any $n\in\Z^+$), let $\|\vecv\|$ denote its standard Euclidean norm and
let $\|\vecv\|_{\Z}$ be the distance to the nearest integer vector;
\begin{align*}
\|\vecv\|_{\Z}=\min_{\veca\in\Z^n}\|\vecv-\veca\|.
\end{align*}
Finally, given $m>k$, $\vecxi\in(\R^2)^k$ and $y>0$, we define the majorant function
\begin{align}\label{deltaNEWDEF}
\delta_{m}(y;\vecxi)=\sum_{\vecq\in\Z^k\setminus\{\bn\}}\sum_{d=1}^{\infty}\frac{\tau(d)}{\|\vecq\|^{m}d^{3/2}}
\biggl(1+\frac{\big\| d\vecq\vecxi\big\|_{\Z}}{d\sqrt y}\biggr)^{\hspace{-3pt}-1},
\end{align}
where $\tau(d)$ is the number of divisors of $d$.

We are now ready to state our first main theorem, Theorem \ref{MAINTHM3gen} below, which may be viewed as an effective version of Theorem \ref{INEFFECTIVETHM}
in the special case when $\Gamma$ is a congruence subgroup \mbox{of $\oGamma$.}

\begin{thm}\label{MAINTHM3gen}
Let $k,N\in\Z^+$ and 
set $X=\GaG$ with $\mathrm{G}=\SL(2,\R)\ltimes(\R^2)^k$
and $\Gamma=\Gamma(N)\ltimes(\Z^2)^k$.
Let $m\geq\max(321,k+1)$,
and set $n=3m+3k+5$ and $\alpha=\frac12(m+321)$.
Then for any $f\in\C_\alpha^n(X)$, any $h\in\C_{43}^5(\R)$, %
and any $M\in\SL(2,\R)$, $\vecxi\in(\R^2)^k$ and $0<y\leq 1$, we have
\begin{align}\label{MAINTHM3genres1}
\biggl|\int_\R f\bigl(\Gamma (1_2,\vecxi)M \mathrm{u}_x\mathrm{a}_y\bigr) h(x)\,dx
-\int_X f\,d\mu\int_{\R}h\,dx\biggr|
\ll
\|f\|_{\C_{\alpha}^{n}}\, \|h\|_{\C_{43}^5}\,\|M\|^{13} %
\, \delta_{m}(y;\vecxi),
\end{align}
where the implied constant depends only on $k$, $N$ and $m$.
\end{thm}

Let us make some comments on this result.
First, note that for any fixed $\vecxi\in(\R^2)^k$
and $m>k$, one has
$\delta_{m}(y;\vecxi)\to0$ as $y\to0$
if and only if
$\vecq\vecxi\notin\Z^2$ for all $\vecq\in\Z^k\setminus\{\bn\}$.
Hence Theorem \ref{MAINTHM3gen} gives an effective version of Theorem~\ref{INEFFECTIVETHM}
when $\Gamma$ is a congruence subgroup of $\oGamma$.
Second, when $m$ is sufficiently large,
the majorant function 
$\delta_{m}(y;\vecxi)$ decays as rapidly as
$y^{\frac14-\ve}$
as soon as either the first or the second column of $\vecxi$
satisfies an appropriate Diophantine condition;
see Lemma~\ref{BSUMDIOPHBOUNDLEM1} below. 
In particular, it follows that this decay rate holds whenever
the first or the second column of $\vecxi$ lies outside a 
set of Hausdorff dimension $k-\frac13$
(see Remark \ref{HdimRemark}).
Furthermore, if the entries of one of the columns of $\vecxi$
are algebraic numbers which together with $1$ are linearly independent over $\Q$,
then via a result of Schmidt \cite{wS70} it follows that
$\delta_m(y;\vecxi)$ decays with the power rate $y^{\min(\frac14,\frac3{4(k+1)})-\ve}$
(see Remark \ref{Algnumbersrem} below).

\vspace{3pt}

Theorem \ref{MAINTHM3gen} %
generalizes 
\cite[Theorems 1.2 and 1.3]{effopp1}
from the case when $M=1_2$ and either the first or second column of $\vecxi$ vanishes,\footnote{Also
in \cite{effopp1} it is assumed that $k\geq2$.} 
to arbitrary $M$ and $\vecxi$.
Theorem~\ref{MAINTHM3gen} also %
generalizes 
\cite[Theorem 3.1]{SASL}
from the case when $k=1$, $M=1_2$ and $N=1$.
However, the bound in Theorem \ref{MAINTHM3gen}
involves stronger Sobolev norms on $f$ and $h$
than in both \cite[Theorems 1.2 and 1.3]{effopp1}
and \cite[Theorem 3.1]{SASL}.
The dependence on $y$ here is also slightly different as compared with \cite[Theorems 1.2 and 1.3]{effopp1}
and \cite[Theorem 3.1]{SASL}.
However, it should be noted that all these results have the same basic feature that the bound
decays like $y^{\frac14-\ve}$
when $\vecxi$ is either Lebesgue generic or satisfies a suitable Diophantine condition.
(Note that the notation in the present paper differs from the notation in \cite{effopp1};
see Remark \ref{effopp1TRANSLrem} for how to translate between the two.)

\vspace{5pt}

In the special case when $M=1_2$, for each $y$,
the orbit $\Gamma(1_2,\vecxi)\mathrm{u}_{\R}\mathrm{a}_y$ considered in
\eqref{MAINTHM3genres1} %
is a lift of a \textit{closed} horocycle on $\Gamma(N)\bs\SL(2,\R)$.
In this case, using the fact that the unipotent elements %
$\mathrm{u}_{jN}$ belong to $\Gamma(N)$ for all $j\in\Z$,
Theorem \ref{MAINTHM3gen} implies the following equidistribution result
concerning orbits which are ``smeared out''
by replacing the density $h(x)$ by $T^{-1}h(T^{-1}x)$ with $T\geq1$ arbitrary.
\begin{cor}\label{MAINTHM3cor2}
Let $k,N\in\Z^+$ and 
set $X=\GaG$ with $\mathrm{G}=\SL(2,\R)\ltimes(\R^2)^k$
and $\Gamma=\Gamma(N)\ltimes(\Z^2)^k$.
Let $m\geq\max(321,k+1)$,
and set $n=3m+3k+5$ and $\alpha=\frac12(m+321)$.
Then for any $f\in\C_\alpha^n(X)$, $\eta\in\C_0^5(\R)$, $h\in\C_2^5(\R)$
and any $\vecxi\in(\R^2)^k$, $0<y\leq1$ and $T\geq1$, we have
\begin{align}\notag
\Biggl|\frac 1T\int_\R f\bigl(\Gamma(1_2, {\vecxi})\mathrm{u}_x\mathrm{a}_y\bigr)
\eta(x)h(T^{-1}x)\,dx-\int_X f\,d\mu\int_{\R}\eta(Tx)h(x)\,dx\Biggr|
\hspace{80pt}
\\\label{MAINTHM3cor2res1}
\ll 
\|f\|_{\C_{\alpha}^{n}}\, \|\eta\|_{\C_0^5} \, \|h\|_{\C_2^5}\,\frac1T\sum_{j\in\Z}
\frac{\delta_{m}(y;\vecxi\mathrm{u}_{jN})}{(1+|j|/T)^{2}},
\end{align}
where the implied constant depends only on $k$, $N$ and $m$.
\end{cor}

Let us note that this corollary provides an effective version of the 
statement that for any $\vecxi\in(\R^2)^k$
such that $\vecq\vecxi\notin\Z^2$ for all $\vecq\in\Z^k\setminus\{0\}$,
we have
\begin{align}\label{MAINTHM3cor2implication}
\frac 1T\int_\R f\bigl(\Gamma(1_2, {\vecxi})\mathrm{u}_x\mathrm{a}_y\bigr)\eta(x)h(T^{-1}x)\,dx
-\int_X f\,d\mu\int_{\R}\eta(Tx)h(x)\,dx\to0
\end{align}
as $y\to0$, \textit{uniformly over all $T\geq1$.}
Indeed, for any $\vecalf\in\R^2$ we have
$\|\vecalf\,\mathrm{u}_{jN}\|_{\Z}\geq\lambda(\vecalf)$
where $\lambda(\vecalf):=\|\alpha_2\|_{\Z}$ if $\alpha_1\in\Z$
and otherwise $\lambda(\vecalf):=\|\alpha_1\|_{\Z}$.
Using this bound (for $\vecalf=d\vecq\vecxi$) 
together with the definition of $\delta_{m}(y;\vecxi\mathrm{u}_{jN})$
(see \eqref{deltaNEWDEF})
and the fact that
$\frac1T\sum_{j\in\Z}\bigl(1+|j|/T\bigr)^{-2}\ll1$ uniformly over $T\geq1$,
it follows that the right-hand side of \eqref{MAINTHM3cor2res1} is
\begin{align}\label{MAINTHM3cor2strongerREMpf2}
\ll\|f\|_{\C_{\alpha}^{n}}\, \|\eta\|_{\C_0^5} \, \|h\|_{\C_2^5}\,
\sum_{\vecq\in\Z^k\setminus\{\bn\}}\sum_{d=1}^{\infty}\frac{\tau(d)}{\|\vecq\|^{m}d^{3/2}}
\biggl(1+\frac{\lambda(d\vecq\vecxi)}{d\sqrt y}\biggr)^{\hspace{-3pt}-1}.
\end{align}
This expression is independent of $T$, and tends to zero as $y\to0$
provided that $\vecq\vecxi\notin\Z^2$ for all $\vecq\in\Z^k\setminus\{0\}$,
thus proving our claim.

Let us note that in the special case $\eta\equiv1$,
Corollary \ref{MAINTHM3cor2}
solves %
the problem described in
\cite[p.\ 150, around (17)]{effopp1}.
Because of this, %
by using
Corollary \ref{MAINTHM3cor2} it should be possible to extend the arguments in
\cite[Sec.\ 9]{effopp1}
so as to prove,
for general $k\geq2$,
an effective version of the main result in
Marklof \cite[Theorem 1.6]{MarklofpaircorrII}
concerning the fine-scale pair correlation density of the sequence
$\|\vecm-\vecalf\|^k$ ($\vecm\in\Z^k$),
where $\vecalf$ is an arbitrary fixed vector in $\R^k$  %
subject to a certain Diophantine condition.
The possibility of choosing, more generally,
$\eta(x)=e^{2\pi itx}$ in Corollary~\ref{MAINTHM3cor2}
is of interest since it opens the door to proving %
effective asymptotic counting results on the number of 
integer vectors $\vecm_1,\vecm_2\in\Z^k$ with $\|\vecm_1\|,\|\vecm_2\|<T$
for which the value of the inhomogeneous quadratic form 
$\|\vecm_1-\vecalf\|^2-\|\vecm_2-\vecalf\|^2$ %
lies in an arbitrary given real interval 
of size $\gg T^{2-k}$.
\footnote{The starting point for obtaining such results is 
the following generalization of the formula \cite[(130)]{effopp1}:
\begin{align*}
&T^{2-k}\int_{\R}
\Theta_f\left(u+T^{-2}i,0;\cmatr{\bn}{\vecalf}\right)
\overline{\Theta_g\left(u+T^{-2}i,0;\cmatr{\bn}{\vecalf}\right)}
\, e^{2\pi itu}\, h\bigl(T^{2-k}u\bigr)\,du
\\
&=\frac1{T^k}\sum_{\vecm_1\in\Z^k}\sum_{\vecm_2\in\Z^k}
f\bigl(T^{-1}(\vecm_1-\vecalf)\bigr)
\,
\overline{g\bigl(T^{-1}(\vecm_2-\vecalf)\bigr)}
\hspace{4pt} \hh\left(-\tfrac12 T^{k-2}\bigl(\|\vecm_1-\vecalf\|^2-\|\vecm_2-\vecalf\|^2+2t\bigr)\right).
\end{align*}
See loc.\ cit.\ for an explanation of the notation.}
Such a result was obtained for $k=2$ in
\cite[Cor.\ 1.5]{effopp1},
but it would be new for $k\geq3$, where we stress that the interval may \textit{shrink} with $T$.\footnote{In a very recent paper by Kim, Marklof and Welsh, \cite[Theorem 1.2]{KMW26}, a similar (but ineffective) result allowing shrinking test intervals is obtained for the 
homogeneous quadratic form 
$\bigl(c_1m_1^2+c_2m_2^2+c_3m_3^2\bigr)-\bigl(c_1m_4^2+c_2m_5^2+c_3m_6^2\bigr)$
on $\Z^6$,
under suitable Diophantine
conditions on the coefficients $c_1,c_2,c_3\in\R_{>0}$.}
We hope to return to these questions in a later paper.

\vspace{5pt}

We next turn to our second main result,
which is an effective equidistribution theorem for long pieces of 
$\mathrm{u}_{\R}$-orbits in $X$.
In order to simplify the presentation,
we will restrict ourselves to the case $N=1$, i.e.\ 
in the following we will keep
$\Gamma'=\SL(2,\Z)$ and $\Gamma=\SL(2,\Z)\ltimes(\Z^2)^k$.

It follows from the results of Ratner \cite{mR91}
that every $\mathrm{u}_{\R}$-orbit in $X$ has a closure which is homogeneous.
In other words, for any given starting point $x=\Gamma g\in X$ ($g\in\mathrm{G}$),
there exists a closed connected subgroup $H\subset\mathrm{G}$ such that
$\mathrm{u}_{\R}\subset H$,
$\Gamma\cap gHg^{-1}$ is a lattice in $gHg^{-1}$,
and the closure of $x\mathrm{u}_{\R}$ in $X$ equals
$xH=\Gamma\bs\Gamma gH$;
and then the orbit $x\mathrm{u}_{\R}$
is asymptotically equidistributed in the space $xH$ 
with respect to its $H$-invariant Borel probability measure
\cite[Theorem B]{mR91}.

For our specific homogeneous space $X$ it is not difficult to 
explicitly list those subgroups $H$ which can occur in the above statement
(see for example the arguments in \cite[pp.\ 735--737]{DMS}),
and thereby in particular deduce a precise criterion for when
$x\mathrm{u}_{\R}$ is asymptotically equidistributed in $(X,\mu)$.
We state this result, in a format which will be convenient for us,
in Theorem \ref{INEFFECTIVETHM2} below.
For the statement we need to introduce some further notation.

Let us write $\ASL(2,\R)=\SL(2,\R)\ltimes\R^2$.
In other words, $\ASL(2,\R)$ is the group [$\G$ for $k=1$];
in particular we represent the elements of $\ASL(2,\R)$ 
by pairs in $\G'\times\R^2$,
with multiplication law given by \eqref{multlaw},
and we view $\G'$ as a subgroup of $\ASL(2,\R)$ through $M\mapsto(M,\bn)$.
The group $\ASL(2,\R)$ acts on $\R^2$ from the right through
\begin{align*}
\vecw(M,\vecv):=\vecw M+\vecv\qquad
\bigl(\vecw\in\R^2, \: (M,\vecv)\in\ASL(2,\R)\bigr).
\end{align*}
Note that for every $g\in\ASL(2,\R)$,
the set $\Z^2g=\{\vecw g\col \vecw\in\Z^2\}$
is a \textit{grid}, i.e.\ a translate of a lattice in $\R^2$.

For any $\vecq\in\R^k$, we denote by $\p_{\vecq}$ the homomorphism
\begin{align*}
\p_{\vecq}:\mathrm{G}\to\ASL(2,\R);\qquad \p_{\vecq}(M,\vecv):=(M,\vecq\vecv).
\end{align*}

\begin{thm}\textup{(Special case of Ratner, \cite[Theorem B]{mR91}.)}
\label{INEFFECTIVETHM2}
Let $\G=\SL(2,\R)\ltimes(\R^2)^k$ and $\Gamma=\SL(2,\Z)\ltimes(\Z^2)^k$,
and let $g\in\G$.
Assume that 
\begin{align}\label{INEFFECTIVETHM2ass}
&\Z^2\p_{\bn}(g)\cap(\{0\}\times\R)=\{\bn\}
\qquad
\text{and}\qquad \forall \vecq\in\Z^k\setminus\{\bn\}:
\hspace{6pt}
\Z^2\p_{\vecq}(g)\cap(\{0\}\times\R)=\emptyset.
\end{align}
Then the orbit $\Gamma g \mathrm{u}_{\R}$
is asymptotically equidistributed in $(X,\mu)$,
i.e., for any bounded continuous function $f$ on $X$, we have
\begin{align}\label{INEFFECTIVETHM2res}
\lim_{T\to+\infty}\frac1T\int_0^Tf(\Gamma g \mathrm{u}_t)\,dt
=\lim_{T\to+\infty}\frac1T\int_{-T}^0f(\Gamma g \mathrm{u}_t)\,dt
=\int_X f\,d\mu.
\end{align}
\end{thm}

We remark that the conditions in \eqref{INEFFECTIVETHM2ass}
are easily seen to be \textit{necessary} for equidistribution.
Indeed, if one of these conditions fails, 
it means that there exist some $\vecq\in\Z^k$ and $\alpha\in\R$,
with either $\vecq\neq\bn$ or $\alpha\neq0$,
such that $\Gamma g$ belongs to the set
\begin{align}\label{XqalphaDEF}
X_{\vecq,\alpha}:=\{\Gamma g\col g\in \G,\: (0,\alpha)\in\Z^2\p_{\vecq}(g)\}.
\end{align}
However, $X_{\vecq,\alpha}$
is a closed embedded submanifold of $X$
of codimension two,
invariant under the $\mathrm{u}_{\R}$-flow;
hence the whole orbit $\Gamma g\mathrm{u}_{\R}$
is contained in $X_{\vecq,\alpha}$,
and thus does not equidistribute in $X$.
(In particular, it should be noted that 
$\Gamma g\in X_{\bn,\alpha}$ holds for some $\alpha\neq0$
if and only if $\Gamma g\mathrm{u}_{\R}$ projects to a closed orbit in $X'$,
i.e.\ a closed horocycle.)

\vspace{5pt}

The following result is our second main theorem.
It provides an effective version of Theorem~\ref{INEFFECTIVETHM2}.
It will be deduced as a consequence of
Theorem \ref{MAINTHM3gen},
via a careful (continuous) splitting of the $\mathrm{u}_{\R}$-orbit under consideration.

For any $T>0$, $g\in \mathrm{G}$ and $\vecq\in\Z^k$, 
we let $\fR_T$ be the rectangle
\begin{align*}
\fR_T:=[-T^{-1},T^{-1}]\times[-1,1]=T^{-\frac12}[-1,1]^2 \mathrm{a}_T^{-1}
\end{align*}
in $\R^2$,
and set
\begin{align}\label{bMxiqTDEF}
S_{g,\vecq}(T)
:=\sup\bigl\{S\geq0\col S\,\fR_T\cap \Z^2\,\p_{\vecq}(g)=\emptyset_{\vecq}\bigr\},
\end{align}
where $\emptyset_{\bn}:=\{\bn\}$ but $\emptyset_{\vecq}:=\emptyset$ for $\vecq\neq\bn$.
We also introduce the shorthand notation
\begin{align}\label{scrLjxDEF}
\scrL_j(x):=x\bigl(\log(2+x^{-1})\bigr)^j\qquad (x>0,\: j\geq0).
\end{align}

\begin{thm}\label{MAINTHM4}
Let $k\in\Z^+$,
$\mathrm{G}=\SL(2,\R)\ltimes(\R^2)^k$,
$\Gamma=\oGamma=\SL(2,\Z)\ltimes(\Z^2)^k$
and 
$X=\GaG$.
Let $m\geq\max(321,k+1)$,
and set $n=3m+3k+5$ and $\alpha=\frac12(m+321)$.
Then for any $f\in\C_\alpha^n(X)$, any $h\in\C_c^5(\R)$ with $\supp(h)\subset[-1,1]$,
and any $g\in \mathrm{G}$ and $T\geq 2$, we have
\begin{align}\notag
&\biggl|\frac1T\int_{\R}f\bigl(\Gamma g\mathrm{u}_t\bigr) h\Bigl(\frac tT\Bigr)\,dt
-\int_X f\,d\mu \int_{\R}h\,dt\biggr|
\\\label{MAINTHM4res}
&\qquad \ll \|f\|_{\C_{\alpha}^{n}}\, \|h\|_{\C_0^5}
\, \Biggl\{\scrL_3\Bigl(S_{g,\bn}(T)^{-\frac12}\Bigr) %
+\sum_{\vecq\in\Z^k\setminus\{\bn\}}\sum_{d=1}^{\infty}\frac{\tau(d)}{\|\vecq\|^{m}d^{3/2}}
\, \scrL_1\biggl(\frac1{1+d^{-1}S_{g,d\vecq}(T)}\biggr)\Biggr\},
\end{align}
where the implied constant depends only on $k$ and $m$.
\end{thm}

Let us note that for any given $g\in\mathrm{G}$,
the expression within brackets in 
the right-hand side of \eqref{MAINTHM4res}
tends to zero as $T\to\infty$
if (and only if) the conditions in \eqref{INEFFECTIVETHM2ass} hold.
Indeed, %
for any fixed $\vecq\in\Z^k$,
the condition $\Z^2\p_{\vecq}(g)\cap(\{0\}\times\R)=\emptyset_{\vecq}$
implies that $S_{g,\vecq}(T)\to\infty$ as $T\to\infty$.
Hence \eqref{INEFFECTIVETHM2ass} ensures that each individual term in
the right-hand side of \eqref{MAINTHM4res} tends to zero;
and since also $\sum_{\vecq\in\Z^k\setminus\{\bn\}}\sum_{d=1}^{\infty}\frac{\tau(d)}{\|\vecq\|^{m}d^{3/2}}<\infty$,
it follows that the whole expression in the right-hand side of \eqref{MAINTHM4res} tends to zero,
as claimed.
This shows that Theorem \ref{MAINTHM4} is indeed an effective version of 
Theorem \ref{INEFFECTIVETHM2},
for \textit{smoothed} ergodic averages.
(The smoothing can easily be removed through an approximation argument,
choosing $h$ in \eqref{MAINTHM4res} 
to approximate the characteristic function of $[0,1]$ or $[-1,0]$;
however the resulting bound will have a slower decay as a function of $T$.)
It is not difficult to verify that for \textit{almost every $g\in\G$},
the right-hand side of
\eqref{MAINTHM4res}
decays as $T^{-\frac14+\ve}$ as $T\to+\infty$;
see Lemmas \ref{SullivanloglawapplLEM} and \ref{THM4boundgenericdecayLEM}
below for more precise statements.
It would be interesting to relate the rate of decay of the right-hand side of
\eqref{MAINTHM4res} more directly to appropriate Diophantine properties of
the point $g\in\G$.   %

Theorem \ref{MAINTHM4} essentially generalizes
\cite[Theorem 1.6]{SASL}
from the case $k=1$ to general $k$.
However \cite[Theorem 1.6]{SASL}
concerns non-smoothed ergodic averages
and involves a weaker Sobolev norm on the test function $f$;
but on the other hand,
the aforementioned $T^{\ve-\frac14}$ decay of the bound in 
Theorem \ref{MAINTHM4} for generic $g\in\G$
is much stronger than the corresponding generic decay rate of the bound in
\cite[Theorem 1.6]{SASL},
which is $T^{\ve-\frac18}$.
(Following the proof of 
\cite[Theorem 1.6]{SASL},
the power of decay would not be %
improved by considering smoothed ergodic averages;
the fact that we obtain a better decay rate in the present work %
instead stems %
from the flexibility of having an
arbitrary $M\in\SL(2,\R)$ in Theorem \ref{MAINTHM3gen}.)

\vspace{10pt}

The organization of this paper is as follows.
In Section \ref{deltamxiSEC} we establish several properties
of the majorant function
$\delta_m(y;\vecxi)$
which appears in the bound in Theorem \ref{MAINTHM3gen}.
Among these properties is a precise form of the statement that
$\delta_m(y;\vecxi)$ decays as $y^{\frac14-\ve}$
as soon as either the first or the second column of $\vecxi$
satisfies an appropriate Diophantine condition
(see Lemma~\ref{BSUMDIOPHBOUNDLEM1}).

In Section \ref{initialstepsSEC} we 
set up basic notation and recall some basic facts from \cite[Sec.\ 4]{SASL}
regarding the Fourier decomposition of 
a given function $f$ on $X$ with respect to the torus variable.
The contribution from the zeroth Fourier coefficient of $f$
to the integral $\int_\R f\bigl(\Gamma (1_2,\vecxi)M \mathrm{u}_x\mathrm{a}_y\bigr) h(x)\,dx$
is a weighted average along a translate of a
horocycle in $\Gamma(N)\bs\SL(2,\R)$,
which is expanding as $y\to0$.
Precise results on the effective equidistribution of such horocycle averages are well-known,
and using such a result we verify in Section \ref{LEADINGTERMSEC} that
the contribution from the zeroth Fourier coefficient of $f$
equals $\int_{\R}h\,dx$ times the volume average of $f$, i.e.\ $\int_X f\,d\mu$, 
up to an error which is by far subsumed by the 
right-hand side in \eqref{MAINTHM3genres1}.
After this, in order to prove Theorem \ref{MAINTHM3gen} it remains to prove that the
total contribution from all the non-zero Fourier coefficients is bounded
by the right-hand side in \eqref{MAINTHM3genres1}.

In Section \ref{cancellationSEC} we prove a bound 
which establishes cancellation in a linear exponential sum
running over all matrices in an arbitrary coset of $\Gamma(N)$ in $\SL(2,\Z)$.
The key tool for this proof is the Hardy--Littlewood--Ramanujan circle method and more precisely, the delta symbol version of it. 
Next, in Section \ref{midapproachSEC} we apply the exponential sum bound from 
Section \ref{cancellationSEC} in order to bound all the contributions from non-zero Fourier coefficients
of the test function $f$ 
to the integral $\int_\R f\bigl(\Gamma (1_2,\vecxi)M \mathrm{u}_x\mathrm{a}_y\bigr) h(x)\,dx$,
and thereby conclude the proof of Theorem \ref{MAINTHM3gen}.
At the end of Section \ref{midapproachSEC}
we also give the proof of Corollary \ref{MAINTHM3cor2}.
Finally, in Section \ref{genorbSEC} we prove Theorem \ref{MAINTHM4},
by exhibiting an appropriate continuous splitting of the given $\mathrm{u}_{\R}$-orbit
and applying Theorem \ref{MAINTHM3gen} to each part.
We conclude the section by proving Lemmas \ref{SullivanloglawapplLEM} and \ref{THM4boundgenericdecayLEM}
which make precise the statement that for generic $g$, the bound in 
Theorem~\ref{MAINTHM4}
decays as $T^{-\frac14+\ve}$.

\vspace{10pt}

\textbf{Acknowledgment:}
This material is based upon work supported by the Swedish Research Council under grant no.\ 2021-06594 while the authors were in residence at the Institut Mittag-Leffler in Djursholm, Sweden, during the program ``Analytic Number Theory'' (January to April 2024), when a large part of this project was carried out. We are grateful to Institut Mittag-Leffler for providing excellent working conditions.
We are also grateful to Damaris Schindler and Wooyeon Kim for helpful discussions.

\vspace{10pt}

\section{Some notation}
\label{NOTATION_SEC}

We use the standard notation $A=O(B)$ or $A\ll B$ meaning $|A|\leq CB$ for some constant $C>0$.
We will also use $A\asymp B$ to denote $A\ll B\ll A$.
The implicit constant $C$ will always be allowed to depend on $k$ and $N$ without any explicit mention.
If we wish to indicate that $C$ also depends on some other quantities $f,g,h$, we will use the notation
$A\ll_{f,g,h} B$ or $A=O_{f,g,h}(B)$.

We will use the standard notation $e(z):=e^{2\pi i z}$.

Recall from Section \ref{intro} that 
$\mathrm{G}'=\SL(2,\R)$, $\mathrm{G}=\mathrm{G}'\ltimes(\R^2)^k$,
$\oGamma=\SL(2,\Z)\ltimes(\Z^2)^k$
and $\Gamma=\Gamma(N)\ltimes(\Z^2)^k$.
We also write
$\oGamma'=\SL(2,\Z)$ and $\Gamma'=\Gamma(N)$.

\section{Properties of the majorant $\delta_{m}(y;\vecxi)$}
\label{deltamxiSEC}

\subsection{Basic bounds}

Recall that $\delta_m(y;\vecxi)$ was defined in \eqref{deltaNEWDEF}.
\begin{lem}\label{deltambasicfactLEM1}
For any $\vecxi\in(\R^2)^k$,
$\delta_m(y;\vecxi)$ is increasing as a function of $y>0$,
and $\delta_m(y';\vecxi)\leq(y'/y)^{\frac12}\delta_m(y;\vecxi)$
for all $y'\geq y>0$.
\end{lem}
\begin{proof}
Immediate by inspection in \eqref{deltaNEWDEF}.
\end{proof}
\begin{lem}\label{basicboundLEM1}
For all real $A\geq1$,
\begin{align}\label{handwrp77midimproved}
\sum_{1\leq d\leq A}\frac{\tau(d)}{\sqrt d}\ll \sqrt{A}\,\log(A+1)
\qquad\text{and}\qquad
\sum_{d\geq A}\frac{\tau(d)}{d^{3/2}}\ll \frac1{\sqrt A} %
\log(A+1). %
\end{align}
\end{lem}
\begin{proof}
Easy consequences of the well-known bound
$\sum_{1\leq d\leq X}\tau(d)\ll X\log(X+1)$.
\end{proof}

\begin{lem}\label{basicboundLEM2}
For any integers $m>k\geq1$, and any $\vecxi\in(\R^2)^k$ and $0<y\leq1$,
\begin{align*}
\sum_{{\vecq}\in\Z^k\setminus\{\bn\}}\sum_{d>y^{-1/2}}\frac{\tau(d)}{\|\vecq\|^m d^{3/2}}
\biggl(1+\frac{\big\| d\vecq\vecxi\big\|_{\Z}}{d\sqrt y}\biggr)^{\hspace{-3pt}-1}
\ll_k\sum_{d>y^{-1/2}}\frac{\tau(d)}{d^{3/2}}
\ll y^{\frac14}\log(y^{-1}+1).
\end{align*}
\end{lem}
\begin{proof}
Immediate from Lemma \ref{basicboundLEM1}.
\end{proof}
The above Lemma \ref{basicboundLEM2}, 
combined with Lemma \ref{deltalowboundLEM} below,
shows that the order of magnitude of 
$\delta_{m}(y;\vecxi)$ would remain the same,
uniformly over all $\vecxi$ and $0<y\leq1$,
if in the definition \eqref{deltaNEWDEF}
we would restrict the summation range for $d$
to $1\leq d\leq y^{-1/2}$.
\begin{lem}\label{deltalowboundLEM}
For any $\vecv\in\R^2$ and $0<y\leq1$,
\begin{align}\label{deltalowboundLEMres}
\sum_{\frac12 y^{-\frac12}\leq d\leq y^{-\frac12}}\frac{\tau(d)}{d^{3/2}}
\Bigl(1+\frac{\|d\vecv\|_{\Z}}{d\sqrt y}\Bigr)^{-1}
\gg y^{\frac14}\log(y^{-1}+1),
\end{align}
where the implied constant is absolute.
 In particular,
$\delta_m(y;\vecxi)\gg y^{\frac14}\log(y^{-1}+1)$,
for all $0<y\leq1$, $m>k$ and $\vecxi\in(\R^2)^k$.
\end{lem}
\begin{proof}
We have $\|\vecw\|_{\Z}\leq\sqrt{1/2}<1$ for all $\vecw\in\R^2$.
Therefore,
$\bigl(1+\frac{\|d\vecv\|_{\Z}}{d\sqrt y}\bigr)^{-1}>\bigl(1+\frac1{d\sqrt y}\bigr)^{-1}
\geq\frac13$ for all $0<y\leq1$
and $d\geq\frac12 y^{-\frac12}$.
As a result, the sum on the left-hand side of
\eqref{deltalowboundLEMres} is
\begin{align*}
\gg y^{\frac34}\sum_{\frac12 y^{-\frac12}\leq d\leq y^{-\frac12}}\tau(d).
\end{align*}
Hence the lemma follows using the well-known estimate
$\sum_{d<X}\tau(d)=X\log X+O(X)$ as $X\to\infty$.
\end{proof}

\subsection{Generic behavior}
Our results on the decay of $\delta_m(y;\vecxi)$
will involve hypotheses on either of the two columns of $\vecxi$.
To state these results, it is convenient to introduce 
the following short-hand notation:
For any vector $\vecpsi\in\R^k$, write $\tvecpsi:=\bigl(\trans\vecpsi\hspace{5pt} \bn\bigr)\in(\R^2)^k$,
i.e.\ $\tvecpsi$ is the $k\times 2$ matrix with left column $\trans\vecpsi$ and vanishing right column;
and then set
\begin{align}\label{deltamxi1DEF}
\delta_m(y;\vecpsi):=\delta_m(y;\tvecpsi).
\end{align}
Using $\bigl\|d\vecq\bigl(\trans\vecxi_1\:\trans\vecxi_2\bigr)\bigr\|_{\Z}\geq 
\|d\vecq\, \trans\vecxi_i\|_{\Z}$
($i=1,2$)
it is then immediate from \eqref{deltaNEWDEF}
that
\begin{align*}
\delta_m\bigl(y;(\trans\vecxi_1\,\trans\vecxi_2)\bigr)\leq\delta_m(y;\vecxi_i)
\qquad\text{for all $m>k$, $\vecxi_1,\vecxi_2\in\R^k$, $y>0$ and $i=1,2$.}
\end{align*}
Hence a bound on 
$\delta_m(y;\vecxi_i)$ for either $i=1$ or $i=2$ %
implies
the same bound on $\delta_m(y;\vecxi)$ where $\vecxi=(\trans\vecxi_1\:\trans\vecxi_2)$.

We have the following metric result,
an analogue of 
\cite[Remark 3]{effopp1}:
\begin{lem}\label{labaaxiLEM}
For any fixed $\ve>0$ and $m>k$
and for Lebesgue almost all
$\vecxi_1\in\R^k$, we have
$\delta_m(y;\vecxi_1)\ll y^{\frac14}(\log(y^{-1}))^{2+\ve}$ as $y\to0$.
\end{lem}
\begin{proof}
It follows from \eqref{deltamxi1DEF} and \eqref{deltaNEWDEF} that
\begin{align}\label{labaaxiLEMpf2}
\int_{(0,1)^k}\delta_m(y;\vecxi_1)\,d\vecxi_1
=\sum_{\vecq\in\Z^k\setminus\{\bn\}}\sum_{d=1}^{\infty}
\frac{\tau(d)}{\|\vecq\|^md^{3/2}}
\int_{(0,1)^k}\biggl(1+\frac{\|d\vecq\trans\vecxi_1\|_{\Z}}{d\sqrt y}\biggr)^{\hspace{-3pt}-1}\,d\vecxi_1.
\end{align}
For any $\vecq$ and $d$ in the above sum,
the map $\vecxi_1\mapsto\|d\vecq\trans\vecxi_1\|_{\Z}$
push-forwards the Lebesgue measure on $(0,1)^k$ to
2 times Lebesgue measure on $(0,\frac12)$.
Hence
\begin{align*}
\int_{(0,1)^k}\biggl(1+\frac{\|d\vecq\trans\vecxi_1\|_{\Z}}{d\sqrt y}\biggr)^{\hspace{-3pt}-1}\,d\vecxi_1
=2\int_0^{1/2}\Bigl(1+\frac x{d\sqrt y}\Bigr)^{-1}\,dx
\ll \begin{cases}
1&\text{if }\: d\sqrt y\geq\frac14
\\[5pt]
d\sqrt y\log\bigl(\frac1{d\sqrt y}\bigr)&\text{if }\: d\sqrt y\leq\frac14.
\end{cases}
\end{align*}
Using this bound in \eqref{labaaxiLEMpf2},
and also using Lemma \ref{basicboundLEM1} and a simple
dyadic decomposition argument,
we obtain
\begin{align}\label{labaaxiLEMpf1}
\int_{(0,1)^k}\delta_m(y;\vecxi_1)\,d\vecxi_1\ll y^{\frac14}\log(y^{-1}),
\qquad\forall 0<y<\tfrac12.
\end{align}
It follows that
for every $0<y<\frac12$ and every $A\geq1$,
the set of $\vecxi_1\in(0,1)^k$ satisfying
$\delta_m(y;\vecxi_1)\geq Ay^{\frac14}\log(y^{-1})$
has Lebesgue measure $\ll A^{-1}$.
Hence, by Borel-Cantelli,
for almost every $\vecxi_1\in(0,1)^k$,
there is some $J\in\Z^+$ such that
for all
$j\geq J$,
$y=2^{-j-1}$ satisfies
$\delta_m(y;\vecxi_1)<j^{1+\ve}y^{\frac14}\log(y^{-1})$.
Since $\delta_m(y;\vecxi_1)$ is increasing,
it then follows that
$\delta_m(y;\vecxi_1)<j^{1+\ve}(2^{-j-1})^{\frac14}\log(2^{j+1})
\ll j^{2+\ve}\,2^{-\frac14\hspace{1pt} j}$
for all $2^{-j-2}\leq y\leq 2^{-j-1}$,
and all $j\geq J$. Therefore, 
$\delta_m(y;\vecxi_1)\ll y^{\frac14}(\log(y^{-1}))^{2+\ve}$
for all sufficiently small $y$.
\end{proof}

\subsection{Behavior when imposing a Diophantine condition}
\label{DiophCondSec}

As in 
\cite[Sec.\ 3]{effopp1},
given real numbers $\kappa\geq k$ and $\alpha\geq1$,
we say that a vector $\vecxi\in\R^k$ is 
\textit{$(\kappa,\alpha)$--LFD}\footnote{LFD is short for linear form Diophantine.}
if there is a constant $c>0$ such that
\begin{align}\label{LFDdef}
\|d\vecq\trans\vecxi\|_{\Z}\geq c d^{-\alpha}\|\vecq\|^{-\kappa}
\qquad\text{for all }\: d\in\Z^+\:\text{ and }\: \vecq\in\Z^k\setminus\{\bn\}.
\end{align}

The next lemma is analogous to %
\cite[Lemma 3.3]{effopp1},
and the proof follows an almost identical argument.
\begin{lem}\label{BSUMDIOPHBOUNDLEM1}
Assume that $\vecxi_1\in\R^k$ is $(\kappa,\alpha)$-LFD,
for some fixed $\kappa\geq k$ and $\alpha\geq1$,
and let $m>k+\frac{3\kappa}{2(\alpha+1)}$.
Then for any fixed $\ve>0$,
\begin{align}\label{BSUMDIOPHBOUNDLEM1res}
\delta_m(y;\vecxi_1)
\ll_{m,\ve,\vecxi_1} \bigl(y^{\frac3{4(\alpha+1)}}+y^{\frac14}\bigr)y^{-\ve}\qquad\text{as }\: y\to0^+.
\end{align}
\end{lem}

The proof will make use of the following auxiliary lemma,
which is related to
\cite[Lemma~3.2]{effopp1}.
\begin{lem}\label{DIOPHAUXLEM1}
Let $\ve>0$, $\eta\in\R$, $c>0$, $\kappa\geq1$,
and assume that
$\|d\eta\|_{\Z}\geq cd^{-\kappa}$ for all $d\in\Z^+$.
Then
\begin{align}\label{DIOPHAUXLEM1res}
\sum_{d=1}^\infty\frac {\tau(d)}{d^{3/2}+Td^{1/2}\| d\eta\|_{\Z}}  %
\ll_{\ve}\Bigl((cT)^{-\frac{3/2}{\kappa+1}}+T^{-\frac12}\Bigr)T^{\ve}
\qquad\text{for all }\: T\geq1.
\end{align}
\end{lem}
\begin{proof}
After using the bound $\tau(d)\ll_{\ve} d^{\ve/2}$,
the proof of \cite[Lemma 3.2]{effopp1}
carries over with very minor and easy modifications:
In place of \cite[(23)]{effopp1}
we obtain, for any $\ell\geq1$,
\begin{align}
\sum_{1\leq d\leq q_\ell/2}\frac{\tau(d)}{d^{1/2}\| d\eta\|_{\Z}}
\ll c^{-1}(\log q_\ell)q_{\ell-1}^{\kappa+(\ve-1)/2},
\end{align}
and in place of 
\cite[(26)]{effopp1}
we obtain
\begin{align}
\sum_{q_\ell/2<d\leq T}\frac{\tau(d)}{d^{3/2}+Td^{1/2}\| d\eta\|_{\Z}}
\ll q_\ell^{(\ve-3)/2} %
+T^{(\ve-1)/2}\log q_\ell,
\end{align}
and just as in the proof of 
\cite[Lemma 3.2]{effopp1},
we conclude by choosing $\ell\geq1$ so that
$q_{\ell-1}\leq (cT)^{\frac1{\kappa+1}}<q_{\ell}$
(having assumed $cT>1$ from the start).
\end{proof}

\begin{proof}[Proof of Lemma \ref{BSUMDIOPHBOUNDLEM1}]
Take $c>0$ such that \eqref{LFDdef} holds;
then by Lemma \ref{DIOPHAUXLEM1}, for 
every $\vecq\in\Z^k\setminus\{\bn\}$ and $0<y\leq1$ we have
\begin{align*}
\sum_{d=1}^{\infty}\frac{\tau(d)}{d^{3/2}}
\biggl(1+\frac{\big\| d\vecq\trans\vecxi_1\big\|_{\Z}}{d\sqrt y}\biggr)^{\hspace{-3pt}-1}
\ll_{\ve}
\biggl(\bigl(c\|\vecq\|^{-\kappa}y^{-\frac12}\bigr)^{-\frac{3/2}{\alpha+1}}+y^{\frac14}\biggr)\,y^{-\ve}.
\end{align*}
Note also that the sum in the left-hand side is
$\ll 1$, trivially.
Hence from the definitions, \eqref{deltamxi1DEF} and \eqref{deltaNEWDEF}, we obtain
\begin{align*}
\delta_m(y;\vecxi_1)
&\ll_{\ve}
\sum_{1\leq\|\vecq\|<(cy^{-1/2})^{1/\kappa}}\|\vecq\|^{-m}
\biggl(\bigl(c\|\vecq\|^{-\kappa}y^{-\frac12}\bigr)^{-\frac{3/2}{\alpha+1}}+y^{\frac14}\biggr)\,y^{-\ve}
+\sum_{\|\vecq\|\geq(cy^{-1/2})^{1/\kappa}}\|\vecq\|^{-m},
\end{align*}
and using $m>k+\frac{3\kappa}{2(\alpha+1)}$,
this leads to the bound in \eqref{BSUMDIOPHBOUNDLEM1res}.
\end{proof}
\begin{remark}\label{HdimRemark}
Recall that by 
\cite[Lemma 3.1]{effopp1},
for any $\kappa>k$ and $\alpha>1$,
the set of all $\vecxi_1\in\R^k$ which are not
$(\kappa,\alpha)$-LFD 
has Hausdorff dimension 
$k-1+\max\bigl(\frac{k+1}{\kappa+1},\frac2{\alpha+1}\bigr)$.
For example, when using this result in combination with
Lemma \ref{BSUMDIOPHBOUNDLEM1},
we obtain a sharpening of the ``Lebesgue almost all''
in Lemma \ref{labaaxiLEM},
at the price of weakening the decay rate $y^{\frac14}(\log(y^{-1}))^{2+\ve}$
to $y^{\frac14-\ve}$, and imposing a stronger lower bound on $m$:
Indeed, applying Lemma \ref{BSUMDIOPHBOUNDLEM1}
with $\kappa=(3k+1)/2$ and $\alpha=2$,
we conclude that if 
$m>(7k+1)/4$
then for every $\vecxi_1\in\R^k$ which is
$(\kappa,\alpha)$-LFD we have $\delta_m(y;\vecxi_1)\ll y^{\frac14-\ve}$. On the other hand, 
the set of all $\vecxi_1\in\R^k$ which are not
$(\kappa,\alpha)$-LFD 
has Hausdorff dimension $k-\frac13$.
\end{remark}
\begin{remark}\label{Algnumbersrem}
By Schmidt, \cite{wS70},
if $\xi_1,\ldots,\xi_k$ are (real) \textit{algebraic} numbers
and
$1,\xi_1,\ldots,\xi_k$ are linearly independent over $\Q$,
then $\vecxi=(\xi_1,\ldots,\xi_k)$ is $(\kappa,\kappa)$-LFD for every $\kappa>k$.
Hence, for such a $\vecxi$,
Lemma \ref{BSUMDIOPHBOUNDLEM1} implies that
for any $m>k+\frac32$ and $\ve>0$
we have
$\delta_m(y;\vecxi)\ll y^{\min(\frac 3{4(k+1)},\frac14)-\ve}$ as $y\to0^+$.
\end{remark}

\section{Some initial steps} %
\label{initialstepsSEC}

\subsection{Fourier decomposition with respect to the torus variable}
\label{FOURIERDECSEC}

As in \cite{SASL} and \cite{effopp1},
the starting point of the proof of Theorem \ref{MAINTHM3gen} 
is to consider %
the Fourier decomposition
of the given test function $f$ on $X$
with respect to the torus variable.
In this section we recall the set-up and some basic bounds
from \cite[Sec.\ 4]{effopp1}
relating to this Fourier decomposition.

There are some differences between the notation in the present paper and the one used in \cite{effopp1};
we explain in Remark \ref{effopp1TRANSLrem} below how to translate between the two.

Given a function $f$ on $X=\GaG$, we will often view $f$ as a function on $\G$ via $f(g)=f(\Gamma g)$,
and we will write $f(M,\vecv)$ in place of $f((M,\vecv))$, for $(M,\vecv)\in \G$. %
For every $\vecn\in(\Z^2)^k$ we have
$(1_2,\vecn)\in\Gamma$ and so
$f((1_2,\vecxi)M)=f((1_2,\vecxi+\vecn)M)$ for all $\vecxi\in(\R^2)^k$ and $M\in \G'$.
Hence, assuming $f\in\C^{k+1}(X)$, we have a Fourier decomposition
\begin{align}\label{FOURIERSERIES}
f((1_2,\vecxi)M)=\sum_{\vecm\in(\Z^2)^k}\widehat f(M,\vecm)\, e(\tr(\vecm\trans\vecxi)).
\end{align}
Here the 
sum on the right-hand side is 
absolutely convergent, uniformly %
over $\vecxi$ and $M$ in any compact subsets of $(\R^2)^k$ and $\G'$.
The Fourier coefficients $\hf(M,\vecm)$ are given by
\begin{align}\label{WHFDEF}
\widehat f(M,\vecm):=\int_{(\Z^2)^k\backslash(\R^2)^k} f((1_2,\vecxi)M)\, e(-\tr(\vecm\trans\vecxi))\,d\vecxi,
\qquad M\in \G',\: \vecm\in(\Z^2)^k.
\end{align}
Here $d\vecxi$ denotes Lebesgue measure on $(\R^2)^k$.
The fact that $f$ is left $T$-invariant for every $T\in\Gamma'=\Gamma(N)$
translates into the relation
\cite[Lemma 4.1]{effopp1}
\begin{align}\label{BASICFOURIERLEMRES1}
\wh f(TM,\vecm)=\wh f(M,\vecm\trans T^{-1}),\qquad\forall T\in\Gamma',\: M\in \G',\:\vecm\in(\Z^2)^k.
\end{align}

Given any $R\in\overline\Gamma'=\SL(2,\Z)$, we set
(as in \cite[(19)]{effopp1}):
\begin{align}\label{FRDEF}
f_R(M,\vecv):=f(R^{-1}(M,\vecv))=f( R^{-1}M, \vecv).
\end{align}
Since $\Gamma'$ is normal in $\overline\Gamma'$, $f_R$ is also left $\Gamma$-invariant, i.e.\ $f_R$ can be
viewed as a function on $X$,
and then the formulas in 
\eqref{FOURIERSERIES}--\eqref{BASICFOURIERLEMRES1}
hold with $f$ replaced by $f_R$.
Let us also note that
\begin{align}\label{effopp1p155lm1}
\widehat{f_R}(M,\vecm)=\hf(R^{-1}M,\vecm \trans R^{-1}),\qquad
\forall R\in\overline\Gamma',\:M\in \G',\: \vecm\in(\Z^2)^k,
\end{align}
and that, for any $f\in\C^m(X)$:
\begin{align}\label{normfReqnormf}
\|f_R\|_{\C_a^m}=\|f\|_{\C_a^m},\qquad \forall m\in\Z_{\geq0},\: a\in\R.
\end{align}

We next introduce some notation relating to the 
right action of $\overline\Gamma'=\SL(2,\Z)$  %
on $(\Z^2)^k$.
We call an orbit for this action %
an \textit{A-orbit} if it contains some element
of the form $(\trans \vecr\: \bn)$ with $\vecr\in\Z^k\setminus\{\bn\}$.
Every other non-zero orbit is called a \textit{B-orbit}.
We fix, once and for all, a set of representatives $A_k,B_k\subset(\Z^2)^k$ such that
$A_k$ contains exactly one element from each A-orbit and $B_k$ contains exactly one element from
each B-orbit,
and furthermore each $\veceta\in A_k$ is of the form $\veceta=(\trans \vecr\: \bn)$
and each $\veceta\in B_k$ has the property
that, writing
$\veceta=(\trans \vecr\:\trans \vecq)$
with $\vecr=(r_1,\ldots,r_k)$ and $\vecq=(q_1,\ldots,q_k)$,
there exist some $1\leq\ell_1<\ell_2\leq k$ such that
$r_j=0$ for all $j<\ell_1$, $q_j=0$ for all $j<\ell_2$,
and $r_{\ell_1}>0$, $0\leq r_{\ell_2}<|q_{\ell_2}|$.
The existence of such a set of representatives $B_k$ is guaranteed by
\cite[Lemma 4.2]{effopp1};
note that $B_k\neq\emptyset$ if and only if $k\geq2$.

Recall that $\overline\Gamma'=\SL(2,\Z)$ and $\Gamma'=\Gamma(N)$. We set
\begin{align}\label{GAMMAinfdef}
\overline\Gamma'_\infty:=\left\{\matr 1n01\col n\in\Z\right\}
\qquad\text{and}\qquad
\Gamma'_\infty:=\Gamma'\cap\overline\Gamma'_\infty=\left\{\matr 1{Nn}01\col n\in\Z\right\}.
\end{align}
For any subgroup $H$ of $\G'$ and any subset $A\subset \G'$ satisfying $HA=A$,
let $H\backslash A$ denote a set of representatives for the distinct cosets $Ha$ ($a\in A$).
We also define 
$\overline\Gamma'_\infty \backslash\overline\Gamma'/\Gamma'$ to be a set of representatives for the
double cosets of the form $\overline\Gamma'_\infty R\Gamma'$ with $R\in\overline\Gamma'$.
Finally, for each $R\in\overline\Gamma'$
we set
$[R]:=\Gamma'R=R\Gamma'$;
this is the set of all matrices in $\overline\Gamma'$ which are congruent to $R$ modulo $N$.
Using this notation, 
the Fourier decomposition in \eqref{FOURIERSERIES}
can be rewritten as follows
\cite[(37)]{effopp1}:
\begin{align}\notag
f((1_2,\vecxi)M)=\widehat f(M,\bn)+\sum_{\veceta\in A_k}
\sum_{R\in\overline\Gamma'_\infty \backslash\overline\Gamma'/\Gamma'} \,
\sum_{T\in\Gamma'_\infty\backslash [R]}
\widehat{f_R}(TM,\veceta)e(\tr(\veceta\trans T^{-1} \trans\vecxi))
\hspace{60pt}
\\\label{BASICFOURIER}
+\sum_{\veceta\in B_k}\sum_{R\in\overline\Gamma'/\Gamma'}\, \sum_{T\in [R]} 
\widehat{f_R}(TM,\veceta)e(\tr(\veceta\trans T^{-1}\trans\vecxi)).
\end{align}
Here it should be noted that
the function $M\mapsto\hf(M,\bn)$ is left $\Gamma'$-invariant,
and for any $\veceta\in A_k$ and $R\in\overline\Gamma'$, the function
$M\mapsto\widehat{f_R}(M,\veceta)$ is left $\Gamma'_\infty$-invariant
(indeed see \eqref{BASICFOURIERLEMRES1} and \eqref{effopp1p155lm1}
and note that $\veceta \trans T^{-1}=\veceta$ for all $\veceta\in A_k$ and $T\in\Gamma'_\infty$).

We will sometimes express $\widehat{f_R}(M,\veceta)$ in terms of Iwasawa co-ordinates,
that is we write %
\begin{align}\label{TFNIWASAWA}
\widehat{f_R}(u,v,\theta;\veceta):=\widehat{f_R}\left(\matr 1u01\matr{\sqrt v}00{1/\sqrt v}
\matr{\cos\theta}{-\sin\theta}{\sin\theta}{\cos\theta},\veceta\right),
\end{align}
for $u\in\R$, $v>0$, $\theta\in\R/2\pi\Z$, $\veceta\in(\Z^2)^k$.

The following bounds on 
$\widehat{f_R}$ and its derivatives will be used later in the paper.
\begin{lem}\label{DERDECAYTFNFROMCMNORMLEM2}
For any %
$R\in\oGamma'$, %
$\alpha\in\R_{\geq0}$, $\vecr\in\Z^k\setminus\{\bn\}$, 
integers $m,\ell_1,\ell_2,\ell_3\geq0$ and $f\in\C^{m+\ell}_\alpha(X)$,
where $\ell=\ell_1+\ell_2+\ell_3$, we have
\begin{align}\label{DERDECAYTFNFROMCMNORMLEM2RES}
&\left|\Bigl(\frac{\partial}{\partial u}\Bigr)^{\ell_1}
\Bigl(\frac{\partial}{\partial v}\Bigr)^{\ell_2}
\Bigl(\frac{\partial}{\partial\theta}\Bigr)^{\ell_3}
\widehat{f_R}\left(u,v,\theta;(\trans \vecr\: \bn)\right)\right|
\ll_{m,\ell,\alpha}\|f\|_{\C_\alpha^{m+\ell}}\|\vecr\|^{-m}v^{\frac m2-\ell_1-\ell_2}\min(1,v^{-\alpha}).
\end{align}
\end{lem}
\begin{proof}
Apply \cite[Lemma 4.5]{effopp1}
to $f_R$, and use \eqref{normfReqnormf}.
\end{proof}

\begin{lem}\label{PARTBDECAYFNLEM2}
For any $R\in\oGamma'$, $0<\beta<\frac12$, $\veceta\in B_k$, $m\geq0$,
$D\in\scrU(\lsl(2,\R))$ of order $\leq k$,
and any $f\in \C^{m+k}_0(X)$ and $T\in \G'$, %
\begin{align}\label{PARTBDECAYFNLEM2RES}
\bigl|\bigl[D\widehat{f_R}\bigr]\bigl(T,\veceta\bigr)\bigr|
\ll_m\frac{\|Df\|_{\C^m_0}}{\|T\|^{m(1-2\beta)}\|\veceta\|^{m\beta}}.
\end{align}
\end{lem}
(Here, of course, in ``$D\wh f_R$'', the differential operator $D$ acts on the \textit{first} variable of $\widehat{f_R}$.)
\begin{proof}
Set $g:=Df_R\in\C^m_0(X)$;
then by differentiation under the integration sign in 
\eqref{WHFDEF} we have
$D\hf_R=\widehat{g}$.
Hence the lemma follows from 
\cite[Lemma 4.6 and Remark 5]{effopp1}
applied to $g$,
combined with the relation $g=Df_R=(Df)_R$ and \eqref{normfReqnormf}.
\end{proof}

\begin{remark}\label{effopp1TRANSLrem}
We here explain how to translate between the notation of the present paper and the notation in \cite{effopp1}.
In \cite{effopp1}, %
the elements of $(\R^2)^{\oplus k}$ are represented as
$2k\times 1$ column matrices,
and $\G'=\SL(2,\R)$ acts from the left on $(\R^2)^{\oplus k}$ through
\begin{align}\label{effopp1TRANSLremdisc1}
\matr abcd \cmatr{\trans\vecv}{\trans\vecw}=\cmatr{a\trans\vecv+b\trans\vecw}{c\trans\vecv+d\trans\vecw}
\qquad\text{for }\: \matr abcd\in\G',\: \vecv,\vecw\in\R^k.
\end{align}
Here %
and in the following, 
we stick to the convention of the present paper
in that elements $\vecv,\vecw$ in $\R^k$ are row matrices;
therefore %
their transposes, $\trans\vecv$ and $\trans\vecw$, are column matrices. %
The elements of $\G=\G'\ltimes(\R^2)^k$ are in \cite{effopp1} represented by
pairs $(M,\vecu)$ in $\G'\times\trans(\R^{2k})$,
where $\trans(\R^{2k})$ denotes the space of $2k\times 1$ column matrices;
and the multiplication law in \cite{effopp1} is
$(M,\vecu)(M',\vecu')=(MM',\vecu+M\vecu')$.
The translation into the notation of the present paper is given by the bijection
\begin{align}\label{effopp1TRANSLremdisc2}
J:\G'\times\trans(\R^{2k})\to\G,  %
\qquad J\bigl((M,\vecu)\bigr):=(1_2,j(\vecu))M,
\end{align}
where $j$ is the map
\begin{align*}
j:\trans(\R^{2k})\to(\R^2)^k,\qquad j\left(\cmatr{\trans\vecv}{\trans\vecw}\right)
:=\bigl(\trans\vecw,-\trans\vecv\bigr),
\qquad (\vecv,\vecw\in\R^k).
\end{align*}
One verifies that
\begin{align}\label{effopp1TRANSLremdisc0}
j(M\vecu)=j(\vecu)M^{-1},\qquad\forall \vecu\in\trans(\R^{2k}),\: M\in\G',
\end{align}
and using this property it is easy to check that 
the map in \eqref{effopp1TRANSLremdisc2} is a Lie group isomorphism.
Note also that $J$ restricts to the identity map on the subgroup $\G'=\SL(2,\R)<\G$.

In \cite[Sec.\ 4]{effopp1},
the Fourier decomposition of
a $\C^{k+1}$ function $F:\G'\times\trans(\R^{2k})\to\CC$
which is left $\{1_2\}\times\trans(\Z^{2k})$--invariant
is expressed as
\begin{align}\label{effopp1TRANSLremdisc4}
F(M,\vecxi)=\sum_{\vecm\in\trans(\Z^{2k})}\hF(M,\vecm) e(\vecm\vecxi),
\end{align}
where %
$\hF$ is a function
on $\G'\times\trans(\Z^{2k})$,
and where ``$\vecm\vecxi$'' denotes the standard scalar product in $\trans(\R^{2k})$.
In particular, if $F:=f\circ J$ for some function $f\in\C^{k+1}(X)$,
then it follows that
the Fourier coefficients $\hf(M,\vecm)$ in \eqref{FOURIERSERIES}, \eqref{WHFDEF}
are related to the ones in \eqref{effopp1TRANSLremdisc4} through
\begin{align}\label{effopp1TRANSLremdisc3}
\hF(M,\vecm)=\hf(M,j(\vecm)),\qquad\forall M\in G',\: \vecm\in\trans(\Z^{2k}).
\end{align}
Using these relations, it is now easy to verify that e.g.\ 
the formulas \eqref{BASICFOURIERLEMRES1}, \eqref{effopp1p155lm1} 
and \eqref{BASICFOURIER} are indeed
equivalent to the corresponding formulas proved in \cite{effopp1}.
\end{remark}

\subsection{Obtaining the leading term}
\label{LEADINGTERMSEC}

When proving Theorem \ref{MAINTHM3gen},
our task is to study the integral
$\int_\R f\bigl(\Gamma (1_2,\vecxi)M \mathrm{u}_x\mathrm{a}_y\bigr) h(x)\,dx$.
Decomposing $f$ as in \eqref{BASICFOURIER}, 
we get
\begin{align}\notag
\int_\R f\bigl(\Gamma (1_2,\vecxi)M \mathrm{u}_x\mathrm{a}_y\bigr) h(x)\,dx
=\int_\R \wh f\left(M\mathrm{u}_x\mathrm{a}_y,\bn\right)h(x)\,dx
\hspace{120pt}
\\\label{MAINSTEP1}
+\sum_{\veceta\in A_k}\sum_{R\in \overline\Gamma'_\infty\backslash\overline\Gamma'/\Gamma'}
\sum_{T\in \Gamma'_\infty \backslash [R]}
e(\tr(\veceta\trans T^{-1}\trans\vecxi))\int_\R
\widehat{f_R}\left(TM\mathrm{u}_x\mathrm{a}_y,\veceta\right)h(x)\,dx
\\\notag
+\sum_{\veceta\in B_k}\sum_{R\in \overline\Gamma'/\Gamma'} \sum_{T\in[R]}
e(\tr(\veceta\trans T^{-1}\trans\vecxi))\int_\R
\widehat{f_R}\left(TM\mathrm{u}_x\mathrm{a}_y,\veceta\right)h(x)\,dx.
\end{align}

The first integral in the right-hand side of \eqref{MAINSTEP1}
is a weighted average along an expanding %
translate 
of a horocycle in 
$X':=\Gamma'\bs G'=\Gamma(N)\backslash\SL(2,\R)$. It is well-known that such an average tends to
$\int_{X'}\hf(\cdot,\bn)\,d\mu_{X'}\int_\R h\,dx
=\int_{\GaG} f\,d\mu\int_\R h\,dx$ 
with a polynomial rate as $y\to0$;
see
\cite[Prop.\ A.6]{KleinbockMargulis}.
We will here derive a precise statement based on an application of\label{KMref}
\cite[Theorem 1]{iha}.

First, note that for every $0<y\leq 1$ and $\beta\geq10y$,
by setting $T:=\beta/y$ and applying 
\cite[Theorem 1 and (14)]{iha}
we have for any $0<\ve<\frac12$:
\begin{align}\notag
\frac1{\beta}\int_0^\beta\hf\bigl(\Gamma\bigl(M\mathrm{u}_x\mathrm{a}_y,\bn\bigr)\bigr)\,dx
&=\frac1T\int_{0}^{T}\hf\bigl(\Gamma\bigl(M\mathrm{a}_y \mathrm{u}_t,\bn\bigr)\bigr)\,dt
\\\label{MAINTERMtreatm1}
&=\int_{X}f\,d\mu+O_{\ve}\bigl(\|f\|_{\C_0^4}\bigr)\cdot\bigl(U^{-\frac12+\ve}+T^{s_1-1}\bigr),
\end{align}
where $U:=T/\scrY\bigl(M\mathrm{a}_y \mathrm{a}_T\bigr)$
and where $s_1\in[\frac12,1)$ is defined by
$s_1=\frac12+\sqrt{\frac14-\lambda_1}$
if there exists a small eigenvalue $0<\lambda_1<\frac14$
of the Laplace operator on the hyperbolic surface
$\Gamma'\bs\HH$,
and otherwise $s_1=\frac12$.
Using the bound by Kim and Sarnak \cite{KS} towards the Ramanujan conjecture,
we know that 
$s_1\in[\frac12,\frac12+\frac7{64}]$.
Furthermore, by
\cite[(12)]{iha} we have
\begin{align}\label{MAINTERMtreatm3}
U=\frac T{\scrY\bigl(M \mathrm{a}_{\beta}\bigr)}\geq\frac{\min(1,\beta^2)}y\, \scrY(M)^{-1}.
\end{align}
Hence, for every $0<y\leq 1$ and $\beta\geq10y$,
\begin{align}\notag
\int_0^\beta &\hf\bigl(\Gamma\bigl(M\mathrm{u}_x\mathrm{a}_y,\bn\bigr)\bigr)\,dx
\\\label{MAINTERMtreatm2}
&=\beta\int_{X}f\,d\mu+O_{\ve}\bigl(\|f\|_{\C_0^4}\bigr)\cdot
\Bigl(|\beta|\min(1,\beta^2)^{\ve-\frac12}y^{\frac12-\ve}\sqrt{\scrY( M)}
+|\beta|^{\frac7{64}+\frac12}y^{\frac12-\frac7{64}}\Bigr).
\end{align}
The absolute signs around $\beta$ 
ensure that the estimate \eqref{MAINTERMtreatm2} is also
valid for all $\beta\leq-10y$.
Indeed, in this case we have
\begin{align*}
\frac1{\beta}\int_0^\beta\hf\bigl(\Gamma\bigl(M\mathrm{u}_x\mathrm{a}_y,\bn\bigr)\bigr)\,dx
=\frac1{T}\int_0^{T}\hf\bigl(\Gamma\bigl(M\mathrm{u}_{\beta}\mathrm{a}_y\mathrm{u}_t,\bn\bigr)\bigr)\,dt
\end{align*}
with $T:=|\beta|/y$, after which %
the above argument carries over,
with
\begin{align*}
U=\frac T{\scrY\bigl(M\mathrm{u}_{\beta}\mathrm{a}_{|\beta|}\bigr)}
=\frac T{\scrY\bigl(M\mathrm{a}_{|\beta|}\mathrm{u}_{-1}\bigr)}
\gg \frac{\min(1,\beta^2)}y \,\scrY(M)^{-1}.
\end{align*}
In the case $|\beta|<10y$, we will use the trivial estimate
\begin{align}\label{MAINTERMtreatm4}
\int_0^\beta &\hf\bigl(\Gamma\bigl(M\mathrm{u}_x\mathrm{a}_y,\bn\bigr)\bigr)\,dx
=\beta\int_{X}f\,d\mu+O\bigl(\|f\|_{\C_0^0}\,|\beta|\bigr).
\end{align}

Returning to the first integral on 
the right-hand side of \eqref{MAINSTEP1}, note that by 
integration by parts,
assuming $h\in\C_a^1(\R)$ with $a>2$, 
\begin{align}\label{MAINTERMtreatm5}
\int_\R \wh f\bigl(M\mathrm{u}_x\mathrm{a}_y,\bn\bigr)h(x)\,dx
=-\int_{\R}h'(\beta) \int_0^\beta &\hf\bigl(M\mathrm{u}_x\mathrm{a}_y,\bn\bigr)\,dx\,d\beta.
\end{align}
Using the estimates \eqref{MAINTERMtreatm2} and \eqref{MAINTERMtreatm4}
in \eqref{MAINTERMtreatm5}, we conclude
that for any $h\in\C_3^1(\R)$,
\begin{align}\label{MAINTERMgen}
\int_\R \wh f\bigl(M\mathrm{u}_x\mathrm{a}_y,\bn\bigr)h(x)\,dx
=\int_{X}f\,d\mu\int_{\R}h\,dx
+O\Bigl(\|f\|_{\C_0^4}\|h\|_{C_3^1}\sqrt{\scrY( M)}\, y^{\frac12-\frac7{64}}\Bigr).
\end{align}
The error term in \eqref{MAINTERMgen} is subsumed (with flying colors)
by the bound in
\eqref{MAINTHM3genres1};
this holds since 
$\delta_m(y;\vecxi)\gg y^{\frac14}\log(1+y^{-1})$
(see Lemma \ref{deltalowboundLEM})
and by the following lemma.

\begin{lem}\label{scrYconnFrobnormLEM2}
For every $M\in\G'$ we have $\scrY(M)\leq\|M\|^2$.
\end{lem}
\begin{proof}
Write $w=u+iv=M(i)\in\HH$
and let $\delta$ be the hyperbolic distance between the points $w$ and $i$.
Then
\begin{align*}
e^{\delta}\leq2\cosh(\delta)=2+\frac{|w-i|^2}{v}=\frac{u^2+v^2+1}{v}=\|M\|^2,
\end{align*}
where the last equality holds by \cite[eq.\ (42)]{effopp1}
(see also \eqref{FROBINIWASAWAEQ} below).
On the other hand,
$\scrY(M)=\tim\gamma(w)$ where $\gamma$ is an element in $\SL(2,\Z)$
for which $\gamma(w)$ lies in the standard fundamental domain for $\SL(2,\Z)$,
viz., $|\tre\gamma(w)|\leq\frac12$ and $|\gamma(w)|\geq1$.
It is known that the hyperbolic distance $\delta'$ between $\gamma(w)$ and $i$
satisfies $\delta'\leq\delta$,
and also $\tim\gamma(w)\leq e^{\delta'}$.
Hence $\scrY(M)\leq e^{\delta}\leq\|M\|^2$.
\end{proof}

\newpage

\section{Cancellation in an exponential sum}
\label{cancellationSEC}
 A crucial tool in the proof of Theorem \ref{MAINTHM3gen}
will be Proposition \ref{EXPSUMprop1} below,
which gives a bound on a linear exponential sum
running over all integer tuples $\veca=(a_1,a_2,a_3,a_4)$
satisfying $a_1a_4-a_2a_3=1$ and a congruence condition.

For any function $w\in\C^n(\R^d)$, set
\begin{align}\label{SpangendimDEF}
\|w\|_{\C_a^n}=\sum_{|\vecbeta|\leq n}\sup_{\vecx\in\R^d}\,(1+\|\vecx\|)^a\,\bigl|(\partial^\vecbeta w)(\vecx) \bigr|,
\end{align}
where the sum is taken over all multi-indices $\vecbeta\in \N^d$ 
satisfying %
$|\vecbeta|\leq n$;
also let $\C_a^n(\R^d)$ be the space of all $w\in\C^n(\R^d)$ with
$\|w\|_{\C_a^n}<\infty$.
(Note that when $d=1$, \eqref{SpangendimDEF} specializes to \eqref{Span1dimDEF}.)
Recall that we have fixed a positive integer $N$,
and that for any $R\in\SL(2,\Z)$,  $[R]$ denotes
the set of matrices 
in $\SL(2,\Z)$ which are congruent to $R$ modulo $N$. 
\begin{prop}\label{EXPSUMprop1}
Let $R\in\SL(2,\Z)$ and $B\geq1$.
For any $\vecalf=(\alpha_1,\alpha_2,\alpha_3,\alpha_4)\in\R^4$, $X\geq1$ and $w\in\C_c^5(\R^4)$ with
$\supp(w)\subset[-B,B]^4$, we have
\begin{align}\label{EXPSUMprop1res}
\sum_{\smatr{a_1}{a_2}{a_3}{a_4}\in[R]}e(\vecalf\cdot\veca)
\, w\Bigl(\frac{\veca}{X}\Bigr)
\ll_{B,N} \|w\|_{\C^5_0}\,X^{2}\sum_{1\leq q\leq X}
\frac{\tau(q)}{q^{3/2}}
\biggl(1+\frac{X\| q\vecalf\|_{\Z}}q\biggr)^{\hspace{-3pt}-1}.
\end{align}
Here $\vecalf\cdot\veca=\alpha_1a_1+\alpha_2a_2+\alpha_3a_3+\alpha_4a_4$, 
and $\tau(q)$ is the number of positive divisors of $q$.
\end{prop}

\vspace{10pt}

Before proving Proposition \ref{EXPSUMprop1},
let us note that %
it can easily be generalized to include non-compactly supported weight functions:
\begin{cor}\label{EXPSUMcor1}
Let $\ve>0$.
For any $\vecalf=(\alpha_1,\alpha_2,\alpha_3,\alpha_4)\in\R^4$, $X\geq1$ and $F\in\C_{7+\ve}^5(\R^4)$ we have
\begin{align}\label{EXPSUMcor1res}
\sum_{\smatr{a_1}{a_2}{a_3}{a_4}\in[R]}e(\vecalf\cdot\veca)
\, F\Bigl(\frac{\veca}{X}\Bigr)
\ll_{\ve,N}\|F\|_{\C_{7+\ve}^5}X^{2}\sum_{1\leq q\leq X}
\frac{\tau(q)}{q^{3/2}}
\biggl(1+\frac{X\| q\vecalf\|_{\Z}}q\biggr)^{\hspace{-3pt}-1}.
\end{align}
\end{cor}

\begin{remark}\label{EXPSUMcor1convREM}
The sum in the left-hand side of \eqref{EXPSUMcor1res} is absolutely convergent;
indeed this only requires that $F\in\C_{2+\ve}^0$.
This claim is proved e.g.\ using dyadic decomposition and the bound
\begin{align}\label{SL2Zbasiccountingbound}
\#\{\veca\in\Z^4\col a_1a_4-a_2a_3=1,\: \|\veca\|\leq Y\}\ll Y^2,
\qquad \forall Y>0.
\end{align}
\end{remark}
\begin{proof}[Proof of Corollary \ref{EXPSUMcor1}]
The corollary follows from Proposition \ref{EXPSUMprop1}
via a standard partition of unity argument.
Fix a function $\varphi\in\C^\infty_c(\R^4)$ such that
$\varphi(\vecx)=1$ for all $\vecx\in[-1,1]^4$
and $\varphi(\vecx)=0$ for all $\vecx\notin(-2,2)^4$.
Set $w_0(\vecx):=\varphi(\vecx)F(\vecx)$
and $w_k(\vecx):=\bigl(\varphi(\vecx)-\varphi(2\vecx)\bigr)F(2^k\vecx)$
for $k=1,2,3,\ldots$.
Then $w_k\in\C_c^5(\R^4)$ and $\supp(w_k)\subset[-2,2]^4$ for all $k\geq0$,
and $F(\vecx)=\sum_{k=0}^\infty w_k(2^{-k}\vecx)$ for all $\vecx\in\R^4$,
where the sum in the right-hand side is essentially finite,
since $\varphi(\vecx)-\varphi(2\vecx)=0$ for all $\vecx\in[-\frac12,\frac12]^4$.
Hence the left-hand side of \eqref{EXPSUMcor1res}
can be decomposed as
\begin{align}\label{EXPSUMcor1newpf1}
\sum_{k=0}^{\infty}\sum_{\smatr{a_1}{a_2}{a_3}{a_4}\in[R]} e(\vecalf\cdot\veca)
\, w_k\Bigl(\frac{\veca}{2^kX}\Bigr).
\end{align}
Note that $\|w_k\|_{C_0^5}\ll_{a} 2^{(5-a)k}\|F\|_{\C_a^5}$
for any $a\geq0$ and $k\in\NN$.
Hence by Proposition~\ref{EXPSUMprop1},
the above sum is
\begin{align*}
&\ll_{\ve,N} \|F\|_{\C_{7+\ve}^5}\,\sum_{k=0}^\infty 
2^{-(2+\ve)k}\, (2^kX)^{2}\sum_{1\leq q\leq 2^kX}
\frac{\tau(q)}{q^{3/2}}
\biggl(1+\frac{2^kX\| q\vecalf\|_{\Z}}q\biggr)^{\hspace{-3pt}-1}
\\
&\leq\|F\|_{\C_{7+\ve}^5}\,X^{2}\,\sum_{k=0}^\infty 
2^{-\ve k}\sum_{q=1}^{\infty}
\frac{\tau(q)}{q^{3/2}}
\biggl(1+\frac{X\| q\vecalf\|_{\Z}}q\biggr)^{\hspace{-3pt}-1}.
\end{align*}
Using $\sum_{k=0}^\infty 2^{-\ve k}\ll_{\ve}1$,
\eqref{handwrp77midimproved},
and 
$X^{-\frac12}\log(X+1)\ll\sum_{\frac12X\leq q\leq X}\tau(q)q^{-\frac32}\bigl(1+X\| q\vecalf\|_{\Z}/q\bigr)^{-1}$
(which can be proved by essentially the same argument as in Lemma~\ref{deltalowboundLEM}),
we obtain the bound
\eqref{EXPSUMcor1res}.
\end{proof}
The rest of this section will be dedicated to the proof of Proposition \ref{EXPSUMprop1}. Our key tool in proving Proposition \ref{EXPSUMprop1} will be provided by the Hardy--Littlewood--Ramanujan circle method and more precisely, the delta symbol version of it.

\subsection{The delta method activated}
For an integer $n$, let $\delta_0(n)=\begin{cases}
1 & \text{ if }n=0 \\
0 & \text{ otherwise}
\end{cases}$, be the function that detects when $n=0$. The delta symbol version of the circle method originated in the work of Duke, Friedlander and Iwaniec \cite{DFI}. It has since been reinterpreted in a form that is more readily applicable to detecting when a polynomial equation is zero, starting with the work of Heath-Brown \cite{Heath-Brown96}. We will in particular begin by recalling a more recent version of the delta method appearing in the work of Marmon and Vishe \cite[Lemma 2.2]{Marmon_Vishe}:

\begin{prop}
\label{prop:Kloos}
For any real $X\geq1$ and any integer $n$, one has
\begin{align}\label{prop:Kloosres1}
 \delta_0(n)=
 \sum_{1\leq q\leq X}
~\starsum_{a=1}^q\intr p_{q}(z)e((a/q+z)n)\,dz+O_K(X^{-K}),
\end{align}
with $K\geq0$ arbitrary. %
Here for each integer $1\leq q\leq X$, $p_q(z):=p_q^{(X)}(z)\in \C^\infty(\RR)$ is a smooth function satisfying the following decay property:
\begin{equation}|p_q(z)|\ll_{K} (1+qX|z|)^{-K},\label{eq:pbound1}
\end{equation}                                            
for any $K\geq 0$.
Also, {\footnotesize ${\displaystyle \starsum_{a=1}^q}$} denotes summation over all $1\leq a\leq q$ relatively prime to $q$.
\end{prop}
Henceforth, throughout this section, we will assume the conditions appearing in the statement of Proposition \ref{EXPSUMprop1}. Let $Q((a_1,a_2,a_3,a_4))=a_1a_4-a_2a_3-1$ be the quadratic form, with homogeneous part $Q_0(\veca)=a_1a_4-a_2a_3$,
defined by the matrix $\frac{1}{2}M_0$, where $M_0=\left(\begin{matrix}
0&0&0&1\\0&0&-1&0\\0&-1&0&0\\1&0&0&0
\end{matrix}\right)$. 
Let $\vecr=(r_1,r_2,r_3,r_4)$ be the vector in $\Z^4$ such that
$R=\matr{r_1}{r_2}{r_3}{r_4}$.
The sum on the left-hand side of \eqref{EXPSUMprop1res} now becomes
\begin{align}\label{Edef}
E:=\sum_{\veca\in \vecr+N\ZZ^4}\delta_0(Q(\veca))e(\vecalf\cdot\veca) w(\veca/X).
\end{align}
We next apply Proposition \ref{prop:Kloos} 
to detect when $Q(\veca)=0$.  %
We choose the parameter $X$ in Proposition \ref{prop:Kloos} 
to be equal to the cut-off variable $X$ appearing in \eqref{Edef}.
Using \eqref{prop:Kloosres1} with $K=4$,
and noticing that 
there are at most $O_B(X^4)$ terms in the above sum
for which $w(\veca/X)\neq0$,
we obtain
\begin{align}\label{eq:E1ex}
E=\sum_{1\leq q\leq X}
~\starsum_{a=1}^q\intr p_{q}(z)S(a/q+z) dz+O_B(\|w\|_{C_0^0}),
\end{align}
where
\begin{equation}
S(z)=\sum_{\veca\in\vecr+N\ZZ^4}w(\veca/X)e(z Q(\veca)+\vecalf\cdot\veca)
\end{equation}
is the corresponding exponential sum.
To estimate $S(\frac aq+z)$, we make a change of variables $\veca=\vecx+qN\vecy$. 
Observe that the value of $e(\frac aqQ(\vecx+qN\vecy))$ is independent of the choice of $\vecy$. Therefore, 
\begin{align*}
S\Bigl(\frac aq+z\Bigr)&=\sum_{\veca\in\vecr+N\ZZ^4}w(\veca/X)
e\Bigl(\Bigl(\frac aq+z\Bigr) Q(\veca)+\vecalf\cdot\veca\Bigr)
\\
&=\sum_{\substack{0\leq \vecx<qN\\ \vecx\equiv\vecr\mod N}}e\Bigl(\frac aqQ(\vecx)\Bigr)
\sum_{\vecy\in\ZZ^4}w\biggl(\frac{\vecx+qN\vecy}X\biggr)e\Bigl(zQ(\vecx+qN\vecy)+\vecalf\cdot (\vecx+qN\vecy)\Bigr).
\end{align*}
Here, by $0\leq \vecx<qN$, we denote that $\vecx$ is an integer vector whose entries each satisfy the condition $0\leq x_1,...,x_4<qN$. 
Applying Poisson summation to the sum over $\vecy$ gives:
\begin{align*}
&\sum_{\substack{0\leq \vecx<qN\\ \vecx\equiv\vecr\mod N}}e\Bigl(\frac aqQ(\vecx)\Bigr)
\sum_{\vecv\in\ZZ^4}\int_{\R^4} 
w\biggl(\frac{\vecx+qN\vecy}X\biggr)e\Bigl(zQ(\vecx+qN\vecy)+\vecalf\cdot (\vecx+qN\vecy)-\vecv\cdot\vecy\Bigr)\,d\vecy.
\end{align*}
Upon further setting $\vecz=\vecx+qN\vecy$, 
and therefore $\vecy=(\vecz-\vecx)/qN$, we reach
\begin{align*}
S &\Bigl(\frac aq+z\Bigr)=\frac1{(qN)^4}
\sum_{\substack{0\leq \vecx<qN\\ \vecx\equiv\vecr\mod N}}
\sum_{\vecv\in\ZZ^4}
e\Bigl(\frac{aNQ(\vecx)+\vecv\cdot\vecx}{qN}\Bigr)
\int_{\R^4}
w(\vecz/X)e\Bigl(zQ(\vecz)+\frac{qN\vecalf-\vecv}{qN}\cdot\vecz\Bigr)\,d\vecz
\\
&=\Bigl(\frac X{qN}\Bigr)^4 
\sum_{\substack{0\leq \vecx<qN\\ \vecx\equiv\vecr\mod N}}
\sum_{\vecv\in\ZZ^4}
e\Bigl(\frac{aNQ(\vecx)+\vecv\cdot\vecx}{qN}\Bigr)
\int_{\R^4} w(\vecz)
e\Bigl(zX^2Q_1(\vecz)+\frac X{qN}(qN\vecalf-\vecv)\cdot\vecz\Bigr)\,d\vecz
\\
&=\Bigl(\frac X{qN}\Bigr)^4 \sum_{\vecv\in\ZZ^4} S_{a,q,\vecr}(\vecv)\, I\biggl(zX^2,\frac X{qN}(qN\vecalf-\vecv)\biggr).
\end{align*}
Here $Q_1(\vecz)=z_1z_4-z_2z_3-X^{-2}=Q_0(\vecz)-X^{-2}$ so that $Q(X\vecz)=X^2Q_1(\vecz)$, 
and
\begin{align*}
S_{a,q,\vecr}(\vecv)&=
\sum_{\substack{0\leq \vecx<qN\\ \vecx\equiv\vecr\mod N}} e_{qN}\bigl(aNQ(\vecx)+\vecv\cdot\vecx\bigr)
\end{align*}
denotes the quadratic exponential sum (with $e_b(z):=e(z/b)$), and
\begin{align*}
I(z,\vecv)&=\int_{\R^4} w(\vecx)e(zQ_1(\vecx)+\vecv\cdot\vecx)\,d\vecx
\end{align*}
denotes the quadratic exponential integral. Substituting back in our expression for $E$ in \eqref{eq:E1ex}, we reach
\begin{align}\label{E1b}
E=\Bigl(\frac X{qN}\Bigr)^4\sum_{1\leq q\leq X}
\sum_{\vecv\in\ZZ^4} S(q,\vecv)\intr p_{q}(z)
I\biggl(zX^2,\frac X{qN}(qN\vecalf-\vecv)\biggr)\,dz+O_B(\|w\|_{C_0^0}),
\end{align}
where
\begin{equation}\label{eq:sqdef}
S(q,\vecv):=\starsum_{a\bmod{q}}S_{a,q,\vecr}(\vecv)=\starsum_{a\bmod{q}}
\sum_{\substack{0\leq \vecx<qN\\ \vecx\equiv\vecr\mod N}}
e_{qN}\bigl(aNQ(\vecx)+\vecv\cdot\vecx\bigr)
\end{equation}
denotes the complete quadratic exponential sum.

\subsection{Exponential sum bounds}
The exponential sums $S(q,\vecv)$ are well studied complete exponential sums and can be precisely evaluated for a general quadratic form $Q$, for any odd $q$ co-prime to the determinant of the matrix defining its homogeneous quadratic form, as seen in \cite[Lemma 26]{Heath-Brown96}. In the special case when $Q(\vecx)=x_1x_4-x_2x_3-1$, we may simplify the argument and obtain uniform bounds which work for all integers $q$ and any $\vecv\in\ZZ^4$.

By the Fourier inversion formula on $(\Z/N\Z)^4$,
for any $a$ with $\gcd(a,q)=1$, we have the following:
\begin{align}\notag
S_{a,q,\vecr} (\vecv) &=N^{-4}\sum_{0\leq\vecc<N}e_N(\vecr\cdot\vecc)
\sum_{\vecx\in(\Z/qN\Z)^4}e_{qN}\bigl(aN\,Q(\vecx)+\vecv\cdot\vecx\bigr) \,e_N\bigl(-\vecc\cdot\vecx\bigr).
\end{align}
Here the sum over $\vecx$ can be handled by writing
$\vecx=\vecz+q\vecy$ with $0\leq\vecz<q$ and 
$0\leq\vecy<N$. %
Note that $e_{qN}(aN\,Q(\vecz+q\vecy))
=e_q(a\,Q(\vecz+q\vecy))=e_q(a\,Q(\vecz))$, %
and for each $\vecz$ we have
\begin{align*}
\sum_{0\leq\vecy<N}e_{qN}\bigl(\vecv\cdot(\vecz+q\vecy)\bigr)e_N\bigl(-\vecc\cdot(\vecz+q\vecy)\bigr)
=e_{qN}\bigl((\vecv-q\vecc)\cdot\vecz\bigr)\sum_{0\leq\vecy<N}e_{N}\bigl((\vecv-q\vecc)\cdot\vecy\bigr),
\end{align*}
which vanishes unless $\vecv-q\vecc\equiv\bn\mod N$.
Hence we get:
\begin{align}\label{expsumdisc2}
S_{a,q,\vecr}(\vecv)=\sum_{\substack{0\leq\vecc<N\\ (q\vecc\equiv\vecv\mod N)}}
e_N(\vecr\cdot\vecc)\sum_{0\leq \vecz<q}e_q\bigl(a\,Q(\vecz)+\vecw\cdot\vecz\bigr),
\end{align}
where we write $\vecw:=N^{-1}(\vecv-q\vecc)\in\Z^4$.
Here %
\begin{align*}
a\,Q(\vecz)+\vecw\cdot\vecz=z_4(az_1+w_4)+z_3(-az_2+w_3)-a+w_1z_1+w_2z_2.
\end{align*}
For a given $\langle z_1,z_2\rangle$, 
the sum over $\langle z_3,z_4\rangle$ in \eqref{expsumdisc2} vanishes
unless $z_1\equiv -a^{-1}w_4$ and $z_2\equiv a^{-1}w_3\mod q$.
Hence
\begin{align*}
S_{a,q,\vecr}(\vecv)=q^2\sum_{\substack{0\leq\vecc<N\\ (q\vecc\equiv\vecv\mod N)}}
e_N(\vecr\cdot\vecc)e_q\bigl(-a-a^{-1}(w_1w_4-w_2w_3)\bigr),
\end{align*}
and so
\begin{align*}
S(q,\vecv)=q^2\sum_{\substack{0\leq\vecc<N\\ (q\vecc\equiv\vecv\mod N)}}
e_N(\vecr\cdot\vecc)\,\starsum_{a\bmod{q}}e_q\bigl(-a-a^{-1}(w_1w_4-w_2w_3)\bigr).
\end{align*}
Using Weil's bound for Kloosterman sums \cite[(1.60)]{IK}, we reach the following bound:
\begin{equation}
|S(q,\vecv)|\leq N^4\tau(q) q^{5/2},
\end{equation}
where $\tau(q)$ denotes the number of positive divisors of $q$. Substituting this back in \eqref{E1b}, along with the decay property of the functions $p_q$, we obtain the following bound for any $K\geq 1$:
\begin{align}\notag
|E|&\ll_{K,B} \|w\|_{C_0^0}
+X^4\sum_{1\leq q\leq X} \frac{\tau(q)}{q^{3/2}}\intr (1+qX|z|)^{-K}
\left(\sum_{\vecv\in\ZZ^4}
\biggl|I\biggl(zX^2,\frac X{qN}(qN\vecalf-\vecv)\biggr)\biggr|\right)\,dz
\\\label{E1b1}
&\ll \|w\|_{C_0^0}+X^2\sum_{1\leq q\leq X} \frac{\tau(q)}{q^{3/2}}\intr (1+q|z|/X)^{-K}
\left(\sum_{\vecv\in\ZZ^4}
\biggl|I\biggl(z,\frac X{qN}(qN\vecalf-\vecv)\biggr)\biggr|\right)\,dz.
\end{align}

\subsection{Bounding the exponential integral} %
Recall that %
\begin{align*}
I(z,\vecv)&=\int_{\R^4} w(\vecx)e(zQ_1(\vecx)+\vecv\cdot\vecx)\,d\vecx
\qquad (z\in\R,\:\vecv\in\RR^4).
\end{align*}
Let $w_0\in\C^\infty_c(\RR^4)$ be a fixed function satisfying
$\supp(w_0)\subset[0,1]^4$, $w_0\geq0$ and 
$\int_{\R^4} w_0(\vecx)\,d\vecx=1$. 
Note that for any $\delta>0$ we have 
\begin{align*}
w(\vecx)&=\delta^{-4}\int_{\R^4} w(\vecx) w_0(\vecu/\delta)\,d\vecu=\delta^{-4}\int_{\R^4} w(\vecx) w_0((\vecx-\vecu)/\delta)\,d\vecu,
\end{align*}
and thus we may write  %
\begin{align*}
I(z,\v)=\int_{\R^4} \int_{\R^4} \delta^{-4}w_0((\vecx-\vecu)/\delta)w(\vecx)e(zQ_1(\vecx)+\v\cdot\vecx)\, d\vecx \, d\vecu.
\end{align*}
Setting $\vecx=\vecu+\delta \vecy$ we obtain
\begin{align}\notag
|I(z,\v)|&=\biggl|\int_{\R^4} \int_{\R^4} w_0(\y)w(\vecu+\delta\y)e\bigl(zQ_1(\vecu+\delta\y)+\v\cdot(\vecu+\delta \y)\bigr)\,d\y \,d\vecu\biggr|
\\\label{Izvcomp1}
&\leq \int_{\R^4} \left|\int_{\R^4} w_0(\y)\,w(\vecu+\delta\y)\,e\bigl(z\delta^2Q_0(\y)\bigr)\,e\bigl(\delta(z \vecu M_0+\v)\cdot \y\bigr)\,d\y \right|\,d\vecu.
\end{align}
At this point, we choose $\delta=(1+|z|)^{-\frac12}$  %
and note that then $0<\delta\leq1$ and $|z\delta^2|\leq 1$.
Because of the factor $w_0(\vecy)\,w(\vecu+\delta\vecy)$ in the integrand,
and since $\supp(w_0)\subset[0,1]^4$ and $\supp(w)\subset[-B,B]^4$,
the integral over $\vecy$ in \eqref{Izvcomp1} can be restricted to $\vecy\in[0,1]^4$.
Similarly, the integral over $\vecu$ can then be restricted to $\vecu\in[-B_1,B_1]$,
where $B_1=B+1$.
For each $\vecu\in[-B_1,B_1]$ satisfying
$\delta\bigl\|z \vecu M_0+\v\bigr\|\geq1$, we estimate the integral over $\vecy$
by integrating by parts five times with respect to $y_j$
for a suitable $j$ and using the fact that
$\Bigl|\frac{\partial^5}{\partial y_j^5}\Bigl(w_0(\y)\,w(\vecu+\delta\y)\,e\bigl(z\delta^2Q_0(\y)\bigr)\Bigr)\Bigr|
\ll\|w\|_{\C_0^5}$ for all $\vecy\in[0,1]^4$ (since $\delta\leq1$ and $|z\delta^2|\leq1$).
For the remaining vectors 
$\vecu\in[-B_1,B_1]$, we bound the integral over $\vecy$ trivially.
This gives:
\begin{align}\label{eq:Ibound}
|I(z,\v)|
\ll \|w\|_{\C_0^5} \int_{\vecu \in [-B_1,B_1]^4} \Bigl(1+(1+|z|)^{-\frac12}\bigl\|z \vecu M_0+\v\bigr\|\Bigr)^{-5}d\vecu,
\end{align}
for any $z\in \RR$ and $\vecv\in\R^4$.

\subsection{The proof of Proposition \ref*{EXPSUMprop1}} 
\label{pfPropEXPSUMprop1SEC}
Combining \eqref{eq:Ibound} with \eqref{E1b1}, we obtain
\begin{equation}\label{eq:E1b1'}\begin{split}
   &|E|\ll_{K,B} \|w\|_{C_0^0} + \|w\|_{\C_0^5}X^2\sum_{1\leq q\leq X}\frac{\tau(q)}{q^{3/2}}\\&\times\intr \Bigl(1+\frac{q|z|}X\Bigr)^{-K}
\left(\sum_{\vecv\in\ZZ^4}
\int_{\vecu \in [-B_1,B_1]^4} \Bigl(1+(1+|z|)^{-\frac12}\Bigl\|z \vecu M_0+\frac X{qN}(qN\vecalf-\vecv)\Bigr\|\Bigr)^{-5}d\vecu\right)\,dz.\end{split}
\end{equation}
We will prove that the following bound holds
for every integer $1\leq q\leq X$ and for any sufficiently large integer $K$:
\begin{align}\notag
\intr \Bigl(1+\frac{q|z|}X\Bigr)^{-K}
\sum_{\vecv\in\ZZ^4} \int_{\vecu \in [-B_1,B_1]^4} \Bigl(1+(1+|z|)^{-\frac12}
\Bigl\|z \vecu M_0+\frac X{qN}(qN\vecalf-\vecv)\Bigr\|\Bigr)^{-5}\,d\vecu \, dz
\hspace{20pt}
\\\label{eq:1}
\ll_{B,K,N} \biggl(1+\frac{X\| Nq\vecalf\|_{\Z}}{q}\biggr)^{-1}.
\end{align}
After combining this bound
with \eqref{eq:E1b1'}
we get:
\begin{align}\label{almostEXPSUMprop1complete}
|E|
&\ll_{B,K,N}\|w\|_{C_0^0}+\|w\|_{\C^5_0}\,X^{2}\sum_{1\leq q\leq X}
\frac{\tau(q)}{q^{3/2}}
\biggl(1+\frac{X\| Nq\vecalf\|_{\Z}}q\biggr)^{\hspace{-3pt}-1}.
\end{align}
Here we have:
\begin{align*}
\sum_{1\leq q\leq X}\frac{\tau(q)}{q^{3/2}}\biggl(1+\frac{X\| Nq\vecalf\|_{\Z}}q\biggr)^{\hspace{-3pt}-1}
&\ll_N \sum_{1\leq q\leq X}\frac{\tau(Nq)}{(Nq)^{3/2}}\biggl(1+\frac{X\| Nq\vecalf\|_{\Z}}{Nq}\biggr)^{\hspace{-3pt}-1}
\\
&\leq \sum_{q=1}^{\infty}\frac{\tau(q)}{q^{3/2}}\biggl(1+\frac{X\| q\vecalf\|_{\Z}}q\biggr)^{\hspace{-3pt}-1}
\\
&\leq\sum_{1\leq q\leq X}\frac{\tau(q)}{q^{3/2}}\biggl(1+\frac{X\| q\vecalf\|_{\Z}}q\biggr)^{\hspace{-3pt}-1}
+O\Bigl(X^{-\frac12}\log(X+1)\Bigr)
\\
&\ll\sum_{1\leq q\leq X}\frac{\tau(q)}{q^{3/2}}\biggl(1+\frac{X\| q\vecalf\|_{\Z}}q\biggr)^{\hspace{-3pt}-1}.
\end{align*}
Here the inequality in the third line above holds by Lemma \ref{basicboundLEM1},
and the last bound follows from the lower bound $\sum_{\frac12X\leq q\leq X}\tau(q)q^{-\frac32}\bigl(1+X\| q\vecalf\|_{\Z}/q\bigr)^{-1}
\gg X^{-\frac12}\log(X+1)$,
which was noted in the proof of 
Corollary \ref{EXPSUMcor1}.
Therefore, in order to establish Proposition \ref{EXPSUMprop1}, it would be enough to establish \eqref{eq:1}.

To start the proof of \eqref{eq:1},
we split the integral over $z$ into two parts, 
$\{|z|\geq \frac{X}{8B_1Nq}\}$ and $\{|z|<\frac{X}{8B_1Nq}\}$, 
and estimate the corresponding contributions separately.
We begin with the case $|z|\geq \frac{X}{8B_1Nq}$.
Given any fixed $z\in\R$ satisfying $|z|\geq \frac{X}{8B_1Nq}$,
we will split the sum over $\vecv\in\Z^4$ into two parts: 
We call $\vecv$ {\em good} if 
$\frac X{Nq}\bigl\|Nq\vecalf-\vecv\bigr\|\geq 3B_1|z|$ and {\em bad} otherwise. 
Note that $\|z\vecu M_0\|\leq 2B_1|z|$ %
for any $\vecu\in[-B_1,B_1]^4$. Therefore, for every good $\vecv$, we have
$\bigl\|z \vecu M_0+\frac X{Nq}(Nq\vecalf-\vecv)\bigr\|\geq \frac X{3Nq}\bigl\|Nq\vecalf-\vecv\bigr\|$.
As a result,
\begin{align}\notag
\sum_{\substack{\vecv\in\ZZ^4\\ \vecv \textrm{ good}}}
\int_{\vecu \in [-B_1,B_1]^4} 
&\Bigl(1+(1+|z|)^{-\frac12}\Bigl\|z \vecu M_0+\frac X{Nq}(Nq\vecalf-\vecv)\Bigr\|\Bigr)^{-5}\,d\vecu
\\\label{vsumuintbound}
&\ll_{N} \frac{(1+|z|)^{\frac 52}q^5}{X^5}\int_{\vecu \in [-B_1,B_1]^4} \sum_{\substack{\vecv\in\ZZ^4\\ \vecv \textrm{ good}}}
\bigl\|Nq\vecalf-\vecv\bigr\|^{-5}d\vecu
\ll_{B,N} \frac{q^5|z|^{\frac 52}}{X^5}, %
\end{align}
where in the last bound we used the fact that
$\bigl\|Nq\vecalf-\vecv\bigr\|\geq \frac{3B_1Nq|z|}{X}\geq\frac38$
for every good $\vecv$,
and also that $|z|\gg_{B,N}1$ 
since $|z|\geq \frac{X}{8B_1Nq}$.
On the other hand, for every bad vector $\vecv$,
we estimate the integral over $\vecu$ trivially by expanding it to the whole of $\R^4$
and making a change of variables 
$\vecu_1=(1+|z|)^{-\frac12}\bigl(z \vecu M_0+\frac X{Nq}(Nq\vecalf-\vecv)\bigr)$:
\begin{align}\notag
\sum_{\substack{\vecv\in\ZZ^4\\ \vecv \textrm{ bad}}}
\int_{\vecu \in [-B_1,B_1]^4} 
\Bigl(1+(1+|z|)^{-\frac12}\Bigl\|z \vecu M_0+\frac X{Nq}(Nq\vecalf-\vecv)\Bigr\|\Bigr)^{-5}\,d\vecu
\hspace{60pt}
\\\label{ztoz1change}
\ll_B\sum_{\substack{\vecv\in\ZZ^4\\ \vecv \textrm{ bad}}}
\frac{(1+|z|)^2}{z^4}
\int_{\vecu_1 \in \R^4}(1+\|\vecu_1\|)^{-5}\,d\vecu_1
\ll_{B,N} z^{-2}\Bigl(1+\frac{q|z|}X\Bigr)^4
\ll_{B,N} \frac{q^4z^2}{X^4},
\end{align}
where we used
the fact that the number of bad vectors $\vecv\in\Z^4$ is 
$\ll \bigl(1+\frac{B_1Nq|z|}X\bigr)^4\ll_{B,N} \bigl(1+\frac{q|z|}X\bigr)^4$,
and furthermore $|z|\gg_{B,N}1$ and $\frac{q|z|}X\gg_{B,N}1$,
since $|z|\geq \frac{X}{8B_1Nq}$.
Combining the bounds for both good and bad vectors $\vecv$,
we obtain the following, for any $K\geq4$:
\begin{align}\notag
\int_{|z|\geq \frac{X}{8B_1Nq}} 
\Bigl(1+\frac{q|z|}X\Bigr)^{-K}
\sum_{\vecv\in\ZZ^4} \int_{\vecu \in [-B_1,B_1]^4} \Bigl(1+(1+|z|)^{-\frac12}
\Bigl\|z \vecu M_0+\frac X{Nq}(Nq\vecalf-\vecv)\Bigr\|\Bigr)^{-5}\,d\vecu \, dz
\\ \notag
\ll_{B,N}\int_{|z|\geq \frac{X}{8B_1Nq}} \Bigl(1+\frac{q|z|}X\Bigr)^{-K}
\biggl(\frac{q^5|z|^{\frac 52}}{X^5}+ \frac{q^4z^2}{X^4}\biggr)\,dz
\hspace{150pt}
\\
\ll \frac{X}{q}\int_{|t|\geq \frac1{8B_1N}}(1+|t|)^{-K}
\biggl(\frac{q^{\frac52}|t|^{\frac52}}{X^{\frac52}}+ \frac{q^2|t|^2}{X^2}\biggr)\,dt
\ll \frac{q}{X}
\ll \biggl(1+\frac {X\|Nq\vecalf\|_{\Z}}q\biggr)^{-1}.    %
\end{align}
Hence we have proved that the contribution from $\{|z|\geq \frac{X}{8B_1Nq}\}$ in \eqref{eq:1}
indeed satisfies the stated bound.

\vspace{5pt}

We next turn to the remaining part $\bigl\{|z|<\frac{X}{8B_1Nq}\bigr\}$.  
For $z$ in this range, $\|z\vecu M_0\|<\frac X{4Nq}$
for any $\vecu\in [-B_1,B_1]^4$.
Choose $\vecv_0\in\Z^4$ so that the minimum of 
$\bigl\|Nq\vecalf-\vecv\bigr\|$ over $\vecv\in\Z^4$ is 
attained for $\vecv=\vecv_0$.
Then for every $\vecv\in\Z^4\setminus\{\vecv_0\}$
we have $\bigl\|Nq\vecalf-\vecv\bigr\|\geq  \frac12$,
and thus
$\Bigl\|z \vecu M_0+\frac X{Nq}(Nq\vecalf-\vecv)\Bigr\|
\geq \frac X{2Nq}\bigl\|Nq\vecalf-\vecv\bigr\|$
for all $\vecu\in [-B_1,B_1]^4$.
Therefore, for any fixed $z$ with $|z|<\frac{X}{8B_1Nq}$,
\begin{align}\notag
\sum_{\substack{\vecv\in\ZZ^4\\ \vecv\neq \vecv_0}}
\int_{\vecu \in [-B_1,B_1]^4} \Bigl(1+(1+|z|)^{-\frac12}
\Bigl\|z \vecu M_0+\frac X{Nq}(Nq\vecalf-\vecv)\Bigr\|\Bigr)^{-5}\,d\vecu 
\hspace{50pt}
\\\notag
\ll_{B}\sum_{\substack{\vecv\in\ZZ^4\\ \vecv\neq \vecv_0}} 
(1+|z|)^{\frac52} \biggl\|\frac X{Nq}(Nq\vecalf-\vecv)\biggr\|^{-5}
\ll_N\frac{q^5}{X^5}(1+|z|)^{\frac52}.
\end{align}
The contribution from these terms to \eqref{eq:1} is
\begin{align*}
\ll_{B,N} \int_{|z|<\frac{X}{8B_1Nq}}
\frac{q^5}{X^5}(1+|z|)^{\frac52}\,dz
\ll \frac{q^{3/2}}{X^{3/2}}
\ll \biggl(1+\frac {X\|Nq\vecalf\|_{\Z}}q\biggr)^{-1},
\end{align*}
i.e.\ this contribution also satisfies the bound
claimed in \eqref{eq:1}.

It now remains to estimate
\begin{equation}\label{remainstoestimate}
\int_{|z|<\frac{X}{8B_1Nq}}\int_{\vecu \in [-B_1,B_1]^4}
\Bigl(1+(1+|z|)^{-\frac12}\Bigl\|z \vecu M_0+\frac X{Nq}(Nq\vecalf-\vecv_0)\Bigr\|\Bigr)^{-5}\, d\vecu \, dz.
\end{equation}
We estimate the inner integral trivially when $|z|\leq1$,
while for $|z|>1$ we expand the inner integral to the whole of $\R^4$
and make the same variable change $\vecu\leftrightarrow\vecu_1$ as in \eqref{ztoz1change}.
Then we obtain that \eqref{remainstoestimate} is
\begin{align}\notag
&\ll B_1^4+\int_{|z|>1}z^{-2}\int_{\vecu_1\in\R^4}(1+\|\vecu_1\|)^{-5}\,d\vecu_1
\ll_B 1.
\end{align}
This suffices for our goal of proving the bound in \eqref{eq:1} 
if $\frac{X\| Nq\vecalf\|_{\Z}}{q}\leq 1$. 
Hence, we will now assume that $\frac{X\| Nq\vecalf\|_{\Z}}{q}>1$. 
Let us expand the integral over $z$ in \eqref{remainstoestimate}
to the whole of $\R$, and then split it into the two parts
$\{|z|\leq\frac{X\| Nq\vecalf\|_{\Z}}{4B_1Nq}\}$ and 
$\{|z|>\frac{X\| Nq\vecalf\|_{\Z}}{4B_1Nq}\}$.
In the first range,
since $\|\vecu M_0\|\leq2B_1$ for all $\vecu\in[-B_1,B_1]^4$,
we have $\bigl\|z \vecu M_0\bigr\|\leq\frac X{2Nq}\|Nq\vecalf\|_{\Z}
=\frac X{2Nq}\bigl\|Nq\vecalf-\vecv_0\bigr\|$,
and so the contribution to \eqref{remainstoestimate} from these $z$ is
\begin{align*}
&\ll\int_{|z|\leq\frac{X\| Nq\vecalf\|_{\Z}}{4B_1Nq}}\int_{\vecu \in [-B_1,B_1]^4}
\biggl((1+|z|)^{-\frac12} \frac X{Nq}\bigl\|Nq\vecalf\bigr\|_{\Z}\biggr)^{-5}\,d\vecu \, dz
\\
&\ll_{B,N} 
\biggl(\frac Xq\|Nq\vecalf\|_{\Z}\biggr)^{-5}
\int_{|z|\leq\frac{X\| Nq\vecalf\|_{\Z}}{4B_1Nq}}
(1+|z|)^{\frac 52}\, dz
\ll_{B,N}\biggl(\frac Xq\|Nq\vecalf\|_{\Z}\biggr)^{-\frac 32},
\end{align*}
where in the last bound we used the fact that 
$\frac{X\| Nq\vecalf\|_{\Z}}{4B_1Nq}\gg_{B,N}1$,
since we are currently assuming
$\frac{X\|Nq\vecalf\|_{\Z}}{q}>1$.
In the remaining range, we again estimate the integral over $\vecu$ by expanding it to whole of $\RR^4$
and making
the same variable change $\vecu\leftrightarrow\vecu_1$ as in \eqref{ztoz1change};
this gives that the contribution to \eqref{remainstoestimate} from these $z$ is
\begin{align*}
\ll_{B,N}\int_{|z|>\frac{X\| Nq\vecalf\|_{\Z}}{4B_1Nq}} z^{-2}
\int_{\RR^4}(1+\|\vecu_1\|)^{-5}\,d\vecu_1\, dz
\ll_{B,N} &\left(\frac{X\| Nq\vecalf\|_{\Z}}{q}\right)^{-1}
\\
&\ll \left(1+\frac{X\| Nq\vecalf\|_{\Z}}{q}\right)^{-1},
\end{align*}
where the last bound holds since 
$\frac{X\|Nq\vecalf\|_{\Z}}{q}>1$.
This completes the proof of the bound in \eqref{eq:1},
and hence also the proof of Proposition \ref{EXPSUMprop1}.
\hfill$\square$

\newpage

\section{The contribution from $A_k$- and $B_k$-orbits}
\label{midapproachSEC}
\label{BKSEC}

In this section we will establish cancellation in each of the sums
appearing in the last two lines of \eqref{MAINSTEP1},
thereby allowing us to conclude
the proof of Theorem \ref{MAINTHM3gen}.
The crucial ingredient will be the exponential sum bound from
the previous section, %
Corollary \ref{EXPSUMcor1}.
However, non-trivial work will be needed to 
put the sums in a format amenable to application of 
Corollary~\ref{EXPSUMcor1}, and to further control the 
derivatives appearing in the bound %
which the corollary provides.

\subsection{Bounding a sum involving a general test function on $\SL(2,\R)$}

We will start by proving a bound on a sum of the form
\begin{align}
\sum_{\smatr abcd\in[R]}
e(\alpha_1a+\alpha_2b+\alpha_3c+\alpha_4d)
\int_{-\infty}^{\infty} \phi\left(\matr abcd M\mathrm{u}_x \mathrm{a}_y \right)h(x)\,dx,
\end{align}
where $\phi$ belongs to an appropriate class of test functions on $\SL(2,\R)$.
Note that a sum of this form appears as the innermost sum in the last line of 
\eqref{MAINSTEP1}.

Let $X_1,X_2,X_3$ be the following elements in the Lie algebra $\lsl(2,\R)$:
\begin{align}\label{X1X2X3def}
X_1=\matr 0100,\qquad
X_2=\matr0010,\qquad
X_3=\matr100{-1}.
\end{align}
These elements are also viewed as left invariant vector fields 
on $\G'=\SL(2,\R)$.
For any $r\in\R_{\geq0}$ and $m\in\Z_{\geq0}$,
let us introduce the following norm on functions $\phi\in\C^m(\G')$,
somewhat analogous to the norm in \eqref{CMNORMDEF}:
\begin{align}\label{CrmnormonSL2def}
\|\phi\|_{\C_r^m}:=\sum_{\ord(D)\leq m}\,\sup_{T\in \G'}\,\|T\|^r\bigl|(D\phi)(T)\bigr|. %
\end{align}
Here the sum is taken over all
monomials $D$ in $X_1,X_2,X_3$ of degree $\leq m$,
and $\|T\|$ is the Frobenius matrix norm
(see \eqref{Frobnorm}).
Also let us write $\C_r^m(\G')$ for the space of all functions 
$\phi\in\C^m(\G')$ with $\|\phi\|_{\C_r^m}<\infty$.

\begin{prop}\label{AkBkgenPROPMgen}
For any $\phi\in \C_{311}^5(\G')$,
$h\in\C_{43}^{5}(\R)$, %
$M\in\G'$, $\vecalf\in\R^4$ and $0<y\leq1$,
\begin{align}\notag
\sum_{\smatr abcd\in[R]}
e(\alpha_1a+\alpha_2b+\alpha_3c+\alpha_4d)
\int_{-\infty}^{\infty} \phi\left(\matr abcd M\mathrm{u}_x \mathrm{a}_y \right)h(x)\,dx
\hspace{60pt}
\\\label{AkBkgenPROPMgenres}
\ll_N \|\phi\|_{\C_{311}^5}\, \|h\|_{\C_{43}^5}\,\|M\|^{13}
\, \sum_{1\leq q\leq y^{-1/2}}\frac{\tau(q)}{q^{3/2}}
\biggl(1+\frac{\| q\vecalf\|_{\Z}}{q\sqrt y}\biggr)^{\hspace{-3pt}-1}.
\end{align}
\end{prop}

To prepare for the proof of Proposition \ref{AkBkgenPROPMgen},
for any real numbers $u,v,s$ with $(u,v)\neq(0,0)$, 
we introduce the short-hand notation 
$[u,v,s]$
to denote the matrix
\begin{align*}
[u,v,s]:=\begin{pmatrix}u & -v/(u^2+v^2) %
\\[3pt] v & u/(u^2+v^2) %
\end{pmatrix}\matr 1s01 \in\G'.
\end{align*}
Note that the map $(u,v,s)\mapsto[u,v,s]$
is a diffeomorphism of $(\R^2\setminus\{\bn\})\times\R$ onto $\G'$. For any matrix $\smatr abcd$ in $\G'$, we have
\begin{align*}
\matr abcd = \left[ a,c,\:\frac{ab+cd}{a^2+c^2}\right].
\end{align*}
Using $(u,v,s)$ as coordinates, %
$\frac{\partial}{\partial u}$
and $\frac{\partial}{\partial v}$ and $\frac{\partial}{\partial s}$
are vector fields on $\G'$.
It is straightforward to check the following relations:
\begin{align}\label{Lieder1}
\frac{\partial}{\partial s}=X_1;
\qquad
\frac{\partial}{\partial u}=
\biggl(\frac{v}{(u^2+v^2)^2}+\frac{2us}{u^2+v^2}+vs^2\biggr)X_1
-vX_2+\biggl(\frac{u}{u^2+v^2}+vs\biggr)X_3,
\end{align}
and
\begin{align}\label{Lieder2}
\frac{\partial}{\partial v}=
\biggl(\frac{-u}{(u^2+v^2)^2}+\frac{2vs}{u^2+v^2}-us^2\biggr)X_1
+uX_2+\biggl(\frac{v}{u^2+v^2}-us\biggr)X_3.
\end{align}

\vspace{5pt}

 Let $\phi$ and $h$ be given as in the statement of Proposition \ref{AkBkgenPROPMgen}.
Let us fix a function $\omega\in\C_c^{\infty}(\R)$ such that $\omega(t)=1$ for all $0\leq t\leq1$
and $\omega(t)=0$ whenever $|t|\geq2$.
For given $0<y\leq1$, we define the function
$F:\R^4\to\CC$ by
$F(\vecx)=0$ if $(x_1,x_3)=\bn$, and otherwise
\begin{align}\label{midapproach1}
F(\vecx)=y\,\omega\bigl(x_1x_4-x_2x_3\bigr)\int_{-\infty}^{\infty} \phi\bigl(\big[ x_1,x_3,s\big]\bigr)\,
h\biggl(ys-\frac{x_1x_2+x_3x_4}{x_1^2+x_3^2}\biggr)\,ds.
\end{align}
(We will see below that the integral is absolutely convergent.)
Let us note that for every $\matr abcd\in\G'$, we have
$y(ad-bc)=y\in(0,1]$ and therefore $\omega\bigl(y(ad-bc)\bigr)=1$.
As a result,
\begin{align}\label{midapproach1a}
F\bigl(\sqrt ya,\sqrt yb,\sqrt yc,\sqrt yd\bigr)
&=y\int_{-\infty}^{\infty} \phi\bigl(\big[\sqrt y a,\sqrt y c,s\big]\bigr)
h\biggl(ys-\frac{ab+cd}{a^2+c^2}\biggr)\,ds
\hspace{40pt}
\\\notag
&=\int_{-\infty}^{\infty} \phi\left( %
\matr abcd\matr{\sqrt y}{x/\sqrt y}0{1/\sqrt y}\right)h(x)\,dx
\hspace{20pt}
\\\notag
&=\int_{-\infty}^{\infty} \phi\left( %
\matr abcd \mathrm{u}_x \mathrm{a}_y \right)h(x)\,dx,
\end{align}
where we substituted 
$s=y^{-1}\bigl(\frac{ab+cd}{a^2+c^2}+x\bigr)$.
In particular, it follows that in the special case $M=1_2$, the left-hand side of
\eqref{AkBkgenPROPMgenres} equals
\begin{align}\label{midapproach2}
\sum_{\smatr abcd\in[R]} e\bigl(\alpha_1 a+\alpha_2 b+\alpha_3c+\alpha_4d\bigr)
F\bigl(\sqrt ya,\sqrt yb,\sqrt yc,\sqrt yd\bigr).
\end{align}
We will start by proving Proposition \ref{AkBkgenPROPMgen} in the special case $M=1_2$,
and later we will extend the proof to the case of general $M$.

\begin{lem}\label{midapproachLEM1}
Let $n\in\NN$ and $a\in\R_{\geq0}$,
and set $r=53n+5a+6$ and $r'=7n+a$.
Then for any 
$\phi\in\C_r^{n}(\G')$, $h\in\C_{r'}^n(\R)$ %
and $0<y\leq1$,
the function $F$ defined in
\eqref{midapproach1} 
lies in $\C_a^n(\R^4)$ and satisfies
\begin{align}\label{midapproachLEM1res}
\|F\|_{\C_a^n}\ll_{n,a}
\|\phi\|_{\C_r^n} \|h\|_{\C_{r'}^n} \cdot y.
\end{align}
\end{lem}

\begin{proof}
From \eqref{midapproach1},
by differentiating under the integration sign\footnote{The fact that this is permitted is justified in a standard manner,
using the bounds and relations 
\eqref{derboundrep}, \eqref{matrnormformula1} and \eqref{PDjellbound}
below.}
and using \eqref{Lieder1} and \eqref{Lieder2}, it follows
that for every multi-index
$\vecalf\in\N^4$ and all $\vecx\in\R^4$ with $(x_1,x_3)\neq\bn$,
we have
\begin{align}\notag
\partial_{\vecx}^{\vecalf}F(\vecx)
=y\,\omega^{(\ell)}  & \bigl(x_1x_4-x_2x_3\bigr)\cdot
\,(x_1^2+x_3^2)^{-2|\vecalf|}
\\\label{midapproach3}
&\times \sum_{D,j,\ell}\int_{-\infty}^{\infty} P_{D,j,\ell}^{\vecalf}(\vecx,s)\cdot
\bigl[D\phi\bigr]\bigl(\big[ x_1,x_3,s\big]\bigr)
h^{(j)}\biggl(ys-\frac{x_1x_2+x_3x_4}{x_1^2+x_3^2}\biggr)\,ds.
\end{align}
In the above sum,
$D$ runs over all monomials in $X_1,X_2,X_3$ of degree $\leq |\vecalf|$,
and $j$ and $\ell$ run over non-negative integers with $j+\ell\leq|\vecalf|-\deg D$,
and each $P_{D,j,\ell}^{\vecalf}(\vecx,s)$ is a polynomial 
of degree $\leq7|\vecalf|$
in the variables
$x_1,x_2,x_3,x_4,s$.

By the definition of the norm $\|\phi\|_{\C_r^m}$, we have
\begin{align}\label{derboundrep}
\bigl|\bigl[D\phi\bigr](T)\bigr|\leq\frac{\|\phi\|_{\C_r^m}}{\|T\|^{r}}
\end{align}
for any monomial $D$ in $X_1,X_2,X_3$ of degree $\leq m$
and any $T\in\G'$.
Here if $T=[x_1,x_3,s]$, we have
\begin{align}\label{matrnormformula1}
\|T\|^2
=(x_1^2+x_3^2)(1+s^2)+\frac1{x_1^2+x_3^2}.
\end{align}
We also have 
\begin{align}\label{PDjellbound}
\bigl|P_{D,j,\ell}^{\vecalf}(\vecx,s)\bigr|\ll_{\vecalf} (1+\|\vecx\|+|s|)^{7|\vecalf|}
\leq (1+\|\vecx\|)^{7|\vecalf|}(1+|s|)^{7|\vecalf|},
\end{align}
for all tuples of $D,j,\ell$ appearing in \eqref{midapproach3}.
Hence it follows from \eqref{midapproach3}
that for any $n\in\NN$, $r,r'\in\R_{\geq0}$,
any multi-index
$\vecalf\in\N^4$ with $|\vecalf|\leq n$,
and all $\vecx\in\R^4$ with $(x_1,x_3)\neq\bn$,
\begin{align}\label{midapproach5}
&\bigl|\partial_{\vecx}^{\vecalf}F(\vecx)\bigr|
\ll_{n}
\|\phi\|_{\C_r^{n}} \|h\|_{\C_{r'}^n}
\cdot y\cdot
\frac{(1+\|\vecx\|)^{7n}}{(x_1^2+x_3^2)^{2|\vecalf|}}\,
I\bigl(|x_1x_4-x_2x_3|\leq2\bigr)
\\\notag
&\hspace{12pt}\times\int_{-\infty}^{\infty} \bigl(1+|s|\bigr)^{7n}
\biggl((x_1^2+x_3^2)(1+s^2)+\frac1{x_1^2+x_3^2}\biggr)^{\hspace{-4pt}-r/2}
\biggl(1+\biggl|ys-\frac{x_1x_2+x_3x_4}{x_1^2+x_3^2}\biggr|\biggr)^{\hspace{-4pt}-r'}
\,ds.
\end{align}
Also note that $\partial_{\vecx}^{\vecalf}F(\vecx)=0$ whenever 
$|x_1x_4-x_2x_3|>2$, since the support of $\omega$ is contained in $[-2,2]$.
In order to bound the expression in \eqref{midapproach5} and complete the proof of 
Lemma \ref{midapproachLEM1},
we will make use of the following auxiliary lemma.

\begin{lem}\label{MAJORANTINEQLEM1}
Let $k_1,k_2,k_3\in\R_{\geq0}$,  %
and set $r=5k_1+2k_2+2k_3$ and $r'=k_1$.
Then 
for every $0<y\leq1$, $s\in\R$ and every $\vecx\in\R^4$ satisfying
$|x_1x_4-x_2x_3|\leq2$ and $(x_1,x_3)\neq\bn$,
we have
\begin{align}\notag
\biggl((x_1^2+x_3^2)(1+s^2)+\frac1{x_1^2+x_3^2}\biggr)^{\hspace{-4pt}-r/2}
\biggl(1+\biggl|ys-\frac{x_1x_2+x_3x_4}{x_1^2+x_3^2}\biggr|\biggr)^{\hspace{-4pt}-r'}
\hspace{60pt}
\\\label{MAJORANTINEQLEM1res}
\ll (1+\|\vecx\|)^{-k_1}\biggl(1+\frac1{x_1^2+x_3^2}\biggr)^{\hspace{-4pt}-k_2}\hspace{3pt}(1+|s|)^{-k_3},
\end{align}
where the implied constant depends only on $k_1,k_2,k_3$.
\end{lem}
\begin{proof}
Set $\vecz:=(x_1,x_3)$ and $\vecw:=(x_2,x_4)$.
Let $\vece=\|\vecz\|^{-1}\vecz$, let $\vecf$ be a unit vector orthogonal to $\vece$ in $\R^2$,
and let $a,b\in\R$ be such that $\vecw=a\vece+b\vecf$.
The assumption $|x_1x_4-x_2x_3|\leq2$ means that the 
parallelogram spanned by $\vecz$ and $\vecw$ has area $\leq2$,
that is, $|b|\cdot\|\vecz\|\leq2$.
Using this and $\vecz\cdot\vecw=a\|\vecz\|$,
it follows that
\begin{align}\notag
1+\|\vecx\|^2
&=1+\|\vecz\|^2+\|\vecw\|^2=1+\|\vecz\|^2+a^2+b^2
\leq 1+\|\vecz\|^2+\frac{|\vecz\cdot\vecw|^2+4}{\|\vecz\|^2}
\\\label{MAJORANTINEQLEM1pf1}
&\ll\bigl(1+\|\vecz\|^2+\|\vecz\|^{-2}\bigr)\bigl(1+|\vecz\cdot\vecw|^2\bigr)
\ll\bigl(\|\vecz\|^2+\|\vecz\|^{-2}\bigr)\bigl(1+|\vecz\cdot\vecw|\bigr)^2.
\end{align}
Hence: %
\begin{align}\notag
(1+\|\vecx\|)^{k_1} &\bigl(1+\|\vecz\|^{-2}\bigr)^{k_2}(1+|s|)^{k_3}
\\\label{MAJORANTINEQLEM1pf4}
&\ll \bigl(\|\vecz\|^2+\|\vecz\|^{-2}\bigr)^{\frac12 k_1}\bigl(1+|\vecz\cdot\vecw|\bigr)^{k_1}
\bigl(1+\|\vecz\|^{-2}\bigr)^{k_2}(1+|s|)^{k_3}
\\\notag
&\ll \bigl(\|\vecz\|^2+\|\vecz\|^{-2}\bigr)^{\frac12 k_1+k_2}\bigl(1+|\vecz\cdot\vecw|\bigr)^{k_1}(1+|s|)^{k_3},
\end{align}
where from now on in this proof we allow the implied constant in any ``$\ll$'' to depend on $k_1,k_2,k_3$,
but on no other variables.

Next set
\begin{align*}
A:=\frac{\vecz\cdot\vecw}{\|\vecz\|^2}=\frac{x_1x_2+x_3x_4}{x_1^2+x_3^2},
\end{align*}
and note that
\begin{align}\label{MAJORANTINEQLEM1pf2}
1+|\vecz\cdot\vecw|
\leq (1+|A|)\bigl(1+\|\vecz\|^2\bigr)
\ll (1+|A|)\bigl(\|\vecz\|^2+\|\vecz\|^{-2}\bigr).
\end{align}
Furthermore, 
\begin{align}\label{MAJORANTINEQLEM1pf3}
1+|A|\leq 1+|ys-A|+|s|\leq(1+|ys-A|)(1+|s|).
\end{align}
Using \eqref{MAJORANTINEQLEM1pf2} and \eqref{MAJORANTINEQLEM1pf3},
it follows that the expression in \eqref{MAJORANTINEQLEM1pf4} is
\begin{align}\label{MAJORANTINEQLEM1pf5}
\ll\bigl(\|\vecz\|^2+\|\vecz\|^{-2}\bigr)^{\frac32 k_1+k_2}\bigl(1+|ys-A|\bigr)^{k_1}(1+|s|)^{k_1+k_3}.
\end{align}
Finally, $\bigl(\|\vecz\|\sqrt{1+s^2}-\|\vecz\|^{-1}\bigr)^2\geq0$, 
and thus
\begin{align}
1+|s|<2\sqrt{1+s^2}\leq \|\vecz\|^2(1+s^2)+\|\vecz\|^{-2}.
\end{align}
Hence the expression in \eqref{MAJORANTINEQLEM1pf5} is
\begin{align}\label{MAJORANTINEQLEM1pf6}
\ll\bigl(\|\vecz\|^2(1+s^2)+\|\vecz\|^{-2}\bigr)^{\frac52 k_1+k_2+k_3}\bigl(1+|ys-A|\bigr)^{k_1}.
\end{align}
Now \eqref{MAJORANTINEQLEM1res} follows,
with $r=5k_1+2k_2+2k_3$ and $r'=k_1$,
from the chain of inequalities in
\eqref{MAJORANTINEQLEM1pf4},
\eqref{MAJORANTINEQLEM1pf5}
and \eqref{MAJORANTINEQLEM1pf6}.
\end{proof}

We can now conclude the proof of Lemma \ref{midapproachLEM1}:
Recall \eqref{midapproach5} 
and the fact that
$\partial_{\vecx}^{\vecalf}F(\vecx)=0$ whenever 
$|x_1x_4-x_2x_3|>2$.
Applying Lemma \ref{MAJORANTINEQLEM1} 
with $k_1=7n+a$, $k_2=2n+1$ and $k_3=7n+2$
(thus $r=53n+5a+6$ and $r'=7n+a$),
we obtain:
\begin{align}\label{midapproach14}
\bigl|\partial_{\vecx}^{\vecalf}F(\vecx)\bigr|\ll_{n,a}
\|\phi\|_{\C_r^n} \|h\|_{\C_{r'}^n}
\cdot y\cdot (1+\|\vecx\|)^{-a}
\biggl(1+\frac1{x_1^2+x_3^2}\biggr)^{\hspace{-4pt}-1}
\end{align}
for all $\vecx\in\R^4$ with $(x_1,x_3)\neq\bn$,
and all $\vecalf$ with $|\vecalf|\leq n$.

Recall also that, by definition, $F(\vecx)=0$ 
whenever $(x_1,x_3)=\bn$.
Note that it follows from \eqref{midapproach14}
that $\partial_{\vecx}^{\vecalf}F(\vecx)\to0$ 
as $(x_1,x_3)\to\bn$, 
in fact uniformly with respect to all $(x_2,x_4)\in\R^2$.
This holds for every $\vecalf$ with $|\vecalf|\leq n$;
hence we conclude that 
$F\in\C^n(\R^4)$,
with $\partial_{\vecx}^{\vecalf}F(\vecx)=0$
whenever $(x_1,x_3)=\bn$ and $|\vecalf|\leq n$.
Finally, the bound \eqref{midapproachLEM1res} now follows from
\eqref{midapproach14}.
This completes the proof of Lemma \ref{midapproachLEM1}.
\end{proof}

The proof of 
Proposition~\ref{AkBkgenPROPMgen}
in the special case $M=1_2$
is now immediate,
using Corollary~\ref{EXPSUMcor1}
and Lemma \ref{midapproachLEM1}.
Indeed, given $\phi$ and $h$ as in 
Proposition~\ref{AkBkgenPROPMgen},
Lemma \ref{midapproachLEM1} implies that 
\begin{align}\label{AkBkgenPROPpf1}
\|F\|_{\C_8^5}\ll\|\phi\|_{\C^5_{311}} \|h\|_{\C_{43}^5}\cdot y.
\end{align}
Recall also that the left-hand side of
\eqref{AkBkgenPROPMgenres} for $M=1_2$ equals
\eqref{midapproach2}.
Hence Proposition~\ref{AkBkgenPROPMgen} 
in the special case $M=1_2$
now follows from 
Corollary \ref{EXPSUMcor1}.

\vspace{5pt}

Finally we extend the proof of 
Proposition~\ref{AkBkgenPROPMgen} to the case of general $M\in\G'$.
Let $R_M$ be the block matrix $R_M=\matr{\trans M}00{\trans M}\in\GL(4,\R)$,
and note that then
for every $\vecx\in\R^4$ we have
\begin{align}\label{xMdef}
\matr{x_1}{x_2}{x_3}{x_4} M=\matr{x_{M,1}}{x_{M,2}}{x_{M,3}}{x_{M,4}},
\qquad \text{where }\: \vecx_M=\begin{pmatrix}x_{M,1}
\\ x_{M,2} \\ x_{M,3} \\ x_{M,4} \end{pmatrix}
:=R_M\cdot\vecx.
\end{align}
We next define
the function $F_M:\R^4\to\CC$ by
\begin{align*}
F_M(\vecx):=F(R_M\cdot\vecx).
\end{align*}
Then using \eqref{midapproach1a}, the left-hand side of
\eqref{AkBkgenPROPMgenres} equals
\begin{align}\label{midapproachMgen2}
\sum_{\smatr abcd\in[R]}
e(\alpha_1a+\alpha_2b+\alpha_3c+\alpha_4d)
\, F_M\bigl(\sqrt ya,\sqrt yb,\sqrt yc,\sqrt yd\bigr).
\end{align}

\begin{lem}\label{FMconnecttoFLEM}
For any $a\in\R_{\geq0}$ and $n\in\Z_{\geq0}$ we have
\begin{align}\label{FMconnecttoFLEMres}
\|F_M\|_{\C_a^n}\ll_{n,a} \|M\|^{n+a}\|F\|_{\C_a^n}.
\end{align}
\end{lem}
\begin{proof}
First let us note that for any $j\in\{1,2,3,4\}$,
with notation as in \eqref{xMdef}, we have:
\begin{align*}
\partial_{x_j}F_M(\vecx)=
\partial_{x_j}F(R_M\vecx)
&=\sum_{\ell=1}^4 (\partial_{\ell}F)(R_M\vecx)\cdot\partial_{x_j} x_{M,\ell}
=\sum_{\ell=1}^4 R_{M,\ell,j}\cdot (\partial_{\ell}F)(R_M\vecx),
\end{align*}
where $R_{M,\ell,j}$ is the $(\ell,j)$th entry of $R_M$,
and $\partial_{\ell}F$ denotes the partial derivative of $F(\vecx)$ with respect to $x_{\ell}$.
Iterating the last computation, it follows that for any $r\geq1$ and any $j_1,\ldots,j_r\in\{1,2,3,4\}$,
\begin{align*}
\partial_{x_{j_1}}\cdots \partial_{x_{j_r}} F_M(\vecx)
=\sum_{\vecell\in\{1,2,3,4\}^r}\biggl(\prod_{i=1}^r R_{M,\ell_i,j_i}\biggr)\cdot
(\partial_{\ell_1}\cdots\partial_{\ell_r}F)(R_M\vecx).
\end{align*}
Here for any $a\geq0$ and $n\geq r$ %
we have
$\bigl|(\partial_{\ell_1}\cdots\partial_{\ell_r}F)(R_M\vecx)\bigr|
\leq \|F\|_{\C_a^n}\cdot (1+\|R_M\vecx\|)^{-a}$,
by \eqref{SpangendimDEF}.
Hence we conclude that for any multi-index $\vecalf$ with $|\vecalf|\leq n$,
\begin{align*}
\bigl|\partial_{\vecx}^{\vecalf} F_M(\vecx)\bigr|
\ll_n \bigl(\sup_{\ell,j}|R_{M,\ell,j}|\bigr)^{|\vecalf|} \,
\|F\|_{\C_a^n}\, (1+\|R_M\vecx\|)^{-a}.
\end{align*}
Multiplying by $(1+\|\vecx\|)^a$, taking the supremum over $\vecx$
and adding over all $\vecalf$, we obtain
(again recall \eqref{SpangendimDEF}):
\begin{align*}
\|F_M\|_{\C_a^n}\ll_n \|F\|_{\C_a^n}\bigl(\sup_{\ell,j}|R_{M,\ell,j}|\bigr)^n
\cdot\sup_{\vecx\in\R^4}\biggl(\frac{1+\|\vecx\|}{1+\|R_M\vecx\|}\biggr)^{a}.
\end{align*}
Here
\begin{align}\label{FMconnecttoFLEMpf1}
\sup_{\vecx\in\R^4}\frac{1+\|\vecx\|}{1+\|R_M\vecx\|}
=\sup_{\vecy\in\R^4}\frac{1+\|R_M^{-1}\vecy\|}{1+\|\vecy\|}
= \max\Bigl(1,\sup_{\vecy\in S^3}\|R_M^{-1}\vecy\|\Bigr),
\end{align}
where $S^3$ denotes the unit sphere $\{\vecy\in\R^4\col\|\vecy\|=1\}$.
Using the definition of $R_M$ we have,
writing $M=\smatr{a'}{b'}{c'}{d'}$:
\begin{align*}
\sup_{\vecy\in S^3}\|R_M^{-1}\vecy\|
\asymp\max(|a'|,|b'|,|c'|,|d'|)\asymp\|M\|,
\end{align*}
and also $\sup_{\ell,j}|R_{M,\ell,j}|\asymp\|M\|$.
Recall also that $\|M\|\geq\sqrt2$ always.
Hence \eqref{FMconnecttoFLEMres} follows.
\end{proof}

Combining Lemma \ref{FMconnecttoFLEM} %
and Lemma \ref{midapproachLEM1},
we conclude that
\begin{align}\label{AkBkgenPROPpf1Mgen}
\|F_M\|_{\C_8^5}\ll\|M\|^{13}\|\phi\|_{\C^5_{311}} \|h\|_{\C_{43}^5}\cdot y.
\end{align}
Recall also that the left-hand side of
\eqref{AkBkgenPROPMgenres} equals
\eqref{midapproachMgen2}.
Hence Proposition \ref{AkBkgenPROPMgen} now follows as a consequence of 
Corollary \ref{EXPSUMcor1}.
\hfill$\square$ $\square$ $\square$

\subsection{The contribution from $B_k$-orbits}

Here we %
bound the sum in the third line of \eqref{MAINSTEP1}.
We assume $k\geq2$, %
since $B_k$ is empty for $k=1$.
\begin{prop}\label{BKBOUNDPROP}
Let $k\geq2$, $N\in\Z^+$
and set $X=\GaG$ with $\G=\SL(2,\R)\ltimes(\R^2)^k$
and $\Gamma=\Gamma(N)\ltimes(\Z^2)^k$.
Let $m_1\geq\max(311,2k+1)$,
and set $n:=3m_1+5$.
Then for any $f\in \C_0^{n}(X)$, 
any $h\in\C_{43}^{5}(\R)$,
and any $M\in\SL(2,\R)$, 
$\vecxi=(\trans\vecxi_1\,\trans\vecxi_2)\in(\R^2)^k$ and $0<y\leq1$,
\begin{align}\notag
\sum_{\veceta\in B_k}\sum_{R\in \overline\Gamma'/\Gamma'} \sum_{T\in [R]}
e\bigl(\tr\bigl(\veceta\trans T^{-1}\trans\vecxi\bigr)\bigr)
\int_\R \wh f_R\left(TM\mathrm{u}_x \mathrm{a}_y ,\veceta\right)h(x)\,dx
\hspace{125pt}
\\\label{BKBOUNDPROPres}
\ll_{m_1}\|f\|_{\C_0^{n}}\, \|h\|_{\C_{43}^5}\,\|M\|^{13}\,
\sum_{\veceta\in B_k}\|\veceta\|^{-m_1}\sum_{1\leq q\leq y^{-1/2}}\frac{\tau(q)}{q^{3/2}}
\biggl(1+\frac{\bigl\| q[\veceta;\vecxi]\bigr\|_{\Z}}{q\sqrt y}\biggr)^{\hspace{-3pt}-1},
\end{align}
where
\begin{align}\label{alfbetaxiDEF}
[\veceta;\vecxi]:=\bigl(\veceta_2\trans\vecxi_2,-\veceta_2\trans\vecxi_1,-\veceta_1\trans\vecxi_2,\veceta_1\trans\vecxi_1\bigr)\in\R^4
\qquad \text{for }\:\veceta=(\trans\veceta_1,\trans\veceta_2)\in B_k.
\end{align}
\end{prop}

The proof of 
Proposition~\ref{BKBOUNDPROP}
is an immediate application of Proposition \ref{AkBkgenPROPMgen}.
Indeed,
the sum over $T$ on
the left-hand side of \eqref{BKBOUNDPROPres}
equals
the left-hand side of \eqref{AkBkgenPROPMgenres}
with $\phi:=\hf_R(\cdot,\veceta)$ and $\vecalf:=[\veceta;\vecxi]$.
Note also that for any $m_1,n\in\NN$,
it follows from
Lemma \ref{PARTBDECAYFNLEM2}
(applied with $\beta=\frac13$ and $m=3m_1$)
that
\begin{align}\label{hfRnormbound}
\|\hf_R(\cdot,\veceta)\|_{\C_{m_1}^n}\ll_{m_1}\|f\|_{\C_0^{3m_1+n}}\|\veceta\|^{-m_1}.
\end{align}
In particular, for any $m_1\geq 311$ we have
$\|\hf_R(\cdot,\veceta)\|_{\C_{311}^5}\ll_{m_1}\|f\|_{\C_0^{3m_1+5}}\|\veceta\|^{-m_1}$.
Hence 
Proposition \ref{AkBkgenPROPMgen}
gives that the
sum over $T$ in \eqref{BKBOUNDPROPres} is
\begin{align*}
\ll_{m_1}
\|f\|_{\C_0^{3m_1+5}}\|\veceta\|^{-m_1}
\, \|h\|_{\C_{43}^5}\,\|M\|^{13}\,
\sum_{1\leq q\leq y^{-1/2}}\frac{\tau(q)}{q^{3/2}}
\biggl(1+\frac{\| q[\veceta;\vecxi]\|_{\Z}}{q\sqrt y}\biggr)^{\hspace{-3pt}-1}.
\end{align*}
Adding now over $\veceta$ and $R$,
requiring $m_1>2k$ to ensure that
$\sum_{\veceta\in B_k}\|\veceta\|^{-m_1}<\infty$, 
we obtain the bound \eqref{BKBOUNDPROPres},
i.e.\ Proposition~\ref{BKBOUNDPROP} is proved.
\hfill$\square$

\subsection{The contribution from $A_k$-orbits}
\label{Aknewsec}

Next we bound the sum in the second line of \eqref{MAINSTEP1}.
\begin{prop}\label{AkboundProp}
Let $k\geq1$, $N\in\Z^+$
and set $X=\GaG$ with $\G=\SL(2,\R)\ltimes(\R^2)^k$
and $\Gamma=\Gamma(N)\ltimes(\Z^2)^k$.
Let $m\geq \max(321,k+1)$,
and set $\alpha=\frac12(m+321)$.
Then for any $f\in\C_{\alpha}^{m+5}(X)$, any $h\in\C_{43}^5(\R)$,
and any $M\in\SL(2,\R)$, 
$\vecxi\in(\R^2)^k$ and $0<y\leq1$,
\begin{align}\notag
\sum_{\veceta\in A_k}
&\sum_{R\in \overline\Gamma'_\infty\backslash\overline\Gamma'/\Gamma'}
\sum_{T\in \Gamma'_\infty \backslash [R]}
e\bigl(\tr(\veceta\trans T^{-1}\trans\vecxi\bigr)\bigr)\int_\R
\wh f_R\left(TM\mathrm{u}_x \mathrm{a}_y ,\veceta\right)h(x)\,dx
\\\label{AkboundPropRES}
&\ll_m \|f\|_{\C_\alpha^{m+5}}
\, \|h\|_{\C_{43}^5}\,\|M\|^{13}
\,\sum_{\veca\in\Z^k\setminus\{\bn\}}\|\veca\|^{-m}\sum_{1\leq q\leq y^{-1/2}}\frac{\tau(q)}{q^{3/2}}
\biggl(1+\frac{\bigl\| q\veca\vecxi %
\bigr\|_{\Z}}{q\sqrt y}\biggr)^{\hspace{-3pt}-1}.
\end{align}
\end{prop}

To start the proof of Proposition \ref{AkboundProp},
let $\veceta=(\trans\veca,\bn)\in A_k$ and
$R\in \overline\Gamma'_\infty\backslash\overline\Gamma'/\Gamma'$ be fixed.
Recall that then $\hf_R(\smatr 1{Nn}01 T;\veceta)=\hf_R(T;\veceta)$ for all $T\in\SL(2,\R)$
and $n\in\Z$,
as noted just below \eqref{BASICFOURIER}.
Let us express $\hf_R$ in Iwasawa coordinates, as in \eqref{TFNIWASAWA};
then the relation just noted reads
$\hf_R(u+Nn,v,\theta;\veceta)=\hf_R(u,v,\theta;\veceta)$.
Fix, once and for all, a function $\psi\in\C_c^\infty(\R)$ satisfying
$\supp(\psi)\subset(-1,1)$ and
$\sum_{n\in\Z}\psi(x+n)=1$ for all $x\in\R$.
Let us now define the function %
$\phi:\SL(2,\R)\to\CC$ by (in Iwasawa coordinates)
\begin{align}\label{tfRDEFrep}
\phi(u,v,\theta):=\psi(N^{-1}u)\,\hf_R(u,v,\theta;\veceta).
\end{align}
Then 
$\supp(\phi)\subset(-N,N)\times\R^+\times(\R/2\pi\Z)$,
and 
\begin{align*}
\sum_{n\in\Z}\phi\bigl(u+Nn,v,\theta\bigr)
=\sum_{n\in\Z}\psi(N^{-1}u+n)\hf_R(u,v,\theta;\veceta)=\hf_R(u,v,\theta;\veceta)
\end{align*}
for all $u,v,\theta$.
In other words,
\begin{align*}
\hf_R(T;\veceta)=\sum_{n\in\Z}\phi\biggl(\matr 1{Nn}01 T\biggr)
=\sum_{\gamma\in\Gamma_\infty'}\phi\bigl(\gamma T\bigr),
\qquad\forall T\in\SL(2,\R).
\end{align*}
Using this formula in \eqref{AkboundPropRES},
and using also the fact that 
$\veceta\trans T^{-1}$ is invariant under $T\mapsto\gamma T$ for any $\gamma\in\Gamma_\infty'$
(since $\veceta=(\trans\veca,\bn)$),
it follows that 
the sum over $T$ in \eqref{AkboundPropRES} can be
rewritten as:
\begin{align}\label{MAINSTEP1line2disc2}
\sum_{T\in [R]}
e\bigl(\tr(\veceta\trans T^{-1}\trans\vecxi\bigr)\bigr)\int_\R
\phi\left(TM\mathrm{u}_x \mathrm{a}_y \right)h(x)\,dx.
\end{align}
This is exactly the sum which appears in the left-hand side of \eqref{AkBkgenPROPMgenres},
with $\vecalf=(0,0,-\veca\trans\vecxi_2,\veca\trans\vecxi_1)$,
where $\trans\vecxi_1$ and $\trans\vecxi_2$ are the column vectors of $\vecxi$.

Next we prove:
\begin{lem}\label{AktfnormboundLEM1}
For any $m,n\in\NN$ and $r\in\R_{\geq0}$
such that $m\geq 2n+r$,
letting $\alpha:=\frac12(m+r)+n$ we have
\begin{align}\label{AktfnormboundLEM1res}
\|\phi\|_{\C_r^n}\ll_{m,n,r}\|f\|_{\C_{\alpha}^{m+n}}\|\veca\|^{-m}.
\end{align}
\end{lem}

\begin{proof}
Let $m,n,r,\alpha$ be given as in the statement of the lemma.
By \eqref{tfRDEFrep} and Lemma \ref{DERDECAYTFNFROMCMNORMLEM2},
for any $f\in\C_{\alpha}^{m+n}(X)$
and any $\ell_1,\ell_2,\ell_3\in\NN$ with $\ell:=\ell_1+\ell_2+\ell_3\leq n$ we have
\begin{align}\label{hfpartderbound}
\bigl|\partial_u^{\ell_1}\partial_v^{\ell_2}\partial_\theta^{\ell_3}\phi(u,v,\theta)\bigr|
&\ll_{m,n,r}
\|f\|_{\C_\alpha^{m+\ell}}\|\veca\|^{-m}
\begin{cases} v^{\frac m2-\ell_2-\alpha}&\text{if }v\geq1
\\
v^{\frac m2-\ell_2-\ell_1}&\text{if }v\leq1,
\end{cases}
\end{align}
for all $u,v,\theta$.
Next, one verifies by a straightforward computation that %
in terms of the Iwasawa coordinates,
the left invariant vector fields 
$X_1,X_2,X_3$ on $\G'=\SL(2,\R)$ (see \eqref{X1X2X3def})
are  given by:
\begin{align}\label{AktfnormboundLEM1pf1}
X_1&=(\cos2\theta)v\partial_u-(\sin2\theta)v\partial_v-(\sin\theta)^2\partial_\theta;
\\\notag
X_2&=(\cos2\theta)v\partial_u-(\sin2\theta)v\partial_v+(\cos\theta)^2\partial_\theta;
\\\notag
X_3
&=2(\sin2\theta)v\partial_u+2(\cos2\theta)v\partial_v+(\sin 2\theta)\partial_\theta.
\end{align}
It follows that for any $s\in\NN$ and any 
$\veck\in\{1,2,3\}^s$,
\begin{align}\label{AktfnormboundLEM1pf2}
\bigl(X_{k_1}X_{k_2}\cdots X_{k_s}\phi\bigr)(u,v,\theta)
=\sum_{\vecell\in S(\veck)}P_{\veck,\vecell}(\theta) v^{\ell_4}
\cdot \partial_u^{\ell_1}\partial_v^{\ell_2}\partial_\theta^{\ell_3}\phi(u,v,\theta),
\end{align}
where $S(\veck)$ is a finite subset of $(\Z_{\geq0})^4$
and each $P_{\veck,\vecell}(\theta)$ is a polynomial in $\sin\theta$ and $\cos\theta$
with coefficients in $\Z$;
and where $\ell_4\leq s$ and $\ell_1+\ell_2+\ell_3\leq s$
for all $\vecell\in S(\veck)$.
It follows from \eqref{hfpartderbound} and \eqref{AktfnormboundLEM1pf2}
that, for any $s\leq n$ and $\veck\in\{1,2,3\}^s$,
\begin{align}\label{hfpartderboun3}
\bigl|\bigl(X_{k_1}X_{k_2}\cdots X_{k_s}\phi\bigr)(u,v,\theta)\bigr|
\ll_{m,n,r}\|f\|_{\C_\alpha^{m+s}}\|\veca\|^{-m}
\begin{cases} v^{\frac m2+s-\alpha}&\text{if }v\geq1
\\
v^{\frac m2-s}&\text{if }v\leq1.
\end{cases}
\end{align}
Recalling the definition of the norm $\|\cdot\|_{\C_r^n}$ of functions on $\SL(2,\R)$
(see \eqref{CrmnormonSL2def})
and using the formula
\begin{align}\label{FROBINIWASAWAEQ}
\left\|\matr 1u01\matr{\sqrt v}00{1/\sqrt
v}\matr{\cos\theta}{-\sin\theta}{\sin\theta}{\cos\theta}\right\|
=\sqrt{\frac{u^2+v^2+1}v},
\end{align}
along with the fact that
$\supp(\phi)\subset(-N,N)\times\R_{>0}\times(\R/2\pi\Z)$
because of the factor $\psi(N^{-1}u)$ in \eqref{tfRDEFrep},
it now follows that 
\begin{align*}
\|\phi\|_{\C_r^n}\ll_{m,n,r}
\|f\|_{\C_\alpha^{m+n}}\|\veca\|^{-m}
\sup_{v>0}\biggl(\Bigl(\frac{N^2+1+v^2}v\Bigr)^{r/2}
\left.\begin{cases} v^{\frac m2+n-\alpha}&\text{if }v\geq1
\\
v^{\frac m2-n}&\text{if }v\leq1
\end{cases}\right\}\biggr).
\end{align*}
Here the supremum is finite since
$m\geq 2n+r$ and $\alpha=\frac12(m+r)+n$.
Hence we obtain the bound in \eqref{AktfnormboundLEM1res}.
\end{proof}

It follows from 
Proposition \ref{AkBkgenPROPMgen},
Lemma \ref{AktfnormboundLEM1} and the observations around \eqref{MAINSTEP1line2disc2},
that
for any $m\geq 321$, letting $\alpha=\frac12(m+321)$, %
we have that the sum over $T$ in \eqref{AkboundPropRES}
is
\begin{align}
\ll_{m} \|f\|_{\C_\alpha^{m+5}}\|\veca\|^{-m}
\|h\|_{\C_{43}^5}\,\|M\|^{13}
\, \sum_{1\leq q\leq y^{-1/2}}\frac{\tau(q)}{q^{3/2}}
\biggl(1+\frac{\| q\vecalf\|_{\Z}}{q\sqrt y}\biggr)^{\hspace{-3pt}-1},
\end{align}
with $\vecalf=(0,0,-\veca\trans\vecxi_2,\veca\trans\vecxi_1)$.
Hence, requiring also $m>k$ so that
$\sum_{\veca\in\Z^k\setminus\{\bn\}}\|\veca\|^{-m}<\infty$,
and noticing that $\|q\vecalf\|_{\Z}=\|q\veca\vecxi\|_{\Z}$,
we obtain the bound
\eqref{AkboundPropRES}
i.e.\ Proposition \ref{AkboundProp} is proved.
\hfill$\square$ 

\subsection{Proof of Theorem \ref*{MAINTHM3gen}}
\label{MAINTHM3genpfSEC}

Recalling \eqref{alfbetaxiDEF},
we note that for any $\veceta=(\trans\veca,\trans\vecb)\in(\Z^2)^k$, $\vecxi\in(\R^2)^k$ and $q\in\Z^+$,
we have
$\bigl\| q[\veceta;\vecxi]\bigr\|_{\Z}\geq\bigl\| q(\veca\trans\vecxi_1,\veca\trans\vecxi_2)\bigr\|_{\Z}
=\|q\veca\vecxi\|_{\Z}$.
Recall also that $B_k$ is a subset of $(\Z^k\setminus\{\bn\})^2$.
Hence in \eqref{BKBOUNDPROPres} we have
\begin{align}\notag
&
\sum_{\veceta\in B_k}\|\veceta\|^{-m_1}\sum_{1\leq q\leq y^{-1/2}}\frac{\tau(q)}{q^{3/2}}
\biggl(1+\frac{\bigl\| q[\veceta;\vecxi]\bigr\|_{\Z}}{q\sqrt y}\biggr)^{\hspace{-3pt}-1}
\\\label{BKBOUNDPROPresdisc}
&\leq\sum_{\veca\in\Z^k\setminus\{\bn\}}
\sum_{\vecb\in\Z^k\setminus\{\bn\}}
(\|\veca\|^2+\|\vecb\|^2)^{-m_1/2}
\sum_{1\leq q\leq y^{-1/2}}\frac{\tau(q)}{q^{3/2}}
\biggl(1+\frac{\bigl\| q\veca\vecxi\bigr\|_{\Z}}{q\sqrt y}\biggr)^{\hspace{-3pt}-1}.
\end{align}
 Here, for every $\veca\in\Z^k\setminus\{\bn\}$,
\begin{align*}
\sum_{\vecb\in\Z^k\setminus\{\bn\}}
(\|\veca\|^2+\|\vecb\|^2)^{-m_1/2}
\leq\#\{\vecb\in\Z^k\setminus\{\bn\}\col\|\vecb\|\leq\|\veca\|\}\cdot\|\veca\|^{-m_1}
\hspace{110pt}
\\
+\sum_{j=0}^\infty
\#\{\vecb\in\Z^k\setminus\{\bn\}\col  2^j\|\veca\|<\|\vecb\|\leq 2^{j+1}\|\veca\|\}\cdot(2^j\|\veca\|)^{-m_1}
\hspace{20pt}
\\
\ll_k\|\veca\|^{k-m_1}+\sum_{j=0}^\infty (2^j\|\veca\|)^{k-m_1}
\ll\|\veca\|^{k-m_1},
\end{align*}
where in the last step we used the fact that $m_1>2k>k$
in Proposition \ref{BKBOUNDPROP}.
It follows that the sum in \eqref{BKBOUNDPROPresdisc} is
\begin{align*}
\ll_k\sum_{\veca\in\Z^k\setminus\{\bn\}}
\|\veca\|^{k-m_1}
\sum_{1\leq q\leq y^{-1/2}}\frac{\tau(q)}{q^{3/2}}
\biggl(1+\frac{\bigl\| q\veca\vecxi\bigr\|_{\Z}}{q\sqrt y}\biggr)^{\hspace{-3pt}-1}.
\end{align*}

Now let $m$ be an arbitrary integer $\geq\max(321,k+1)$
and set $\alpha=\frac12(m+321)$,
as in Proposition \ref{AkboundProp}.
Then set $m_1=m+k$. %
For this $m_1$, we get
$n=3m_1+5=3m+3k+5$ in Proposition \ref{BKBOUNDPROP},
and since $n>m+5$,
the sum of the
bounds in \eqref{BKBOUNDPROPres} and \eqref{AkboundPropRES}
is 
\begin{align}\label{AkBkboundsum}
\ll_m \|f\|_{\C_{\alpha}^n}\,\|h\|_{\C_{43}^5}\,\|M\|^{13}
\sum_{\veca\in\Z^k\setminus\{\bn\}}\|\veca\|^{-m}\sum_{1\leq q\leq y^{-1/2}}\frac{\tau(q)}{q^{3/2}}
\biggl(1+\frac{\bigl\| q\veca\vecxi\bigr\|_{\Z}}{q\sqrt y}\biggr)^{\hspace{-3pt}-1}.
\end{align}
Theorem \ref{MAINTHM3gen} is now a consequence 
of \eqref{MAINSTEP1},
the estimate in \eqref{MAINTERMgen},
the observation just below \eqref{MAINTERMgen},
and the bound in \eqref{AkBkboundsum}.
\hfill$\square$ $\square$ $\square$

\subsection{Proof of Corollary \ref*{MAINTHM3cor2}}
\label{BTsec}
Fix a function $\omega\in\C^\infty_c(\R)$ satisfying
$\omega\geq0$,
$\supp(\omega)\subset[-N,N]$
and $\sum_{j\in\Z}\omega(x-jN)=1$ for all $x\in\R$.
Then
\begin{align}\notag
\frac 1T\int_\R f\bigl(\Gamma(1_2, {\vecxi})\mathrm{u}_x\mathrm{a}_y\bigr)\eta(x)h(T^{-1}x)\,dx
\hspace{200pt}
\\\label{MAINTHM3cor2pf1}
=\frac 1T\sum_{j\in\Z}\int_\R f\bigl(\Gamma(1_2, {\vecxi})\mathrm{u}_x\mathrm{a}_y\bigr)
\,\omega(x-jN)\eta(x)h(T^{-1}x)\,dx.
\end{align}
Substituting $x=x_{\new}+jN$
and using $\Gamma(1_2,{\vecxi})\mathrm{u}_{jN}=\Gamma \mathrm{u}_{jN}(1_2,\vecxi\mathrm{u}_{jN})
=\Gamma(1_2,\vecxi\mathrm{u}_{jN})$,
which holds since $\mathrm{u}_{jN}\in %
\Gamma$,
we obtain %
\begin{align}\label{MAINTHM3cor2pf10}
=\frac1T\sum_{j\in\Z}\int_\R f\bigl(\Gamma\bigl(1_2,{\vecxi}\mathrm{u}_{jN}\bigr)\mathrm{u}_x\mathrm{a}_y\bigr)
\,\tth_j(x)\,dx
\end{align}
with $\tth_j(x):=\omega(x)\,\eta(x+jN)\,h\bigl(T^{-1}(x+jN)\bigr)$.
Applying Theorem \ref{MAINTHM3gen},
we get %
\begin{align}\label{MAINTHM3cor2pf12}
=\frac 1T\sum_{j\in\Z}\biggl(\int_X f\,d\mu\int_{\R}\tth_j(x)\,dx
+O\Bigl(\|f\|_{\C_\alpha^n} \|\tth_j\|_{\C_{43}^5}\delta_m(y;\vecxi\mathrm{u}_{jN})\Bigr)\biggr).   %
\end{align}
Here
\begin{align*}
\sum_{j\in\Z} \int_{\R}\tth_j(x)\,dx
=\sum_{j\in\Z} \int_{\R}\omega(x-jN) \eta(x)h(T^{-1}x)\,dx
=T\int_{\R} \eta(Tx)h(x)\,dx.
\end{align*}
Also, for each $j\in\Z$
we have $\tth_j\in\C^5_c(\R)$ with $\supp(\tth_j)\subset[-N,N]$,
and hence since $T\geq1$,
\begin{align}\label{MAINTHM3cor2pf11}
\|\tth_j\|_{\C_{43}^5}\ll \|\eta\|_{\C_0^5}\|h\|_{\C^5_2}\cdot (1+|j|/T)^{-2}.
\end{align}
Using these facts in \eqref{MAINTHM3cor2pf12},
we obtain \eqref{MAINTHM3cor2res1}.
\hfill$\square$

\section{General orbits}
\label{genorbSEC}

\subsection{Proof of Theorem \ref*{MAINTHM4}}
\label{MAINTHM4pfsubsec1}
Let $f,h,g$ and $T$ be given as in the statement of Theorem~\ref{MAINTHM4}.
Take $M\in\SL(2,\R)$ and $\vecxi\in(\R^2)^k$ so that $g=(1_2,{\vecxi})M$.
Choose $\gamma\in\SL(2,\Z)$
so that $\gamma^{-1} M\mathrm{a}_T$ lies inside the standard fundamental domain for
$\SL(2,\Z)\bs\SL(2,\R)$, i.e.\ so that
\begin{align}\label{funddomDEF}
\bigl|\gamma^{-1} M\mathrm{a}_T(i)\bigr|\geq1
\quad\text{and}\quad
\bigl|\tre\bigl(\gamma^{-1} M\mathrm{a}_T(i)\bigr)\bigr|\leq\tfrac12.
\end{align}
Then we have
\begin{align}\label{scrYlowerbd}
\scrY(M\mathrm{a}_T )=
\tim \bigl(\gamma^{-1} M\mathrm{a}_T (i)\bigr)\geq\tfrac12\sqrt3.
\end{align}
Let us write 
\begin{align}\label{ygTdef}
y=y_g(T):=\frac{\scrY(M\mathrm{a}_T )}T.
\end{align}
\begin{remark}\label{ygTSg0Tremark}
The notation ``$y_g(T)$'' is the same as in
\cite[eq.\ (6)]{SASL}.
Let us also note that
\begin{align}\label{ygTSg0TremarkRES}
y=y_g(T)\asymp S_{g,\bn}(T)^{-2},
\end{align}
with absolute implied constants.
Indeed, recall
\eqref{bMxiqTDEF},
and note that
the condition
$S\,\fR_T\cap \Z^2\,\p_{\bn}(g)=\{\bn\}$
is equivalent with $ST^{-\frac12}[-1,1]^2\cap \Z^2\,\p_{\bn}(g\mathrm{a}_T)=\{0\}$;
therefore
$T^{-\frac12}\, S_{g,\bn}(T)$ is comparable with the Euclidean length of the 
shortest non-zero vector in the lattice $\Z^2\,\p_{\bn}(g\mathrm{a}_T)$.
From this, the relation \eqref{ygTSg0TremarkRES}
follows from the facts noted in
\cite[just above eq.\ (6)]{SASL}.
Note also that \eqref{ygTSg0TremarkRES} gives for
the first term appearing in the bound in \eqref{MAINTHM4res}:
\begin{align}\label{ygTSg0TremarkRES2}
\scrL_3\Bigl(S_{g,\bn}(T)^{-\frac12}\Bigr)\asymp\scrL_3(y^{\frac14}).
\end{align}
\end{remark}

Continuing with the proof of Theorem \ref{MAINTHM4},
note that if $y\geq 10^{-4}$ then $\scrL_3\bigl(S_{g,\bn}(T)^{-\frac12}\bigr)\asymp\scrL_3(y^{\frac14})\gg1$
so that \eqref{MAINTHM4res} holds trivially.
Hence from now on let us assume that $y<10^{-4}$.
It follows that $T=\scrY(M\mathrm{a}_T )/y>10^3$.

The following result is the central step of the proof of Theorem \ref*{MAINTHM4}.
\begin{prop}\label{MAINTHM4pfPROP}
Under the assumptions introduced %
above,
\begin{align}\label{MAINTHM4pfPROPres}
\biggl|\frac1T\int_{\R}f\bigl(\Gamma g\mathrm{u}_t\bigr) h\Bigl(\frac tT\Bigr)\,dt
-\int_X f\,d\mu \int_{\R}h\,dt\biggr|
\hspace{170pt}
\\\notag
\ll \|f\|_{\C_{\alpha}^n}\|h\|_{\C_0^5}\Biggl(\sqrt y+
\sum_{\ell\in\Z}    %
\frac{Ty}{(Ty+|\ell|)^2}\,\delta_m\Biggl(
\biggl(\frac{Ty+|\ell|}{T\sqrt y}\biggr)^2;\vecxi\gamma\mathrm{u}_{\ell}\Biggr)\Biggr).
\end{align}
\end{prop}
\begin{proof}
First assume $\scrY(M\mathrm{a}_T )\leq 100$. %
Using 
$\Gamma g=\Gamma\gamma^{-1}g
=\Gamma (1_2,{\vecxi}\gamma)\gamma^{-1}M$
and $\mathrm{u}_t=\mathrm{a}_T \, \mathrm{u}_{t/T}\, \mathrm{a}_{T^{-1}}$, we have
\begin{align*}
\frac1T\int_{\R}f\bigl(\Gamma g\mathrm{u}_t\bigr) h\Bigl(\frac tT\Bigr)\,dt
=\int_{\R}f\Bigl(\Gamma(1_2,{\vecxi}\gamma)\gamma^{-1} M\,\mathrm{a}_T\, \mathrm{u}_x\,\mathrm{a}_{T^{-1}}\Bigr) h(x)\,dx.
\end{align*}
Since $\gamma^{-1} M\mathrm{a}_T $ belongs to the standard fundamental domain
for $\SL(2,\Z)\bs\SL(2,\R)$ and also
$\tim\bigl(\gamma^{-1} M\mathrm{a}_T (i)\bigr)\leq100$,
it follows that $\|\gamma^{-1} M\mathrm{a}_T \|\ll1$.
Recall also that we are assuming $\supp(h)\subset[-1,1]$.
Hence Theorem \ref{MAINTHM3gen} gives
\begin{align}\label{nonclosedgendisc33case4pre}
\frac1T\int_{\R}f\bigl(\Gamma g\mathrm{u}_t\bigr) h\Bigl(\frac tT\Bigr)\,dt
=\int_X f\,d\mu \int_{\R}h\,dt
+O\biggl(\|f\|_{\C_{\alpha}^{n}}\, \|h\|_{\C_0^5}\,\delta_{m}(T^{-1};\,{\vecxi}\gamma)\biggr).
\end{align}
This implies the bound in \eqref{MAINTHM4pfPROPres};
indeed the error term in \eqref{nonclosedgendisc33case4pre}
is subsumed by the contribution from $\ell=0$ in the right-hand side of 
\eqref{MAINTHM4pfPROPres}, by Lemma~\ref{deltambasicfactLEM1}
and since $T^{-1}=\scrY(M\mathrm{a}_T )^{-1}y\leq\frac2{\sqrt 3}y$.

\vspace{5pt}

It remains to consider the case $\scrY(M\mathrm{a}_T )>100$.
Let us set
\begin{align}\label{MAINTHM4pf1pre}
\matr abcd:=\gamma^{-1} M \mathrm{a}_T\in\SL(2,\R).
\end{align}
Then %
$\tim(\gamma^{-1} M\mathrm{a}_T(i))=(c^2+d^2)^{-1}$,
and so
\begin{align}\label{MAINTHM4pf1}
c^2+d^2=\frac1{\scrY(M\mathrm{a}_T )}<\frac1{100}.
\end{align}
Let us fix, once and for all, a function $\Phi\in\C_c^{\infty}(\R)$ satisfying
$\Phi\geq0$, 
$\supp(\Phi)\subset[-1,1]$ and $\int_{\R}\Phi(z)\,dz=1$.
We then decompose our orbital integral as follows:
\begin{align}\label{MAINTHM4pf10}
\frac1T\int_{\R}
&f\bigl(\Gamma g \mathrm{u}_t\bigr) h\Bigl(\frac tT\Bigr)\,dt
=\int_{\R}\int_{\R}f\bigl(\Gamma g\mathrm{u}_{Ts}\bigr) h(s)\,\Phi\Bigl(\frac{z-s}{(cs+d)^2}\Bigr)\,(cs+d)^{-2}\,ds\,dz.
\end{align}
This identity holds by Fubini's Theorem
and since 
$(cs+d)^{-2}\,\int_{\R}\Phi\bigl(\frac{z-s}{(cs+d)^2}\bigr)\,dz=1$
for all $s\in\R\setminus\{-\frac dc\}$;
it also follows that the double integral in the right-hand side is absolutely convergent.

\vspace{5pt}

Given any $z\in\R\setminus\{-\frac dc\}$,
we will now analyse the integral over $s$ in \eqref{MAINTHM4pf10}.
Let us define
\begin{align}\label{Fzdef}
F_z(s):=\Phi\Bigl(\frac{z-s}{(cs+d)^2}\Bigr)\,(cs+d)^{-2}.
\end{align}
Note that since %
$z\neq-\frac dc$,
the function $F_z$ is $\C^\infty$ on all $\R$, with the understanding that
(if $c\neq0$) $F_z(-\frac dc):=0$.
Let us also write
\begin{align}\label{tjdefs}
t=t(z):=(cz+d)^{-2}   %
\qquad\text{and}\qquad
j=j(z):=\left\lfloor\frac{az+b}{cz+d}\right\rfloor\in\Z.
\end{align}
Using $\mathrm{u}_{Ts}=\mathrm{a}_{T/t} \mathrm{u}_{ts} \mathrm{a}_{t/T}$
and $\Gamma g=\Gamma \mathrm{u}_{-j}\gamma^{-1}g
=\Gamma(1_2,\vecxi\gamma \mathrm{u}_j) \mathrm{u}_{-j}\gamma^{-1}M$,
and substituting $s=z+t^{-1}x$,
we may rewrite the integral over $s$ in \eqref{MAINTHM4pf10} as follows:
\begin{align}\label{MAINTHM4pf10a}
&\int_{\R}f\bigl(\Gamma g\mathrm{u}_{Ts}\bigr) h(s)\,F_z(s)\,ds
=\frac1t\int_{\R}f\bigl(\Gamma (1_2,\vecxi\gamma \mathrm{u}_{j}) \tM %
\mathrm{u}_x \mathrm{a}_{t/T}\bigr)\, \tth_z(x)\,dx,
\end{align}
where 
\begin{align}\label{MAINTHM4pf10b}
\tM=\tM(z):=\mathrm{u}_{-j}\gamma^{-1}M\mathrm{a}_{T/t} \mathrm{u}_{tz}
&=\mathrm{u}_{-j}\matr abcd \mathrm{a}_{t^{-1}}\mathrm{u}_{tz}
\\\notag
&=\matr{(a-jc)/\sqrt t}{(az+b-j(cz+d))\sqrt t}{c/\sqrt t}{(cz+d)\sqrt t}.
\end{align}
and
\begin{align}\label{thzDEF}
\tth_z(x)=h(z+t^{-1}x)\,F_z(z+t^{-1}x).
\end{align}
Using $\supp(h)\subset[-1,1]$ and \eqref{Fzdef},
it follows that for every $s$ satisfying $h(s)F_z(s)\neq0$,
we have $|s|\leq1$ and $|z-s|\leq(cs+d)^2$,
and 
\begin{align}\label{MAINTHM4pf11}
0<(cs+d)^2\leq c^2+2|cd|+d^2\leq 2(c^2+d^2) =2\,\scrY(M\mathrm{a}_T )^{-1}
<\tfrac1{50},
\end{align}
where we also used \eqref{MAINTHM4pf1}.
These imply that $|z|<\frac{51}{50}$,
and also
\begin{align}\label{MAINTHM4pf14prepre}
\bigl|(cs+d)-(cz+d)\bigr|
=|c| |z-s|\leq |c|(cs+d)^2
< 10^{-1}50^{-\frac12}|cs+d|
< 70^{-1}|cs+d|.   %
\end{align}
Therefore,
\begin{align}\label{MAINTHM4pf11a}
|cz+d|<\tfrac{71}{70}\,|cs+d|
< \tfrac 32\, \scrY(M\mathrm{a}_T )^{-\frac12},
\end{align}
where we used \eqref{MAINTHM4pf11} and $\frac{71}{70}\, \sqrt 2<\frac32$.
It will also be useful to record that
\eqref{MAINTHM4pf14prepre} implies
$\bigl|(cs+d)-(cz+d)\bigr|<69^{-1}|cz+d|$,
and so
\begin{align}\label{MAINTHM4pf11b}
(cs+d)^2< 2(cz+d)^2
\qquad\text{and}\qquad
(cs+d)^{-2}< 2(cz+d)^{-2}.
\end{align}

Next we claim that
\begin{align}\label{MAINTHM4pf12goal}
\|\tth_z\|_{\C_{43}^5}\ll\|h\|_{\C^5_0}\,t.
\end{align}
To prove this, one starts by verifying %
that there exist %
polynomials
$P_{n,m}(X,Y,Z)\in\Z[X,Y,Z]$ for 
$n\geq m\geq0$ such that for all $n\geq0$:
\begin{align}\label{MAINTHM4pf12}
F_z^{(n)}(s)=\sum_{m=0}^n \Phi^{(m)}\Bigl(\frac{z-s}{(cs+d)^2}\Bigr)\,P_{n,m}\Bigl((cs+d)^{-1},z-s,c\Bigr),
\end{align}
and furthermore, for any $n\geq m\geq0$,
a monomial $X^{\ell}Y^vZ^{\alpha}$
can appear with a non-zero coefficient in 
$P_{n,m}(X,Y,Z)$ only if
$v,\alpha\leq n$ and 
$2\leq \ell\leq 2+2n+v$.
Next note that for any $s$ with $h(s)F_z(s)\neq0$
we have $|z-s|\leq(cs+d)^2$ and
$\frac{70}{71}|cz+d|\leq |cs+d|\leq \frac{70}{69}|cz+d|$
(by \eqref{MAINTHM4pf14prepre})
and $t(z)=(cz+d)^{-2}>\frac49\,\scrY(M\mathrm{a}_T)>1$
(by \eqref{MAINTHM4pf11a}),
and hence for any $0\leq v\leq n$ and $2\leq\ell\leq2+2n+v$:
\begin{align*}
\Bigl|(cs+d)^{-\ell}(z-s)^v\Bigr|\ll_{\ell,v}|cz+d|^{-\ell+2v}   %
=t^{\frac12\ell-v}\leq t^{1+n+\frac12v-v}=t^{1+n-\frac12v}\leq t^{n+1}.
\end{align*}
Recall also that $|c|<\frac1{10}$, by \eqref{MAINTHM4pf1}.
Hence it follows that 
for any $s$ with $h(s)F_z(s)\neq0$,
and any $n\geq0$,
we have $\bigl|F^{(n)}_z(s)\bigr|\ll_n t^{n+1}$.
Using also Leibniz' rule, we conclude that
\begin{align*}
\biggl|\frac{d^n}{ds^n}\bigl(h(s)F_z(s)\bigr)\biggr|
\ll_n\|h\|_{\C^n_0}\,t^{n+1},
\qquad\forall s\in\R.
\end{align*}
Hence, via \eqref{thzDEF},
$\|\tth_z\|_{\C^n_0}\ll_n \|h\|_{\C^n_0}\,t$.
Recall from previous observations that
$h(s)F_z(s)\neq0$ implies $|z-s|\leq(cs+d)^2< %
2(cz+d)^2=2t^{-1}$ (see in particular \eqref{MAINTHM4pf11b});
therefore $\supp(\tth_z)\subset[-2,2]$.
Hence \eqref{MAINTHM4pf12goal} follows.

Next we will verify that for any 
$z\in\R\setminus\{-\frac dc\}$ for which $\tth_z\not\equiv0$,
the matrix $\tM=\tM(z)$ defined in \eqref{MAINTHM4pf10b}
satisfies
\begin{align}\label{MAINTHM4pf10bfact}
\|\tM\|\ll 1.
\end{align}
For this, recall that we have seen that $\tth_z\not\equiv0$
implies that $t>\frac49\,\scrY(M\mathrm{a}_T)>1$;
hence the bottom left entry of $\tM$ satisfies
$|c/\sqrt t|<|c|<10^{-1}$.
Using the definitions of $t=t(z)$ and $j=j(z)$ in \eqref{tjdefs},
it also follows that the bottom right entry of $\tM$ has absolute value $1$,
and that the top right entry satisfies
\begin{align*}
\bigl|(az+b-j(cz+d))\sqrt t\bigr| %
=\biggl|\frac{az+b}{cz+d}-j\biggr|<1.
\end{align*}
Using these facts together with $\det(\tM)=1$, it follows that the remaining entry of $\tM$ satisfies
$\bigl|(a-jc)/\sqrt t\bigr|<2$.
Hence \eqref{MAINTHM4pf10bfact} holds.

Now if $t(z)/T\leq1$, then Theorem \ref{MAINTHM3gen} applies 
to the right-hand side of \eqref{MAINTHM4pf10a},
and using \eqref{MAINTHM4pf12goal}
and \eqref{MAINTHM4pf10bfact}
we conclude %
that for any fixed $m,n,\alpha$ as in the statement of
Theorem \ref{MAINTHM3gen},
\begin{align}\notag
\int_{\R} f\bigl(\Gamma g\mathrm{u}_{Ts}\bigr)  h(s)\,F_z(s) %
\,ds    %
&=\frac1t\int_X f\,d\mu \int_{\R} \tth(x)\,dx
+O\Bigl(\|f\|_{\C_{\alpha}^n}\|h\|_{\C_0^5}\,\delta_{m}\Bigl(\frac tT;\, {\vecxi} \gamma \mathrm{u}_{j}\Bigr)\Bigr)
\\\label{MAINTHM4pf21pre1}
&=\int_X f\,d\mu \int_{\R} h(s)\,F_z(s) %
\,ds
+O\Bigl(\|f\|_{\C_{\alpha}^n}\|h\|_{\C_0^5}\,\delta_{m}\Bigl(\frac tT;\, {\vecxi} \gamma \mathrm{u}_{j}\Bigr)\Bigr).
\end{align}

We will also require a trivial bound, to use when $t(z)/T$ is not small.
First note that for every $z\in\R\setminus\{-\frac dc\}$ we have
\begin{align}\label{MAINTHM4pf21pre2pf1}
\int_{\R}|h(s)|F_z(s)\,ds
\leq 2\|h\Phi\|_{\L^{\infty}}\int_{z-2(cz+d)^2}^{z+2(cz+d)^2}(cz+d)^{-2}\,ds
=8\|h\Phi\|_{\L^{\infty}}\ll\|h\|_{\L^{\infty}},
\end{align}
where we used the fact that
$h(s)F_z(s)\neq0$ implies that
$|s|\leq1$ and $|z-s|\leq(cs+d)^2<2(cz+d)^2$,
and also $(cs+d)^{-2}<2(cz+d)^{-2}$;
indeed see \eqref{MAINTHM4pf11b}.
It follows from \eqref{MAINTHM4pf21pre2pf1} that
\begin{align}\label{MAINTHM4pf21pre2}
\int_{\R} f\bigl(\Gamma g\mathrm{u}_{Ts}\bigr)  h(s)\,F_z(s) %
\,ds    %
=\int_X f\,d\mu \int_{\R} h(s)\,F_z(s) %
\,ds
+O\Bigl(\|f\|_{\L^{\infty}}\|h\|_{\L^{\infty}}\Bigr).  %
\end{align}

\vspace{5pt}

Recall that $t(z):=(cz+d)^{-2}$. Now set
\begin{align}\label{fSdef}
\fS:=\{z\in[-2,2]\col t(z)>10^{-2}T\}.
\end{align}
Let us write $Y:=\scrY(M\mathrm{a}_T )$.
Note that if $z\in\fS$ then $|cz+d|<10\,T^{-\frac12}< \frac1{10}\,Y^{-\frac12}$
(the last bound holds by \eqref{ygTdef} and since $y<10^{-4}$),
and so
\begin{align}\label{MAINTHM4pf27}
2|c|\geq|cz|\geq|d|-|cz+d|>|d|-\tfrac1{10}\,Y^{-\frac12}.
\end{align}
By using \eqref{MAINTHM4pf27} if $|d|>\frac9{10}\,Y^{-\frac12}$
and otherwise using \eqref{MAINTHM4pf1}, we conclude:
If $\fS\neq\emptyset$ then $|c|>\frac25\,Y^{-\frac12}$.
Since $z\in\fS\Rightarrow |z+\frac dc|<10|c|^{-1}T^{-\frac12}$,
it follows that
\begin{align}\label{MAINTHM4pf18}
\bigl|\fS\bigr| %
\leq\frac{20}{|c|\sqrt T}<50\sqrt{\frac{Y}T}=50\sqrt y,
\end{align}
where $|\fS|$ denotes the Lebesgue measure of $\fS$.
Now in \eqref{MAINTHM4pf10},
recall that $h(s)F_z(s)\neq0$ implies both $|s|\leq1$ and $|z|<\frac{51}{50}<2$;
then apply 
\eqref{MAINTHM4pf21pre1} for each $z\in[-2,2]\setminus\fS$,
and \eqref{MAINTHM4pf21pre2} for each $z\in\fS$.
Using also \eqref{MAINTHM4pf18},
and the fact that $\int_{-2}^2F_z(s)\,dz=1$ for all $s\in(-1,1)\setminus\{-\frac dc\}$
while $h(s)=0$ when $|s|\geq1$,
it follows that
\begin{align}\notag
&\frac1T\int_{\R}f\bigl(\Gamma g\mathrm{u}_t\bigr) h\Bigl(\frac tT\Bigr)\,dt
=\int_X f\,d\mu \int_{-2}^2\int_{\R} h(s)\,F_z(s)\,ds\,dz
\\\notag
&\hspace{140pt}
+O\biggl(\|f\|_{\C_{\alpha}^n}\|h\|_{\C_0^5}\biggl(\sqrt y+\int_{[-2,2]\setminus\fS}\delta_{m}\Bigl(\frac tT;\,\vecxi\gamma \mathrm{u}_{j}\Bigr)\,dz\biggr)\biggr)
\\\label{MAINTHM4pf21}
&=\int_X f\,d\mu \int_{\R}h(s)\,ds
+O\biggl(\|f\|_{\C_{\alpha}^n}\|h\|_{\C_0^5}\biggl(\sqrt y+\int_{[-2,2]\setminus\fS}\delta_{m}\Bigl(\frac tT;\,\vecxi\gamma \mathrm{u}_{j}\Bigr)\,dz\biggr)\biggr).
\end{align}

It remains to bound the integral over $z\in [-2,2]\setminus\fS$,
wherein it should be remembered that
$t=t(z)=(cz+d)^{-2}$
and $j=j(z):=\left\lfloor\frac{az+b}{cz+d}\right\rfloor$.
Using \eqref{MAINTHM4pf1pre}, the second inequality in \eqref{funddomDEF}
is equivalent with
$|ac+bd|\leq\frac12(c^2+d^2)$. %
Combining this with 
${\displaystyle b=\frac{d(ac+bd)-c}{c^2+d^2}}$,
${\displaystyle a=\frac{c(ac+bd)+d}{c^2+d^2}}$
and \eqref{MAINTHM4pf1},
it follows that
\begin{align}\label{MAINTHM4pf22}
\max(|a|,|b|)\leq\tfrac12 {Y}^{-\frac12}+{Y}^{\frac12}<2\sqrt{{Y}}.
\end{align}
Hence for all $z\in[-2,2]$ we have
$|az+b|<6\sqrt{{Y}}$
and so 
\begin{align}\label{MAINTHM4pf23}
|j(z)|\leq 1+\biggl|\frac{az+b}{cz+d}\biggr|
<1+6\sqrt{{Y}\, t(z)}
<7\sqrt{{Y}\, t(z)}.
\end{align}
(The last inequality holds since ${Y}>100$ and $|cz+d|<3{Y}^{-\frac12}$, viz., $t(z)>9^{-1}{Y}>1$.)
It follows from \eqref{MAINTHM4pf23} and \eqref{fSdef} that
for all $z\in[-2,2]\setminus\fS$ we have
$|j(z)|<\sqrt{{Y}T}=T\sqrt y$.
Using also the fact that $\delta_{m}(y;\vecxi)$ is an increasing function of $y$,
it follows that %
\begin{align}\label{MAINTHM4pf26}
\int_{[-2,2]\setminus\fS}\delta_{m}\biggl(\frac{t(z)}T;\,\vecxi \gamma \mathrm{u}_{j(z)}\biggr)\,dz
\leq\sum_{|\ell|<T\sqrt y}\bigl|A_{\ell}\bigr|\,
\delta_{m}\biggl(\frac{t_{\ell}}T;\:{\vecxi} \gamma \mathrm{u}_{\ell}\biggr),
\end{align}
where the sum runs over all integers $\ell$ satisfying $|\ell|<T\sqrt y$. Here,
\begin{align*}
A_{\ell}:=\bigl\{z\in[-2,2]\setminus\fS\col j(z)=\ell\bigr\},
\qquad \qquad
t_{\ell}:=\sup\{t(z)\col z\in A_{\ell}\},
\end{align*}
and $\bigl|A_{\ell}\bigr|$ denotes the Lebesgue measure of $A_{\ell}$.

For any $z\in[-2,2]$ such that
$t(z)\geq 100{Y}$
we have 
$|cz+d|\leq 10^{-1}{Y}^{-\frac12}$; thus via
\eqref{MAINTHM4pf1} and \eqref{MAINTHM4pf22},
\begin{align}
|az+b|=\biggl|\frac{a(cz+d)-1}c\biggr|
>\frac{1-2\cdot 10^{-1}}{{Y}^{-1/2}}=\tfrac45 \sqrt{{Y}},
\end{align}
and so 
\begin{align}\label{MAINTHM4pf24}
|j(z)|\geq\biggl|\frac{az+b}{cz+d}\biggr|-1>\tfrac45\sqrt{{Y}\, t(z)}-1
>\tfrac12\sqrt{{Y}\, t(z)}.
\end{align}

Note that   %
$\ell\leq\frac{az+b}{cz+d}<\ell+1$
for all $z\in A_{\ell}$.
The function $z\mapsto\frac{az+b}{cz+d}$ is 
injective on $\R\setminus\{-\frac dc\}$
and has derivative
$\frac d{dz}\bigl(\frac{az+b}{cz+d}\bigr)=t(z)$.
It follows from \eqref{MAINTHM4pf1} that
$t(z)\geq \frac16\, Y$ for all $z\in[-2,2]$;
furthermore, \eqref{MAINTHM4pf23} implies that
$t(z)>\frac1{49}\ell^2/{Y}$ for all $z\in A_{\ell}$.
Hence we have
\begin{align}\label{MAINTHM4pf28}
|A_{\ell}|\leq\min\biggl(\frac6{{Y}},\frac{49{Y}}{\ell^2}\biggr)
\ll\frac Y{(Y+|\ell|)^2},   %
\qquad\forall\ell.
\end{align}
Furthermore, it follows from the statement around \eqref{MAINTHM4pf24}
that for every integer $\ell$ with
$|\ell|\leq 5{Y}$ we have $t_{\ell}\leq 100{Y}$,
while if $|\ell|>5{Y}$ then
$t_{\ell}\leq 4\ell^2/{Y}$.
Hence
\begin{align}\label{MAINTHM4pf29}
t_{\ell}\ll \frac{(Y+|\ell|)^2}Y, %
\qquad\forall \ell.
\end{align}
Using %
\eqref{MAINTHM4pf28}, \eqref{MAINTHM4pf29}
and Lemma \ref{deltambasicfactLEM1},
and also recalling that ${Y}=Ty$,
it follows via \eqref{MAINTHM4pf26} that %
\begin{align*}
\int_{[-2,2]\setminus\fS}\delta_{m}\biggl(\frac{t(z)}T;\,\vecxi \gamma \mathrm{u}_{j(z)}\biggr)\,dz
\ll\sum_{|\ell|<T\sqrt y}
\frac{Ty}{(Ty+|\ell|)^2}\,\delta_m\Biggl(
\biggl(\frac{Ty+|\ell|}{T\sqrt y}\biggr)^2;\vecxi\gamma\mathrm{u}_{\ell}\Biggr).
\end{align*}
Using this bound in \eqref{MAINTHM4pf21},
we obtain \eqref{MAINTHM4pfPROPres} \footnote{Note that in \eqref{MAINTHM4pfPROPres}
we have ignored the summation condition $|\ell|<T\sqrt y$;
however, since $\delta_m(y';\vecxi)\asymp 1$ uniformly over all $y'\geq1$ and $\vecxi\in(\R^2)^k$,
the total contribution from the terms with $|\ell|\geq T\sqrt y$
is $\asymp\sqrt y$,
i.e.\ subsumed by the first term in the bound in \eqref{MAINTHM4pfPROPres}.},
and so Proposition \ref{MAINTHM4pfPROP} is proved.
\end{proof}

We will now further analyse the sum over $\ell$ in 
the right-hand side of \eqref{MAINTHM4pfPROPres}.
By \eqref{deltaNEWDEF}, 
this sum equals
\begin{align}\label{nonclosedgendisc60pre1}
\sum_{\vecq\in\Z^k\setminus\{\bn\}}\sum_{d=1}^{\infty}\frac{\tau(d)}{\|\vecq\|^{m}d^{3/2}}
\sum_{\ell\in\Z}\frac{Ty}{(Ty+|\ell|)^2}
\biggl(1+\frac{T\sqrt y}{(Ty+|\ell|)d}\big\| \bigl(v_1,\ell v_1+v_2\bigr)\big\|_{\Z}\biggr)^{\hspace{-3pt}-1},
\end{align}
where we have introduced 
the shorthand notation $(v_1,v_2):=d\vecq\vecxi\gamma$.   %
Let us also set $w_j:=\|v_j\|_{\Z}\in[0,\frac12]$ for $j=1,2$.
Then 
$\big\| \bigl(v_1,\ell v_1+v_2\bigr)\big\|_{\Z}\asymp w_1+\|s\ell w_1+w_2\|_{\Z}$ for all $\ell\in\Z$,
where $s=1$ if either both or none of $v_1$ and $v_2$ lie in 
$[0,\frac12]+\Z$, %
otherwise $s=-1$.
Hence the expression in \eqref{nonclosedgendisc60pre1} is   %
\begin{align}\label{nonclosedgendisc60pre2}
\asymp\sum_{\vecq\in\Z^k\setminus\{\bn\}}\sum_{d=1}^{\infty}\frac{\tau(d)}{\|\vecq\|^{m}d^{3/2}}
\sum_{\ell\in\Z}\frac{Ty}{(Ty+|\ell|)^2}
\biggl(1+\frac{T\sqrt y}{(Ty+|\ell|)d}\bigl(w_1+\|\ell w_1+w_2\|_{\Z}\bigr)\biggr)^{\hspace{-3pt}-1}.
\end{align}
Here we will use the following bound.
\begin{lem}\label{nonclosedKeyBoundLEM1new}
For any $w_1,w_2\in[0,\frac12]$, $\alpha\geq\frac1{10}$ and $\beta>0$,
we have
\begin{align}\label{nonclosedKeyBoundLEM1newres}
\sum_{{j}\in\Z}\frac{\alpha}{(\alpha+|{j}|)^2}
\biggl(1+\frac{\alpha\beta}{\alpha+|{j}|}\bigl(w_1+\|{j} w_1+w_2\|_{\Z}\bigr)\biggr)^{\hspace{-3pt}-1}
\hspace{60pt}
\\\notag
\ll \scrL_1\biggl(\frac1{1+\alpha\beta w_1+\beta w_2}\biggr)
+\scrL_2\biggl(\frac1{1+\beta }\biggr),
\end{align}
where the implied constant is absolute.
\end{lem}
\begin{proof}
If $\beta \leq10$ then the lemma holds trivially,
since then in the right-hand side
of \eqref{nonclosedKeyBoundLEM1newres} 
we have %
$\scrL_2((1+\beta )^{-1})\asymp1$,
while the left-hand side 
is bounded above by
$\sum_{{j}\in\Z}\frac{\alpha }{(\alpha +|{j}|)^2}\asymp 1$.
Hence from now on we may assume $\beta \geq10$.
We will also assume $w_1>0$ in the following;
the remaining case $w_1=0$ may be handled %
by an easy direct analysis, or %
at the end of the proof by
taking the limit $w_1\to0$ in the result proved for $w_1>0$.

The set of integers can be expressed as the disjoint union of the intervals
\begin{align*}
I_k:=\{j\in\Z\col jw_1+w_2\in k+(-\tfrac12,\tfrac12]\}
\qquad (k\in\Z).
\end{align*}
Hence the left-hand side of \eqref{nonclosedKeyBoundLEM1newres} equals
$\sum_{k\in\Z}S_k$, where
\begin{align*}
S_k:=\sum_{j\in I_k}\frac{\alpha }{(\alpha +|j|)^2}
\biggl(1+\frac{\alpha \beta }{\alpha +|{j}|}\bigl(w_1+\|{j} w_1+w_2\|_{\Z}\bigr)\biggr)^{\hspace{-3pt}-1}.
\end{align*}
Let us also set %
\begin{align}\label{nonclosedKeyBoundLEM1newpf4}
j_k:=\biggl\lfloor \frac{k-w_2}{w_1}+\frac12\biggr\rfloor.
\end{align}
This means that $j_k$ is the unique integer satisfying
$j_kw_1+w_2\in k+\bigl(-\tfrac12w_1,\tfrac12w_1\bigr]$.
Hence $j_k\in I_k$. It also follows that each $j\in I_k$ satisfies
$|j-j_k|<\frac1{2w_1}+\frac12<w_1^{-1}$
and 
\begin{align*}
w_1+\|{j} w_1+w_2\|_{\Z}
=w_1+\bigl|jw_1+w_2-k\bigr|
=w_1+\bigl|(j-j_k)w_1+j_kw_1+w_2-k\bigr|
\hspace{50pt}
\\
\asymp (1+|j-j_k|)w_1,
\end{align*}
where the last relation holds since $\bigl|j_kw_1+w_2-k\bigr|\leq\frac12w_1$.
Hence
\begin{align}\label{nonclosedKeyBoundLEM1newpf1new}
S_k\ll \sum_{j_k-w_1^{-1}<j<j_k+w_1^{-1}}
\frac{\alpha }{(\alpha +|j|)^2}\biggl(1+\frac{\alpha \beta }{\alpha +|{j}|}(1+|j-j_k|)w_1\biggr)^{\hspace{-3pt}-1}.
\end{align}
Setting $j=n+j_k$ and comparing the terms for $n$ and $-n$, we get
\begin{align}\label{nonclosedKeyBoundLEM1newpf1newA}
S_k 
&\ll \sum_{0\leq n<w_1^{-1}}\frac{\alpha }{(\alpha +|n-J_k|)\bigl(\alpha +|n-J_k|+\alpha \beta w_1(n+1)\bigr)},
\end{align}
where $J_k:=|j_k|$.

One way to bound the sum in 
\eqref{nonclosedKeyBoundLEM1newpf1newA},
which we will use for small $k$,
is as follows.
Extend the summation to \textit{all} $n\geq0$,
and then note that the term corresponding to any $n$ in the interval
$J_k<n\leq 2J_k$ is subsumed by the term corresponding to
$n'=2J_k-n$, since $|n-J_k|=|n'-J_k|$ while $n+1\geq n'+1$.
Hence we may restrict the summation to the three intervals
(i) $0\leq n<\frac12(J_k-1)$, (ii)
$\frac12(J_k-1)\leq n\leq J_k$ and (iii) $n>2J_k$.
(Note that (i) is void unless $J_k\geq2$.)
For $n$ belonging to the interval (i), %
we have $|n-J_k|\asymp J_k$ ($\geq2$) and we use $m=n+1$ as summation variable;
for $n$ belonging the interval (ii) we have
$n+1\asymp J_k+1$ and we use $m=J_k-n$ as summation variable.
In this way we obtain 
(we extend the summation ranges somewhat in cases (i) and (ii) to simplify the notation): %
\begin{align}\label{nonclosedKeyBoundLEM1newpf1newB}
S_k\ll\frac \alpha {(\alpha +J_k)^2}\sum_{m=1}^{J_k}\frac1{1+Am}
+\sum_{m=0}^{J_k}\frac \alpha {(\alpha +m)(B+m)}
+\sum_{n=2J_k+1}^\infty\frac C{(\alpha +n)(C+n)},
\end{align}
where $A=\frac{\alpha \beta w_1}{\alpha +J_k}$,
$B=\alpha +\alpha \beta w_1(J_k+1)$ and
$C=\frac \alpha {1+\alpha \beta w_1}$.
By a direct case-by-case analysis, this leads to
\begin{align}\label{nonclosedKeyBoundLEM1newpf1newC}
S_k\ll \frac{\alpha }{(\alpha +J_k)^2}
\left.\begin{cases}
J_k&\text{if }\: AJ_k\leq1
\\[3pt]
A^{-1}\log(2+AJ_k)&\text{if }\: AJ_k>1
\end{cases}\right\}
+
\left.\begin{cases}
{\displaystyle\frac{J_k+1}B+\scrL_1\Bigl(\frac C\alpha \Bigr)}&\text{if }\: J_k\leq \alpha 
\\[13pt]
{\displaystyle \scrL_1\Bigl(\frac \alpha B\Bigr)+\frac C{J_k}} &\text{if }\: J_k> \alpha 
\end{cases}\right\}.
\end{align}
We claim that this implies
\begin{align}\label{nonclosedKeyBoundLEM1newpf1newD}
S_k\ll\scrL_1\biggl(\frac1{1+(\alpha +J_k)\beta w_1}\biggr).
\end{align}
The fact that the right-hand side of \eqref{nonclosedKeyBoundLEM1newpf1newD}
bounds the first term in 
\eqref{nonclosedKeyBoundLEM1newpf1newC}  %
follows by noticing that
if $AJ_k\leq1$ then $\frac{\alpha +J_k}{\alpha J_k}\geq \beta w_1$
and hence $\frac{(\alpha +J_k)^2}{\alpha J_k}\gg 1+(\alpha +J_k)\beta w_1$,
while if $AJ_k>1$ then
$J_k\beta w_1\geq J_kA>1$ and thus also
$\alpha ^{-1}(\alpha +J_k)^2A=(\alpha +J_k)\beta w_1\gg 1+(\alpha +J_k)\beta w_1$.
The fact that the right-hand side of \eqref{nonclosedKeyBoundLEM1newpf1newD}
bounds the remaining terms in 
\eqref{nonclosedKeyBoundLEM1newpf1newC}  %
follows by noticing that
if $J_k\leq \alpha $ then $\frac B{J_k+1}=\frac \alpha {J_k+1}+\alpha \beta w_1
\gg 1+\alpha \beta w_1\gg 1+(\alpha +J_k)\beta w_1$
(here we used $\alpha \geq \frac1{10}$ to get $\frac \alpha {J_k+1}\gg1$)
and also $\alpha/C=1+\alpha \beta w_1\gg 1+(\alpha +J_k)\beta w_1$,
while if $J_k>\alpha $ then
$B/\alpha =1+(J_k+1)\beta w_1\gg 1+(\alpha +J_k)\beta w_1$
and $C/J_k=\frac{J_k}\alpha +J_k\beta w_1> 1+J_k\beta w_1\gg 1+(\alpha +J_k)\beta w_1$.
Hence \eqref{nonclosedKeyBoundLEM1newpf1newD} indeed holds.

It follows from \eqref{nonclosedKeyBoundLEM1newpf4} that
$\cdots<j_{-1}<j_0\leq 0<j_1<j_2<\cdots$
and $J_k\geq J_0$ for all $k$,
and that $J_0+1\asymp\frac{w_2}{w_1}+1$.
Hence, using also the fact that $\alpha \geq \frac1{10}$,
we get from \eqref{nonclosedKeyBoundLEM1newpf1newD}:
\begin{align}\label{nonclosedKeyBoundLEMpf20}
\sum_{|k|\leq5}S_k\ll\scrL_1\biggl(\frac1{1+(\alpha +J_0)\beta w_1}\biggr)
\ll\scrL_1\biggl(\frac1{1+\alpha \beta w_1+\beta w_2}\biggr).
\end{align}

\vspace{5pt}

Next, when $|k|>5$,
we have %
$J_k\geq\frac{|k|-1}{w_1}\geq\frac5{w_1}$,
so that \eqref{nonclosedKeyBoundLEM1newpf1newA} implies
\begin{align}\label{nonclosedKeyBoundLEMpf21pre}
S_k\ll \frac \alpha {(\alpha +J_k)^2}\sum_{0\leq n<w_1^{-1}}\frac1{1+A(n+1)},
\qquad\text{where }\: A:=\frac{\alpha \beta w_1}{\alpha +J_k}.
\end{align}
Considering the three cases $A\leq w_1$, $w_1<A\leq1$ and $A>1$ separately,
this gives
\begin{align}\label{nonclosedKeyBoundLEMpf21}
S_k\ll \frac \alpha {(\alpha +J_k)^2}\cdot\frac{\log(2+Aw_1^{-1})}{A+w_1}
\ll\frac{\alpha w_1\log(2+\beta )}{(\alpha w_1+|k|)(\alpha \beta w_1+|k|)}.
\end{align}
Hence
\begin{align*}
\sum_{|k|>5}S_k
\ll \alpha w_1\log(2+\beta )\sum_{k=6}^\infty\frac{1}{(\alpha w_1+k)(\alpha \beta w_1+k)}.
\end{align*}
Here $0<\alpha w_1<\alpha \beta w_1$, and we find that
\begin{align}\label{nonclosedKeyBoundLEMpf23}
\sum_{k=6}^\infty\frac{1}{(\alpha w_1+k)(\alpha \beta w_1+k)}
\ll\begin{cases}
1&\text{if }\: \alpha \beta w_1\leq1
\\
(\alpha \beta w_1)^{-1}\log(2+\beta )&\text{if }\: \alpha \beta w_1>1.
\end{cases}
\end{align}
Therefore,
\begin{align}\label{nonclosedKeyBoundLEMpf22}
\sum_{|k|>5}S_k\ll\scrL_2(\beta ^{-1}).
\end{align}
Adding the two bounds \eqref{nonclosedKeyBoundLEMpf20}
and \eqref{nonclosedKeyBoundLEMpf22}
(and recalling that we are currently assuming $\beta \geq10$),
we obtain the bound in \eqref{nonclosedKeyBoundLEM1newres},
i.e.\ Lemma \ref{nonclosedKeyBoundLEM1new} is proved.
\end{proof}
Applying Lemma \ref{nonclosedKeyBoundLEM1new}
with $\alpha=Ty$ and $\beta=\frac1{d\sqrt y}$
(this is permitted since $\alpha=Ty=\scrY(M  \mathrm{a}_T )\geq\frac12\sqrt3>\frac1{10}$),
we obtain that the expression in \eqref{nonclosedgendisc60pre2} is
\begin{align}\notag
\ll\sum_{\vecq\in\Z^k\setminus\{\bn\}} & \sum_{d=1}^{\infty}\frac{\tau(d)}{\|\vecq\|^{m}d^{3/2}}
\biggl(\scrL_1\biggl(\biggl(1+\frac{T\sqrt y}d w_1+\frac{w_2}{d\sqrt y}\biggr)^{-1}\biggr)
+\scrL_2\biggl(\frac{d\sqrt y}{d\sqrt y+1}  %
\biggr)\biggr)
\\\label{nonclosedgendisc60}
&\ll \scrL_3(y^{\frac14})+\sum_{\vecq\in\Z^k\setminus\{\bn\}}\sum_{d=1}^{\infty}\frac{\tau(d)}{\|\vecq\|^{m}d^{3/2}}
\scrL_1\Bigl(\frac1{1+d^{-1}S}\Bigr),
\qquad\text{where }\: S:=T\sqrt y w_1+\frac{w_2}{\sqrt y}.
\end{align}
For the last estimate, we used the bounds
$\sum_{d\leq y^{-1/2}}\frac{\tau(d)}{d^{3/2}}\scrL_2(d\sqrt y)\ll\scrL_3(y^{\frac14})$
and $\sum_{d>y^{-1/2}}\frac{\tau(d)}{d^{3/2}}\ll\scrL_1(y^{\frac14})$,
which follow from Lemma \ref{basicboundLEM1}.

\vspace{5pt}

In order to complete the proof of Theorem \ref{MAINTHM4}, now
it only remains to bound $S$ %
in terms of the quantity $S_{g,d\vecq}(T)$ defined in \eqref{bMxiqTDEF}.
To this end, let us write 
$\smatr {a'}{b'}{c'}{d'} %
:=\gamma^{-1} M \mathrm{a}_T$ \footnote{This is as in\label{apbpcpdpfootnote}
\eqref{MAINTHM4pf1pre} except that we use the
variable names $a',b',c',d'$ to avoid a notation clash with our summation variable $d$;
note also that we are currently \textit{not} assuming
$\scrY(M\mathrm{a}_T )>100$; 
however we do have $\scrY(M\mathrm{a}_T )\geq\frac12\sqrt 3$; see \eqref{scrYlowerbd}.},
and consider the grid
\begin{align}\label{MAINTHM4pf31}
\Z^2\,\p_{d \vecq}(g)\mathrm{a}_T
=\bigl(\Z^2+d\vecq{\vecxi}\bigr)M\mathrm{a}_T
=\bigl(\Z^2+d\vecq{\vecxi}\bigr)\gamma  \matr {a'}{b'}{c'}{d'} 
=\bigl(\Z^2+(d\vecq\trans{\vecxi}_{\gamma,1},d\vecq\trans{\vecxi}_{\gamma,2})\bigr)\matr {a'}{b'}{c'}{d'}.
\end{align}
In the last expression we introduced the notation 
$(\trans{\vecxi}_{\gamma,1},\trans{\vecxi}_{\gamma,2}):=\vecxi\gamma$.
Now %
$w_j=\|v_j\|_{\Z}=\|d\vecq\trans{\vecxi}_{\gamma,j}\|_{\Z}$,
and hence there exists a choice of signs
$\ve_1,\ve_2\in\{1,-1\}$ such that
the point $(\ve_1w_1,\ve_2w_2)$ belongs to the grid
$\Z^2+(d\vecq\trans{\vecxi}_{\gamma,1},d\vecq\trans{\vecxi}_{\gamma,2})$;
therefore the grid in \eqref{MAINTHM4pf31} %
contains the point
\begin{align}\label{MAINTHM4pf31a}
(\ve_1w_1,\ve_2w_2) \matr {a'}{b'}{c'}{d'}
=\bigl(\ve_1w_1a'+\ve_2w_2c',\ve_1w_1b'+\ve_2w_2d'\bigr).
\end{align}
As in \eqref{MAINTHM4pf1} we have
${c'}^2+{d'}^2=\scrY(M\mathrm{a}_T)^{-1}=(Ty)^{-1}$;
thus $|c'|,|d'|\leq (Ty)^{-\frac12}$.
Also the argument giving
\eqref{MAINTHM4pf22}
still applies (see footnote \ref{apbpcpdpfootnote})
to give
$|a'|,|b'|<2\,\scrY(M\mathrm{a}_T)^{\frac12}=2(Ty)^{\frac12}$.
Recall also that $S:=T\sqrt y w_1+y^{-1/2}w_2$;
therefore
$w_1\leq (T\sqrt y)^{-1}S$ and $w_2\leq \sqrt yS$.
It follows %
the point in \eqref{MAINTHM4pf31a} lies in the square
$3ST^{-\frac12} [-1,1]^2$.
Applying $\mathrm{a}_T^{-1}$ to this point,
we conclude that the grid
$\Z^2\,\p_{d \vecq}(g)$
contains a point in the rectangle
$3 ST^{-\frac12} [-1,1]^2\mathrm{a}_T^{-1}
=3 S\, \fR_T$.
By \eqref{bMxiqTDEF}, this implies that
\begin{align}
S_{g,d\vecq}(T)\leq 3\,S.
\end{align}
It follows that the expression in \eqref{nonclosedgendisc60} is
\begin{align}\label{nonclosedgendisc63}
&\ll
\scrL_3(y^{\frac14})
+\sum_{\vecq \in\Z^k\setminus\{\bn\}}\sum_{d=1}^{\infty}\frac{\tau(d)}{\|\vecq\|^md^{3/2}}
\scrL_1\biggl(\frac1{1+d^{-1}S_{g,d\vecq}(T)}\biggr).
\end{align}
To sum up, we have proved that the expression in \eqref{nonclosedgendisc63}
is an upper bound for the sum over $\ell$ in the right-hand side of \eqref{MAINTHM4pfPROPres}.
Hence, recalling also \eqref{ygTSg0TremarkRES2},
we obtain the bound in \eqref{MAINTHM4res},
i.e.\ Theorem \ref{MAINTHM4} is proved.
\hfill$\square$

\subsection{Properties of the bound in Theorem \ref*{MAINTHM4}}

Here we wish to show that for generic initial points, %
the right-hand side of \eqref{MAINTHM4res} decays like $T^{-\frac14+\ve}$. %
The two lemmas below provide more precise statements.

The following bound is a direct consequence of 
Sullivan's logarithm law for geodesics
\cite[Sec.\ 9]{Sullivan}.
Note that $S_{g,\bn}(T)$ only depends on the $\SL(2,\R)$-component of $g\in\G$.
\begin{lem}\label{SullivanloglawapplLEM}
For any given $\ve>0$,
we have for Haar-almost every $M\in\SL(2,\R)$:
\begin{align}\label{SullivanloglawapplLEMres}
\scrL_3\Bigl(S_{M,\bn}(T)^{-\frac12}\Bigr)\ll T^{-\frac14}(\log T)^{\frac{13}4+\ve}
\qquad\text{as }\: T\to\infty.
\end{align}
\end{lem}
\begin{proof}
By \cite[Theorem 6]{Sullivan}
and \cite[eq.\ (14)]{iha},
we have for Haar-almost every $M\in\SL(2,\R)$:
\begin{align*}
\limsup_{T\to\infty}\frac{\log\scrY(M\mathrm{a}_T )}{\log\log T}=1.
\end{align*}
This implies (via \eqref{ygTdef})
that $y_M(T)<T^{-1}(\log T)^{1+\ve}$ %
for all sufficiently large $T$,
and so, by \eqref{ygTSg0TremarkRES2},
we obtain \eqref{SullivanloglawapplLEMres} 
(with $\ve/4$ in place of $\ve$).
\end{proof}
\begin{lem}\label{THM4boundgenericdecayLEM}
Let $m>k\geq1$.
Write $g=(1_2,{\vecxi})M$.
Then for any fixed $M\in\SL(2,\R)$
and any given $\ve>0$,
we have for Lebesgue almost all ${\vecxi}\in(\R^2)^k$:
\begin{align}\label{THM4boundgenericdecayLEMres}
\sum_{\vecq\in\Z^k\setminus\{\bn\}}\sum_{d=1}^{\infty}\frac{\tau(d)}{\|\vecq\|^{m}d^{3/2}}
\, \scrL_1\biggl(\frac1{1+d^{-1}S_{g,d\vecq}(T)}\biggr)
\ll T^{-\frac14}(\log T)^{3+\ve}
\end{align}
as $T\to\infty.$
\end{lem}

\begin{proof}
The strategy of the proof is the same as for Lemma \ref{labaaxiLEM}.
Let us write $C$ for the unit cube $C=((0,1)^2)^k$,
and consider the integral
\begin{align}\label{THM4boundgenericdecayLEMpf1}
\int_C &\sum_{\vecq\in\Z^k\setminus\{\bn\}}\sum_{d=1}^{\infty}\frac{\tau(d)}{\|\vecq\|^{m}d^{3/2}}
\, \scrL_1\biggl(\frac1{1+d^{-1}S_{g,d\vecq}(T)}\biggr)
\,d{\vecxi}.
\end{align}
Note that the grid $\Z^2\p_{d\vecq}(g)$ appearing inside the definition of
$S_{g,d\vecq}(T)$
(see \eqref{bMxiqTDEF}) equals
\begin{align*}
\Z^2\p_{d\vecq}(g)
=\Z^2\p_{d\vecq}((1_2,\vecxi)M)
=(\Z^2+d\vecq\vecxi)M.
\end{align*}
Recall also that $\fR_T=T^{-\frac12} [-1,1]^2\mathrm{a}_T^{-1}$;
therefore
\begin{align*}
S_{g,d\vecq}(T) &=\sup\bigl\{S\geq0\col S\,\fR_T\cap (\Z^2+d\vecq\vecxi)M=\emptyset\bigr\}
=T^{\frac12}\, r\bigl((\Z^2+d\vecq\vecxi)M\mathrm{a}_T\bigr),
\end{align*}
where $r(L')$ for any grid $L'\subset\R^2$ is defined by
\begin{align*}
r(L'):=\sup\bigl\{r\geq0 \col [-r,r]^2\cap L'=\emptyset\bigr\}.
\end{align*}
Let us also set
\begin{align*}
L_{M,T}:=\Z^2M\mathrm{a}_T.
\end{align*}
Now change order of summation and integration in 
\eqref{THM4boundgenericdecayLEMpf1},
and note that for any 
fixed $\vecq$ and $d$ appearing in the sums, %
the pushforward of Lebesgue measure on $C$
by the map $\vecxi\mapsto d\vecq\vecxi\mod\Z^2$ 
equals Lebesgue measure on the torus $\R^2/\Z^2$,
which by the map 
$\vecw\mapsto \vecw M\mathrm{a}_T$ is
transformed to Lebesgue measure on the torus $\R^2/L_{M,T}$.
Hence it follows that the integral in 
\eqref{THM4boundgenericdecayLEMpf1} can be rewritten as
\begin{align}\label{THM4boundgenericdecayLEMpf2}
&\sum_{\vecq\in\Z^k\setminus\{\bn\}}\sum_{d=1}^{\infty}\frac{\tau(d)}{\|\vecq\|^{m}d^{3/2}}
J(d^{-1}T^{\frac12};L_{M,T})
\quad
\text{with }\:
J(a;L):=\int_{\R^2/L}\scrL_1\biggl(\frac1{1+a\cdot r(L+\vecv)}\biggr)\,d\vecv.
\end{align}

Next we claim that
\begin{align}\label{THM4boundgenericdecayLEMpf3}
J(a;L)\ll\scrL_1\Bigl(\frac1{1+a}\Bigr)
\end{align}
uniformly over all $a>0$ and all lattices $L<\R^2$ of  co-area one.
For $0<a\leq10$, \eqref{THM4boundgenericdecayLEMpf3} is immediate from the fact that %
$J(a;L)\ll1$; hence from now on 
we may assume $a\geq10$.
By dyadic decomposition,
\begin{align}\notag
J(a;L) &=\sum_{k\in\Z}\int_{\{\vecv\in\R^2/L\col 2^{-k-1}< r(L+\vecv)\leq 2^{-k}\}}
\scrL_1\biggl(\frac1{1+a\cdot r(L+\vecv)}\biggr)\,d\vecv
\\\label{THM4boundgenericdecayLEMpf4}
&\ll\sum_{k\in\Z} \bigl|\bigl\{\vecv\in\R^2/L\col r(L+\vecv)\leq 2^{-k}\bigr\}\bigr| 
\cdot\scrL_1\biggl(\frac1{1+2^{-k}a}\biggr),
\end{align}
where $|\cdot|$ denotes Lebesgue measure on $\R^2/L$.
Now note that $r(L+\vecv)\leq 2^k$ 
holds if and only if
$[-2^{-k},2^{-k}]^2\cap (L+\vecv)\neq\emptyset$,
and the area of the image of the square
$[-2^{-k},2^{-k}]$ in $\R^2/L$ is bounded above both by
the area of the original square, and by $|\R^2/L|=1$.
Hence the last sum is
\begin{align*}
&\leq\sum_{k\in\Z} \min\bigl(4\cdot 2^{-2k},1\bigr)\cdot \scrL_1\biggl(\frac1{1+2^{-k}a}\biggr)
\\
&\ll \sum_{k\leq0}\frac{\log\bigl(2^{-k}a\bigr)}{2^{-k}a}
+\sum_{0<k\leq\log_2a}2^{-2k}\frac{\log a}{2^{-k}a}
+\sum_{k>\log_2a}2^{-2k}
\ll\frac{\log a}a\asymp\scrL_1\Bigl(\frac1a\Bigr)\asymp\scrL_1\Bigl(\frac1{1+a}\Bigr).
\end{align*}
This completes the proof of \eqref{THM4boundgenericdecayLEMpf3}.

Using the formula in \eqref{THM4boundgenericdecayLEMpf2}
and the bound \eqref{THM4boundgenericdecayLEMpf3},
it follows that for all $T\geq1$, 
the integral in \eqref{THM4boundgenericdecayLEMpf1} is
\begin{align}\label{THM4boundgenericdecayLEMpf5}
\ll \frac{\log\bigl(1+\sqrt T\bigr)}{\sqrt T}\sum_{1\leq d\leq\sqrt T}\frac{\tau(d)}{\sqrt d}
+\sum_{d>\sqrt T}\frac{\tau(d)}{d^{3/2}}
\ll T^{-\frac14}\bigl(\log(1+T)\bigr)^2.
\end{align}
See Lemma \ref{basicboundLEM1} for the last inequality.

Now the same type of argument as in the end of the proof of
Lemma \ref{labaaxiLEM}
applies to prove that \eqref{THM4boundgenericdecayLEMres} 
holds for Lebesgue almost all $\vecxi\in C$.
Hence \eqref{THM4boundgenericdecayLEMres} 
in fact holds for Lebesgue almost all $\vecxi\in (\R^2)^k$,
since the sum in the left-hand side of \eqref{THM4boundgenericdecayLEMres}
is invariant under $\vecxi\mapsto\vecxi+\veca$ for any $\veca\in(\Z^2)^k$.
\end{proof}

\end{document}